\documentclass[11pt,a4paper]{amsart}

\usepackage[T1]{fontenc}
\usepackage{amsmath, amssymb, graphicx, amsthm, amsfonts}
\usepackage{a4wide}
\usepackage{graphpap}
\usepackage{hyperref}
\usepackage{rotating}

\title{Geodesic flow, left-handedness, and templates}

\author{Pierre Dehornoy}
\address{Institut Fourier, 100 rue des maths, BP 74, 38402 Saint Martin d'H\`eres cedex, France}
\email{pierre.dehornoy@ujf-grenoble.fr}
\urladdr{http://www-fourier.ujf-grenoble.fr/~dehornop/}
\date{First version: December 29, 2011, Last revision: August 4, 2014}

\newcommand{\bord}{\partial}
\newcommand{\barre}{\overline}
\newcommand{\BoundTorus}{\barre{\un\Side}}
\newcommand{\BranchPQ}{\mathrm{BS}_{P\to Q}}
\newcommand{\BranchQP}{\mathrm{BS}_{Q\to P}}

\newcommand{\BS}{\mathrm{BS}}

\newcommand{\cc}[1]{\overline{#1}}
\newcommand{\cci}{\overline{\imath}}
\newcommand{\Cone}{C}

\newcommand{\D}{\mathbb D}
\newcommand{\discr}[2]{#1_{#2}}
\newcommand{\deform}[1]{F_{#1}}

\newcommand\GF[1]{\Phi_{#1}}
\newcommand{\Gr}{\mathcal G}

\newcommand{\Hy}{\mathbb H}

\newcommand{\Isom}{\mathrm{Isom^+}}
\newcommand{\isot}[2]{f_{#2,#1}}

\newcommand{\lk}{\mathrm{Lk}}

\newcommand{\hzero}[2]{h_0(#1,#2)}
\newcommand{\hun}[2]{h_1(#1,#2)}

\newcommand{\Orb}{\Sigma}
\newcommand{\Orbpq}{\Orb_{p, q, \infty}}
\newcommand{\Orbdq}{\Orb_{2, q, \infty}}
\newcommand{\Orbg}{\Orb_{2,3,4g+2}}

\newcommand{\PSLZ}{\mathrm{PSL}_2(\Z)}
\newcommand{\Pol}[1]{\mathrm{Pol}_{#1}}
\newcommand{\PPol}[1]{\mathrm{Pol_{#1}^\circ}}
\newcommand{\prot}{\widetilde{\mathcal B}}
\newcommand{\Polg}{P_{4g+2}}
\newcommand{\Polgint}{\hat P_{4g+2}}

\newcommand{\Q}{\mathbb Q}

\newcommand{\R}{\mathbb R}
\newcommand{\Rib}{\mathrm{Rib}}
\newcommand{\resp}{{\it resp.\ }}

\newcommand{\Sph}{\mathbb S}

\newcommand{\Side}{e}
\newcommand{\Sp}{S^\gamma_2}
\newcommand{\Sq}{S^\gamma_Q}
\newcommand{\Sb}{S^\gamma_e}
\newcommand{\Spi}{S^\gamma_\pi}
\newcommand{\Sd}{S^\gamma_\bord}
\newcommand{\Ss}{S^\gamma_{V_0}}
\newcommand{\Sss}{S^\gamma_{V_1}}

\newcommand{\T}{\mathbb T}
\newcommand{\Tiling}{\mathcal T}
\newcommand{\Tile}{T}
\newcommand{\temp}{\mathcal B}
\newcommand{\tinf}{a_Z}
\newcommand{\Tpq}{\temp_{\Gamma_{p,q}, \Tiling_{P,Q}}}
\newcommand{\Tdq}{\temp_{\Gamma_{2,q}, \Tiling_{P,Q}}}

\newcommand{\un}{T^1}
\newcommand{\uQ}{\barre{\un\Delta_Q/\Gamma_Q}}
\newcommand{\uP}{\barre{\un\Delta_P/\Gamma_P}}
  
\newcommand{\Vis}[2]{\mathrm{Vis}_{e_{#1}}^{e_{#2}}}
\newcommand{\vzero}[2]{v_0(#1,#2)}
\newcommand{\vun}[2]{v_1(#1,#2)}

\newcommand{\WT}[1]{\Theta_\mathrm{wheel}(#1)}

\newcommand{\Z}{\mathbb Z}

\numberwithin{equation}{section}

\theoremstyle{plain}

\newtheorem*{theoA}{Theorem A}
\newtheorem*{theoB}{Theorem B}

\theoremstyle{definition}

\newtheorem{defi}[equation]{Definition}

\newtheorem{ques}[equation]{Question}
\newtheorem{conj}[equation]{Conjecture}

\theoremstyle{plain}

\newtheorem{prop}[equation]{Proposition}

\newtheorem{lemma}[equation]{Lemma}
\newtheorem{theo}[equation]{Theorem}
\newtheorem{coro}[equation]{Corollary}

\theoremstyle{remark}

\newtheorem{remark}[equation]{Remark}

\normalsize 
\begin{document}

\begin{abstract}
We establish that, for every hyperbolic orbifold of type~$(2,q,\infty)$ and for every orbifold of type~$(2,3,4g{+}2)$,  the geodesic flow on the unit tangent bundle is left-handed. This implies that the link formed by every collection of periodic orbits $(i)$ bounds a Birkhoff section for the geodesic flow, and $(ii)$ is a fibered link. 
We also prove similar results for the torus with any flat metric. 
Besides, we observe that the natural extension of the conjecture to arbitrary hyperbolic surfaces (with non-trivial homology) is false.
\end{abstract}

\maketitle

\section{Introduction}
\label{S:Introduction}

In this paper, we investigate the dynamical properties of certain particular $3$-dimensional flows, namely the geodesic flows attached to surfaces and $2$-dimensional orbifolds. If $\Sigma$ is a Riemannian surface or, more generally, a Riemannian $2$-dimensional orbifold, that is, a space locally modelled on quotients of surfaces under the action of discrete rotation groups, the unit tangent bundle~$\un\Sigma$ is a $3$-manifold, and the geodesics of~$\Sigma$ induce a natural complete flow in~$\un\Sigma$. This flow is called the \emph{geodesic flow} of~$\un\Sigma$, hereafter denoted by~$\GF\Sigma$. What we do here is to specifically study the way the periodic orbits of~$\GF\Sigma$ may wrap one around the other.

In every 3-dimensional manifold~$M$, the linking number of two disjoint links can be defined in a non-ambiguous way whenever the links are null-homologous, that is, have a trivial image in~$H_1(M; \Q)$~\cite{Kaiser}. When the latter group is trivial, that is, when $M$ is a rational homology sphere, the linking number is always defined, and it yields a topological invariants of links. 

If $\Sigma$ is a 2-dimensional orbifold, every geodesic on~$\Sigma$ can be lifted to~$\un\Sigma$ in two ways, yielding a pair of orbits of~$\GF{\Sigma}$. It follows from Birkhoff's results~\cite{Birkhoff} that the linking number of any two such pairs of orbits is the opposite of the number of intersections of the geodesics, hence is nonpositive. This implies that, in a geodesic flow, there are always many pairs of orbits with a negative linking number. By contrast, there is no simple construction necessarily leading to collections of orbits with a positive linking number, and it makes sense to raise

\begin{ques}
\label{Q:OneHandedKnot}
Assume that~$\Sigma$ is a Riemannian 2-dimensional orbifold. Let~$\gamma, \gamma'$ be two null-homologous collections of periodic orbits of~$\GF{\Sigma}$. Does $\lk(\gamma, \gamma') <0$ necessarily hold?
\end{ques}

There are two cases when the answer to Question~\ref{Q:OneHandedKnot} is known to be positive, namely when $\Sigma$ is a sphere~$\Sph^2$ with a round metric and when $\Sigma$ is the modular surface~$\Hy^2/\PSLZ$~\cite{GhysJapan}. In the latter article, \'Etienne Ghys actually proves stronger results involving the natural extension of the linking number to arbitrary measures. Namely, he defines a complete flow~$\Phi$ in a homology $3$-sphere~$M$ to be \emph{left-handed} if the linking number of every pair of $\Phi$-invariant measures is always negative, and proves that the above two flows are left-handed. It is then natural to raise 



\begin{ques}[Ghys]
\label{Q:OneHanded}
Assume that $\Sigma$ is a Riemannian 2-dimensional orbifold satisfying $H_1(\un\Sigma, \Q) = 0$. Is the geodesic flow~$\GF{\Sigma}$ on~$\un\Sigma$ necessarily left-handed?
\end{ques}

By definition, a positive answer to Question~\ref{Q:OneHanded} implies a positive answer to Question~\ref{Q:OneHandedKnot}. As we shall explain, the converse implication, that is, the fact that the negativity of the linking number for pairs of periodic orbits implies the negativity of the linking number for arbitrary invariant measures, is true whenever the flow has sufficiently many periodic orbits, in particular when the flow is of Anosov type. 

The aim of this paper is to provide positive answers to Questions~\ref{Q:OneHandedKnot} and~\ref{Q:OneHanded} in new cases, namely when $\Sigma$ is a hyperbolic orbifold of type~$(2, q, \infty)$ with $q \ge 3$ and when $\Sigma$ is a hyperbolic orbifold of type~$(2, 3, 4g+2)$ with $g\ge 2$. 

\begin{theoA}
\label{T:OneHanded}
Assume that $\Sigma$ is $(a)$ either an orbifold of type~$(2, q, \infty)$ with $q \ge 3$, equipped with a negatively curved metric, or $(b)$ an orbifold of type~$(2, 3, 4g+2)$ with $g\ge2$,  equipped with a negatively curved metric. Then 

\noindent $(i)$ any two null-homologous collections of periodic orbits of~$\GF{\Sigma}$ have a negative linking number,

\noindent$(ii)$ the geodesic flow of~$\un\Sigma$ is left-handed.
\end{theoA}

In the case of a good orbifold with zero curvature, that is, a quotient of a torus with a flat metric, the unit tangent bundle always has non-trivial homology. Nevertheless it makes sense to address Question~\ref{Q:OneHandedKnot}. In this case as well, the answer is (almost) always positive.

\begin{theoB}
\label{T:TorusSimple}
Assume that~$\Sigma$ is a quotient of the torus~$\T^2$ equipped with a flat metric. Then any two collections~$\gamma, \gamma'$ of orbits of~$\GF{\Sigma}$ whose projections on~$\Sigma$ intersect have a negative linking number.
\end{theoB}

On the other hand, we give two examples showing that, when~$\Sigma$ is not a homology sphere or its curvature has a non-constant sign, Question~\ref{Q:OneHandedKnot} has a negative answer.

\begin{prop}
\label{P:CEx}
$(i)$ Let~$\Sigma$ be a hyperbolic surface. Then there exist two null-homologous collections~$\gamma, \gamma'$ of periodic orbits of~$\GF{\Sigma}$ with~$\lk(\gamma, \gamma') > 0$.

\noindent $(ii)$ Let~$\Sigma$ be a sphere with two non-intersecting simple geodesics. Then there exist two null-homologous collections~$\gamma, \gamma'$ of periodic orbits of~$\GF{\Sigma}$ with~$\lk(\gamma, \gamma') > 0$. The geodesic flow~$\GF{\Sigma}$ is not left-handed.
\end{prop}


When Questions~\ref{Q:OneHandedKnot} and~\ref{Q:OneHanded} have positive answers, an important consequence is the existence of many Birkhoff sections. A Birkhoff section for a non-singular flow on a $3$-manifold is a compact surface whose boundary is the union of finitely many periodic orbits of the flow, whose interior is transverse to the flow and intersects every orbit infinitely many times. The existence of a Birkhoff section for a flow is very useful as, in this case, studying the dynamics of the flow essentially reduces to studying the first return map on the section. Therefore, it is natural to wonder whether a flow admits Birkhoff sections. Now, as explained by Ghys \cite{GhysJapan}, the left-handedness of a flow implies the existence, for every finite collection of periodic orbits, of a Birkhoff section bounded by this collection. Thus our current results imply

\begin{coro}
\label{C:Birkhoff}
If $\Sigma$ is one of the orbifolds mentioned in Theorem~A, every finite null-homologous collection of periodic orbits of~$\GF{\Sigma}$ bounds a Birkhoff section. 
\end{coro}

Next, it is known~\cite{These} that every link that is the boundary of a Birkhoff section for a flow is fibered. Therefore, a direct consequence of Corollary~\ref{C:Birkhoff} is

\begin{coro}
\label{C:Fibered}
If $\Sigma$ is one of the orbifolds mentioned in Theorem~A, every link in~$\un\Sigma$ formed by a null-homologous collection of periodic orbits of the flow~$\GF{\Sigma}$ is fibered. 
\end{coro}

Similar statements hold in the case of the flat torus (see Theorem~\ref{T:Torus}), with, in addition, an explicit simple formula for the genus of the involved Birkhoff sections. 

Let us give a few hints about proofs. The case of the torus~$\T^2$ is the most simple one. It can be solved by elementary means, and it appears as a sort of warm-up. The key point is to encode every null-homologous collection~$\gamma$ of periodic orbits of~$\GF{\T^2}$ into some convex polygon~$\Pol{\gamma}$ in the affine plane~$\R^2$ with integral vertices. Using~$\Pol\gamma$ and VanHorn-Morris' helix boxes~\cite{JVHM}, we classify Birkhoff sections up to isotopy and derive their existence and the explicit formulas for the genus and  the linking number of two null-homologous collections of periodic orbits (Theorem~\ref{T:Torus}). Once these formulas are available, the negativity of the linking numbers easily follows (Corollary~\ref{C:Torus}). 

For Theorem~A, the proofs rely on a common principle but require specific ingredients depending on the orbifold. Our strategy decomposes in two steps. We first develop a general method for investigating the geodesic flow on a hyperbolic orbifold. A \emph{multitemplate} is a geometric 2-dimensional branched surface carrying a flow. This notion generalises Birman-Williams' notion of template~\cite{BW}, that have been introduced for studying hyperbolic flows. Here we prove that, given an orbifold~$\Sigma$, for every tessellation~$\Tiling$ of the hyperbolic plane adapted to~$\Sigma$, there exists a multitemplate~$\temp_\Tiling$ embedded in~$\un\Sigma$ such that the set of periodic orbits of~$\GF{\Sigma}$ is isotopic to a subset of the periodic orbits of~$\temp_\Tiling$ (Proposition~\ref{P:Template}). 
Moreover, if the orbifold~$\Sigma$ has at least one cusp, we can choose the tessellation~$\Tiling$ so that the set of periodic orbits of the geodesic flow is isotopic to the whole set of  periodic orbits of the template~$\temp_\Tiling$. 
This result provides a combinatorial description of the isotopy classes of the periodic orbits of~$\GF\Sigma$ in terms of some finite data specifying the orbifold. 
Note that the construction of the multi-template follows the strategy proposed by Birman and Williams for hyperbolic flows~\cite{BW}. 
In one sentence: we choose a Markov partition for the flow and contract the stable direction.

To complete the proof in the case when $\Sigma$ is an orbifold of type~$(2,q,\infty)$ with~$q\ge 3$, we start from the fact that $\un\Sigma$ is diffeomorphic to the complement of a certain knot~$K_\infty$ in some lens  space, and we choose a particular compactification. Then, choosing an adapted tessellation of the hyperbolic plane and using the template provided by Proposition~\ref{P:Template}, we estimate the linking number of an arbitrary pair of 
collections of periodic orbits and see that it is always negative. Along the way, we also compute the linking number of a geodesic with the knot~$K_\infty$ (Proposition~\ref{P:LkInfty}), a function of interest in number theory.

To complete the proof in the case of the orbifolds $\Sigma_{2,3,4g+2}$, the most delicate case, we use a covering of~$\Sigma_{2,3,4g+2}$ by some explicit genus~$g$ surface~$\Sigma_g$. Then, we use the template of Proposition~\ref{P:Template} to bound the linking number of two collections of periodic orbits of~$\GF{\Sigma_g}$ in terms of some associated combinatorial data. More precisely, we start from a tessellation of~$\Hy^2$ by $4g+2$-gons. For every periodic geodesic~$\gamma$ in~$\Sigma_g$ and for every pair of edges~$(e_i, e_j)$ in a tile of the tessellation, we denote by $b_{i,j}(\gamma)$ the number of times the projection of~$\gamma$ goes from~$e_i$ to $e_j$. Then, for every pair of geodesics $\gamma, \gamma'$, we show that the linking number~$\lk(\gamma, \gamma')$ is bounded above by a certain bilinear form~$S_{4g+2}$ involving the coefficients~$b_{i,j}(\gamma)$ and $b_{i,j}(\gamma')$. The form~$S_{4g+2}$ is not negative on the whole cone of vectors with positive coordinates (a manifestation of Proposition~\ref{P:CEx}). What we do here is to show that the form~$S_{4g+2}$ is negative on the subcone of vectors that come from liftings of geodesics of~$\Sigma_{2,3,4g+2}$, which is enough to deduce the main result (Proposition~\ref{P:237}). The reason why the proof works in this case, unlike for general families of geodesics on~$\Sigma_g$, is that a familly of geodesics on~$\Sigma_{2,3,4g+2}$ lifts to a family on~$\Sigma_g$ that admits many symmetries, and that these symmetries force the associated coefficients~$b_{i,j}$ to live in a small subcone where the bilinear form~$Q_{4g+2}$ is negative. 

It should be noted that, in the case of orbifolds of type~$(2,q,\infty)$, a result similar to Proposition~\ref{P:Template} has been established by Tali Pinsky~\cite{Tali} in a previous work. 
Precisely, when~$\Sigma$ is the orbifold~$\Hy/\PSLZ$, Ghys~\cite{GhysMadrid} proved that the periodic orbits of the geodesic flow can be distorted on a template which coincides with the geometric Lorenz template, so that periodic orbits are Lorenz knots~\cite{BW}. His construction corresponds to ours when~$\Sigma$ is the orbifold $\Hy/\PSLZ$ (which is of type~$(2,3,\infty)$) and $\Tiling$ the tessellation of~$\Hy^2$ by ideal triangles. Later, Pinsky~\cite{Tali} generalized Ghys' construction to orbifolds of type~$(2,q,\infty)$. Her construction can be recovered in our setting using a tiling of~$\Hy^2$ by ideal regular $q$-gons. The presentations of Ghys and Pinsky differ from ours in the sense that they construct a template by opening the cusp in the associated orbifold, thus distorting the underlying manifold~$\un\Sigma$, and then contracting the stable direction of the geodesic flow. The notion of discretisation of geodesics (Definition~\ref{D:Discretisation}) allows us to construct multitemplates even when the considered orbifold has no cusp. 

The plan of the article is as follows. First, we recall some basic definitions---linking number, orbifold, unit tangent bundle, geodesic flow---and prove two general lemmas on left-handed flows in Section~\ref{S:Basics}. We then treat the case of the torus in Section~\ref{S:Torus}. Next, we turn to hyperbolic orbifolds and construct a template for the geodesic flow on every orbifold in Section~\ref{S:Template}, where we prove Proposition~\ref{P:Template}. We then complete the case of orbifolds of type~$(2,q,\infty)$ in Section~\ref{S:q}. We investigate the geodesic flows on surfaces of genus~$g$ and complete the case of the orbifolds of type~$(2,3,4g+2)$ in Section~\ref{S:237}. Finally, we construct the counter-examples of Proposition~\ref{P:CEx} and discuss further questions in Section~\ref{S:Question}.

\noindent{\bf Acknowledgement.} I thank my phD advisor \'Etienne Ghys for numerous discussions on left-handed flows and templates and for his strong support. I also thank Maxime Bourrigan for answering many of my topological questions, and Patrick Massot for explaining me the content of J.\,VanHorn-Morris article~\cite{JVHM}. Finally I thank the two anonymous referees for pointing out numerous vague points in previous versions.

\section{Definition and motivation}
\label{S:Basics}

Here we set the general context. We recall the needed definitions, and establish some preliminary results.

\subsection{Orbifolds and their unit tangent bundles}
\label{S:Orfifolds}

A \emph{Riemannian, orientable, 2-dimensional orbifold}~$\Orb$ is a topological surface locally modelled on a Riemannian surface modulo actions by finite subgroups of rotations~\cite{Thurston}. More precisely $\Orb$ consists of a topological surface~$X_\Orb$ with an atlas of covering charts $\phi_i: V_i\to U_i$, where~$\{U_i\}$ is a collection of open sets of~$X_\Orb$ closed under finite intersections, $\{V_i\}$ is a collection of open sets of a Riemannian surface, such that to each~$V_i$ is associated a finite group~$\Gamma_i$ of rotations of~$V_i$ identifying~$U_i$ with $V_i/\Gamma_i$, and such that every change of charts $\phi_i^{-1}\circ \phi_j$, when defined, consist of isometries. 

In the sequel we will restrict ourselves to orbifolds which are also \emph{good}, meaning that the whole underlying space~$X_\Orb$ admits a finite degree covering by a surface (which needs not to be compact), say~$\Sigma_0$. In this case, the orbifold $\Orb$ can be identified with the quotient~$\Sigma_0/\Gamma_0$ for some discrete subgroup $\Gamma_0$ of~$\Isom(\Sigma_0)$. The universal cover of~$\Sigma_0$ is defined as the universal cover of~$\Orb$, hereafter denoted by~$\tilde\Orb$. One can then identify~$\Orb$ with the quotient~$\tilde\Orb/\Gamma$ for some discrete subgroup~$\Gamma$ of $\Isom(\tilde\Sigma)$. The latter subgroup is called the~\emph{fundamental group} of~$\Orb$. If $\Orb$ has a constant curvature, then~$\tilde\Orb$ is either the sphere~$\Sph^2$, the Euclidean plane~$\R^2$ or the hyperbolic plane~$\Hy^2$. Accordingly, the orbifold~$\Orb$ is  said to be \emph{spherical, Euclidean}, or \emph{hyperbolic}.


By definition, the orbifold structure transports the metrics, so that each point~$x$ of a good 2-orbifold admits a neighbourhood of the form~$V_x/\Gamma_x$ where $V_x$ is an open disc in~$\tilde\Sigma$ and $\Gamma_x$ a finite group of rotations. The order of~$\Gamma_x$ is called the~\emph{index} of~$x$. A point with index $1$ is called~\emph{regular}, otherwise it is called \emph{singular}. It is important to note that singular points are isolated.

We now turn to the unit tangent bundle of an orbifold. Let $\Orb$ be a good 2-orbifold with fundamental group $\Gamma$. Then the action of $\Gamma$ on $\tilde\Sigma$ by isometries is properly discontinuous
. The \emph{unit tangent bundle}~$\un\Orb$ of~$\Orb$ is defined to be the quotient of the total space $\un\tilde\Orb$ of the unit tangent bundle of $\tilde\Sigma$ by the action of $\Gamma$ on the tangent space of $\tilde\Orb$, {\it i.e.}, $\un\Orb = (\un\tilde\Orb)/\Gamma$. 

Let us illustrate this definition with two examples which are important for the sequel. Assume that~$\D^2$ is an open disc. Its unit tangent bundle~$\un\D^2$ then consists of the set of unit vectors tangent to~$\D^2$. The unit tangent vectors based at a given point form a circle, so that the manifold $\un\D^2$ is a solid torus. 

Consider the action of~$\Z/p\Z$ on~$\D^2$ by rotations of angles that are multiple of~$2\pi/p$. The action is not free because the center of~$\D^2$ is fixed. It is the only point with non-trivial stabilizor. The quotient~$\D^2/(\Z/p\Z)$ is then an orbifold. Denote it by~$\D^2_p$. Since the action of~$\Z/p\Z$ is by isometries, it can be extended to the unit tangent bundle~$\un\D^2$. Given a point with polar coordinates~$(x, \theta)$ on~$\D^2$, and a unit tangent vector making an angle~$\phi$ with the horizontal direction, an element~$\bar k$ of $\Z/p\Z$ then acts by $\bar k\cdot (r, \theta, \phi) = (r, \theta+2k\pi/p, \phi+2k\pi/p)$. The action on~$\un\D^2$ is therefore free, and the quotient $\un\D^2/(\Z/p\Z)$ is a manifold. It is the unit tangent bundle~$\un\D^2_p$ to~$\D^2_p$. It is also a solid torus (see Figure~\ref{F:UTBundle}). 

\begin{figure}[hbt]
	\begin{center}
	\begin{picture}(145,100)(0,0)
	\put(0,0){\includegraphics*[scale=.7]{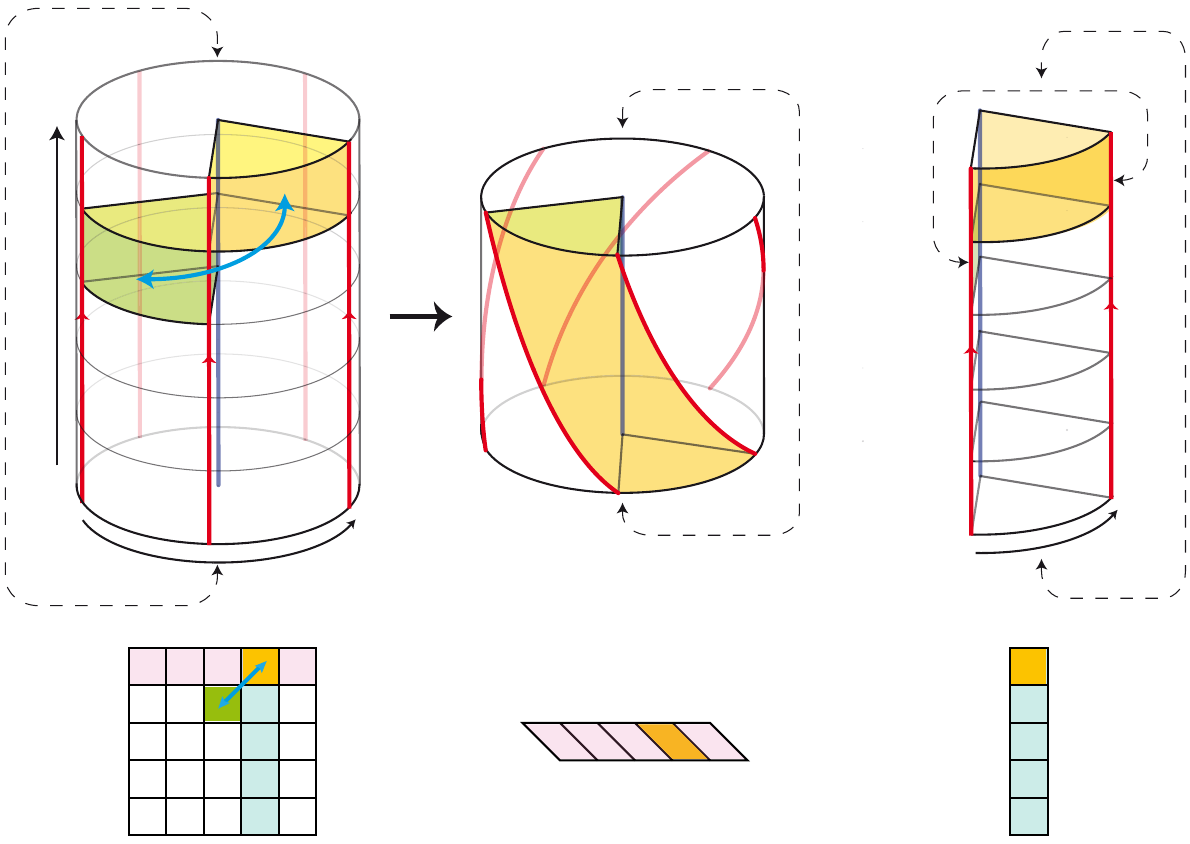}}
	\put(39,32){$\theta$}
	\put(3,79){$\phi$}
	\put(48,64){$\pi$}
	\end{picture}
	\end{center}
	\caption{\small On the top left, the unit tangent bundle $\un\D^2$ to a disc~$\D^2$. It is a solid torus. The action of $\Z/p\Z$ (here with $p=5$) is indicated with a blue arrow. It is a screw-motion. Thus $\un\D^2$ can be seen as a tower formed of $p^2$ pieces of cheese, where the generator of $\Z/p\Z$ acts by a vertical translation plus a $2\pi/p$-rotation. On the bottom left, the boundary of $\un\D^2$ with the rind of the $p^2$ pieces. 
	A horizontal storey of $p$ pieces of cheese is then a fundamental domain for the action (in the center). The quotient is obtained by identifying the floor and the ceiling of the storey with a $-2\pi/p$-rotation. Every meridian disc intersects each fiber $p$ times, except the central fiber, which it intersects only once. This model (called the \emph{storey model}) shows that the unit tangent bundle is a Seifert fibered bundle~\cite{Montesinos}. 
	The $p$ pieces of cheese located between two vertical walls form another fundamental domain (on the right). The quotient is obtained by identifying two vertical walls with a vertical translation of length~$2\pi/p$ (assuming the thickness of the cake to be~$2\pi$). We call this model the \emph{slice-of-cake model}. Figures~\ref{F:TemplateQuotient}, \ref{F:ProjTemp}, \ref{F:ProjOrbit} and \ref{F:ProjTemplate} are drawn using this model.} 
	\label{F:UTBundle}
\end{figure}
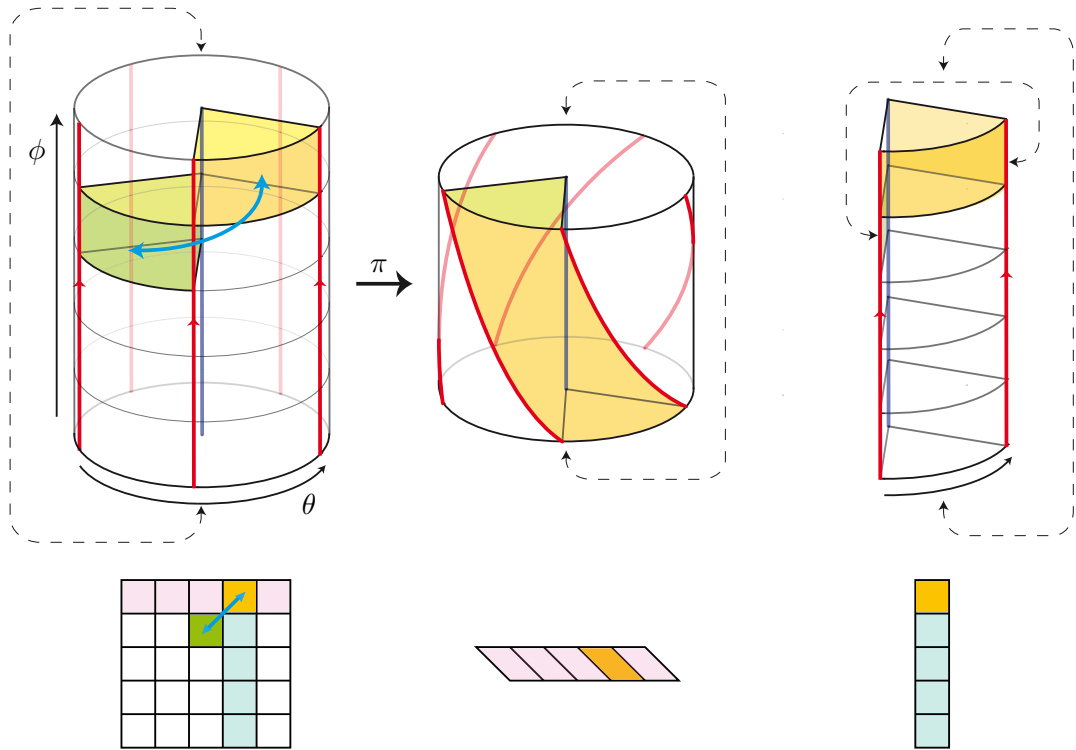

As every point in an orbifold admits a neighbourhood of the form $\D^2$ or~$\D_p^2$ for some~$p$, the unit tangent bundle of every orbifold is obtained by gluing solid tori of type~$\un\D^2$ or $\un\D_p^2$.

\subsection{The geodesic flow on the unit tangent bundle}

Assume that $\Sigma$ is a good $2$-dimensional orbifold. 
The orientation of~$\Sigma$ defines an orientation on the tangent planes, whence an orientation on~$\un\Sigma$. 

Assume now that $\underline\gamma$ is an oriented curve drawn on~$\Sigma$. For~$p$ lying on~$\underline\gamma$, let $T_p(\underline\gamma)$ be the unit tangent vector to~$\underline\gamma$ at~$p$. Then the family of all pairs $(p, T_p(\underline\gamma))$ is an oriented curve in~$\un\Sigma$, the \emph{lifting} of~$\underline\gamma$ in~$\un\Sigma$. 
In particular, the oriented geodesics of~$\Sigma$ are canonically lifted to~$\un\Sigma$. 
More precisely, for every point~$p$ in~$\Sigma$ and every direction~$v$ in~$\Sph^1$, there exists a unique geodesic~$\underline\gamma_{p,v}$ of~$\Sigma$ going through~$p$ with the direction~$v$. Now, for~$t$ in~$\R$ and~$(p, v)$ in~$\un\Sigma$, let us define~$\GF{\Sigma}(t, (p, v))$ to be~$(p', v')$ where $x'$ is the unique point of~$\underline\gamma_{p,v}$ at distance~$t$ from~$p$ and $v'$ is the unit tangent vector to~$\underline\gamma_{p,v}$ at~$p'$. Then $\GF{\Sigma}$ is a continuous map of~$\R \times \un\Sigma$ to~$\un\Sigma$ and, by construction, it is additive in the first coordinate. Hence $\GF{\Sigma}$ is what is called a complete flow on~$\un\Sigma$, and it is naturally called the \emph{geodesic flow} on~$\un\Sigma$. By construction, the liftings of the geodesics of~$\Sigma$ in~$\un\Sigma$ are the orbits of the geodesic flow (but they are not geodesic in~$\un\Sigma$, since no metric has been defined there).  


\subsection{Linking number and left-handed flows}

Assume that $M$ is a 3-manifold, and that~$K$, $K'$ are two null-homologous links in~$M$. Then there exists an oriented surface~$S$ (or even a simplicial 2-chain) with boundary~$K$ that is transverse to~$K'$. The intersection points between~$S$ and~$K'$ then have an orientation, and their sum defines the algebraic intersection number~$\mathrm{Int}(S, K')$. Adding a closed~2-chain to $S$ does not change the intersection number since~$K'$ is null-homologous, so that~$\mathrm{Int}(S, K')$ depends on~$K$ and~$K'$ only. It is the \emph{linking number} of the pair~$K, K'$, denoted by~$\lk(K,K')$.

In the last fifty years, several works~\cite{Schwartzmann, Arnold, GG, B-M} have emphasized the interest of considering a vector field as a long knot, or, more precisely, of considering invariant measures under the flow as (infinite) invariant knots. Following this idea, given a flow~$\Phi$ on a rational homology sphere~$M$, one can generalize the standard definition of the linking number for pairs of periodic orbits to pairs of invariant measures (see Arnold's work on asymptotic linking number~\cite{Arnold}).
Ghys then suggested to look at those flows for which this linking number is always negative, and called them \emph{left-handed flows}. We refer to the original article~\cite{GhysJapan} for a discussion about the motivations and the properties of these flows. Below we only mention the result explaining that, for a flow with many periodic orbits, left-handedness can be deduced from the negativity of the linking numbers of pairs of periodic orbits only. A flow~$\Phi$ is said \emph{knot-shadowable} if, for every $\Phi$-invariant measure~$\mu$, there exists a sequence~$(\gamma_n)$ of periodic orbits of~$\Phi$ such that the sequence of the Dirac measures on~$\gamma_n$ weakly converges to~$\mu$.

\begin{lemma}
\label{L:Approx}
Assume that~$\Phi$ is a knot-shadowable flow. If the linking number of every pair of periodic orbits of~$\Phi$ is negative, then~$\Phi$ is left-handed.
\end{lemma}

\begin{proof}
Assume that~$\mu, \mu'$ are two invariant measures. Let~$(\gamma_n), (\gamma'_n)$ be two distinct sequences of knots that converge to~$\mu, \mu'$. Write~$t_n, t'_n$ for the lengths of~$\gamma_n, \gamma'_n$ respectively. Then it is known~\cite{GhysJapan} that the sequence $\frac 1 {t_nt'_n}\lk(\gamma_n, \gamma'_n)$ converges to~$\lk(\mu, \mu')$, which is therefore negative.
\end{proof}

Lemma~\ref{L:Approx} is useful only for flows that are knot-shadowable. This is the case for flows of Anosov type, and in particular for the geodesic flows on hyperbolic 2-orbifolds. Thus a positive answer to Question~\ref{Q:OneHanded} follows from a positive answer to Question~\ref{Q:OneHandedKnot}. In short, if the curvature is negative, we only have to compute linking numbers of pairs of knots for proving left-handedness.

\subsection{Coverings}

We complete this introductory section with an observation about the behaviour of linking numbers under quotient. The result is easy, but useful, as it gives new left-handed flows from old ones. It will be crucial for establishing the left-handedness of~$\GF{\Orb_{2,3,7}}$ (Proposition~\ref{P:237}).

\begin{lemma}
\label{L:Covering}
Assume that $M, \hat M$ are two 3-manifolds with a covering map~$\pi: \hat M\to M$ of index~$n$. Let~$K, K'$ be two null-homologous links in~$M$. Write~$\hat K, \hat K'$ for the $\pi$-equivariant lifts of~$K, K'$ in~$\hat M$. Then the links~$\hat K, \hat K'$ are null-homologous, and we have $\lk(K, K') = \frac 1 n\lk(\hat K, \hat K')$.
\end{lemma}

\begin{proof}
Let~$S$ be an oriented surface with boundary~$K$. Write~$\hat S$ for its~$\pi$-equivariant lift in~$\hat M$. Then we have~$\pi(\bord \hat S) = K$, hence $\bord \hat S = \hat K$. Therefore~$\hat K$ is also null-homologous. Since~$\hat S$ and~$\hat K'$ are~$\pi$-equivariant, every intersection point of~$S$ with~$K'$ lifts to~$n$ intersection points of~$\hat S$ with~$\hat K'$, so that $\lk(K, K') = \frac 1 n\lk(\hat K, \hat K')$ holds.
\end{proof}

If we have a covering map between two orbifolds~$\hat \Orb\to\Orb$, then the map extends to the unit tangent bundles and it commutes with the geodesic flow. Lemma~\ref{L:Covering} then implies that the sign of the linking numbers in~$\un\Orb$ are the same as those in~$\un\hat\Orb$, so that, if the geodesic flow~$\GF{\hat\Orb}$ is left-handed, so does~$\GF{\Orb}$. For instance, as the geodesic flow on~$\un\Sph^2$ is left-handed~\cite{GhysJapan}, we deduce that the same holds for any quotient of~$\Sph^2$, such as the Poincar\'e sphere~$\Orb_{2,3,5}$.

\begin{coro}
\label{C:Sph2}
Let~$\Orb$ be a spherical 2-orbifold. Then the geodesic flow~$\GF{\Orb}$ is left-handed.
\end{coro}


\section{Birkhoff sections for the geodesic flow on a flat torus}
\label{S:Torus}

This section is devoted to the geodesic flow~$\GF{\T^2}$ on a torus with a flat metric. Our aim is to establish Theorem~B. By the way, we shall completely classify Birkhoff sections up to isotopy and show that (almost) every collection of periodic orbits bounds a Birkhoff section (Theorem~\ref{T:Torus} and Corollary~\ref{C:Torus}).

We first parametrize the geodesic flow on a flat torus and define the polygon~$\Pol\gamma$ associated with a finite collection~$\gamma$ of periodic orbits  (\S\,\ref{SS:Polygon}). Next, we describe how Birkhoff sections may look like, first in the neighbourhood  of so-called regular levels (\S\,\ref{S:RegularLevels}), then in the neighbourhood of critical levels with the help of helix boxes (\S\,\ref{S:Boxes}). Finally, pieces are glued together in~\S\,\ref{S:PolSurface}. 


\subsection{The polygon associated with a collection of periodic orbits}
\label{SS:Polygon}

We show how to encode finite collections of periodic orbits of the geodesic flow~$\GF{\T^2}$ in~$\un\T^2$ using polygons whose vertices have integral coordinates.

Throughout this section, $\T^2$ denotes the torus equipped with a flat metric. By definition, $\T^2$ is a quotient $\R^2/\Z^2$ of the Euclidean plane. For all $p, p'$ in~$\T^2$, the translation by $p'-p$ carries the tangent plane at $p$ to the tangent plane at~$p'$. Therefore, the unit tangent bundle~$\un\T^2$ is $\T^2 \times \Sph^1$. Next, the geodesics of~$\T^2$ are induced by those of~$\R^2$. Their liftings in~$\un\T^2$ are horizontal, that is lie is some level~$\T^2\times\{\theta\}$ for some~$\theta$ in~$\Sph^1$. Hence we have $\GF{\T^2}(t, (x, y, \theta)) = (x + t\cos\theta, y + t \sin\theta, \theta)$.
If $\tan\theta$ is a rational number, then, for every initial value of~$(x, y)$, the associated orbit goes back to~$(x, y)$ in finite time, and, conversely, every finite orbit of~$\GF{\T^2}$ is of this type. In such a case, we define $\theta$ to be the \emph{slope} of the orbit, and the unique pair $(p, q)$ of coprime numbers verifying $\tan\theta = p/q$ and $p$ is of the same sign as~$\cos\theta$ to be the~\emph{code} of the orbit.

Assume that $\gamma$ is a finite collection of periodic orbits of~$\GF{\T^2}$. We define the \emph{combinatorial type} of~$\gamma$ to be the sequence $((n_1, \theta_1, p_1, q_1), \ldots , (n_k, \theta_k, p_k, q_k))$, such that $\gamma$ consists of $n_1$~orbits of slope $\theta_1$, plus $n_2$~orbits of slope $\theta_2$, \ldots , plus $n_k$~orbits of slope $\theta_k$, we have $\tan\theta_1=p_1/q_1, \ldots , \tan\theta_k=p_k/q_k$,  and $\theta_1, \ldots , \theta_k$ make an increasing sequence in~$[0, 2\pi)$. 


\begin{lemma}
\label{L:ZeroSum}
Assume that $\gamma$ is a finite collection of periodic orbits in~$\GF{\T^2}$. Let $((n_1, , \theta_1, p_1, q_1)$, ..., $(n_k, \theta_k, p_k, q_k))$ be the combinatorial type of~$\gamma$. Then the image of~$\gamma$ in~$H_1(\un\T^2; \Z)$ is zero if and only if $\sum n_i(p_i, q_i)=(0,0)$ holds.
\end{lemma}


\begin{proof}
The image of an orbit with slope~$(p, q)$ in~$H_1(\un\T^2; \Z)$ admits the coordinates~$(p, q, 0)$ in the standard basis. Indeed, the class of a straight line with code $(p,q)$ on~$\T^2$ is $(p,q)$ in this basis. As the lifts of the geodesics of~$\T^2$ in~$\un\T^2$ are horizontal, the third coordinate of the lift of a geodesic in~$\un\T^2$ is constant. Therefore the third coordinate of its image in~$H_1(\un\T^2; \Z)$ is zero. The result then  follows from the additivity of homology. 
\end{proof}

Here comes the main definition of this section.

\begin{defi}
\label{D:Polygon}
(See Figure~\ref{F:Polygon}.) Assume that $\gamma$ is a null-homologous collection of periodic orbits in~$\GF{\T^2}$ with combinatorial type $((n_1,\theta_1, p_1, q_1)$, ..., $(n_k, \theta_k, p_k, q_k))$. The \emph{polygon} $\Pol\gamma$ of~$\gamma$ is the $k$-vertex polygon of~$\R^2$ whose $j$th vertex is $\sum_{i=1}^j n_i(p_i, q_i)$ for $j = 1, ..., k$.
\end{defi}

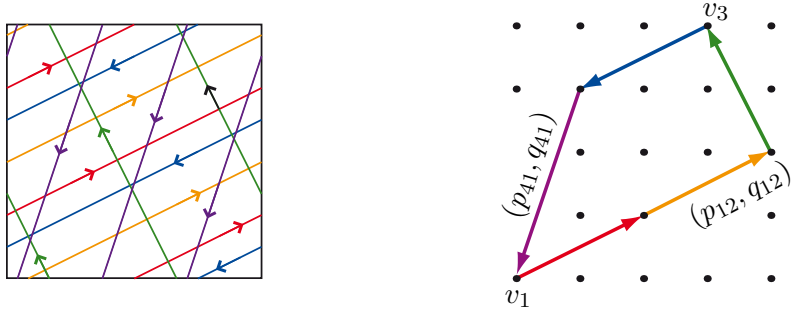
\begin{figure}[htb]
	\begin{center}
	\begin{picture}(115,38)(0,0)
	\put(0,0){	\includegraphics*[width=.7\textwidth]{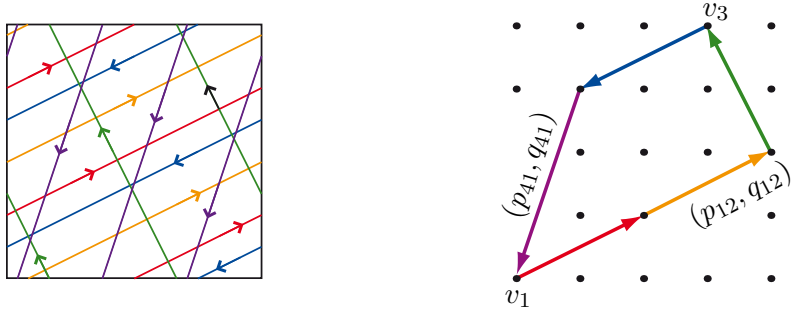}}
	\put(71,1){$v_{1}$}
	\put(97,39){$v_{3}$}
	\put(73,12){\begin{rotate}{73}{$(p_{41}, q_{41})$}\end{rotate}}
	\put(96,11){\begin{rotate}{28}{$(p_{12}, q_{12})$}\end{rotate}}
	\end{picture}
	\end{center}
	\caption{A null-homologous family~$\gamma$ of periodic orbits of the geodesic flow, and the associated polygon~$\Pol{\gamma}$.} 
	\label{F:Polygon}
\end{figure}

Owing to the order condition on the slopes in the combinatorial type, $\Pol\gamma$ is a convex polygon and, as $p_i$ and~$q_i$ are coprime for every~$i$, the only points on the boundary of~$\Pol\gamma$ that have integral coordinates are the vertices plus the intermediate points of the form $\sum_{i=1}^{j-1} n_i(p_i, q_i) + m(p_j, q_j)$ with $m < n_j$.

\subsection{Transverse surfaces and regular levels}
\label{S:RegularLevels}

We now turn to surfaces in~$\un\T^2$ transverse to~$\GF{\T^2}$, with the aim of connecting the existence of such a surface with boundary~$\gamma$ with the properties of the associated polygon~$\Pol\gamma$.

Hereafter, for every~$\theta$ in~$\R/2\pi\Z$, the subset of~$\un\T^2$ made of the points whose last coordinate is~$\theta$ will be called the \emph{$\theta$th level} of~$\un\T^2$, denoted by~$L_\theta$. As $\un\T^2$ is trivial, every level is a copy of~$\T^2$. If $\gamma$ is a null-homologous collection of periodic orbits in~$\GF{\T^2}$ with combinatorial type $((n_1, \theta_1, p_1, q_1)$, ..., $(n_k, \theta_k, p_k, q_k))$, the $k$~angles $\theta_i$, as well as the associated levels of~$\un\T^2$, will be called \emph{$\gamma$-critical}, whereas the other angles will be called \emph{$\gamma$-regular}.

\begin{lemma}
\label{L:RegularLevel}
Assume that $\gamma$ is a null-homologous collection of periodic orbits of~$\GF{\T^2}$ and $S$ is a surface with boundary~$\gamma$ whose interior is transverse to~$\GF{\T^2}$. For~$\theta$ in~$\R/2\pi\Z$, let $S_\theta$ be the intersection of~$S$ with the level~$L_\theta$. Then, if $\theta$ is $\gamma$-regular, $S_\theta$ is a union of disjoint circles. If $\theta, \theta'$ are $\gamma$-regular and the interval $(\theta, \theta')$ contains no $\gamma$-critical angle, $S_\theta$ and $S_{\theta'}$ are homologous.
\end{lemma}

\begin{proof}
By construction, the geodesic flow~$\GF{\un\T^2}$ is tangent to~$L_\theta$ whereas, by assumption, $S$ is transverse to~$\GF{\un\T^2}$. Hence $S$ and $L_\theta$ are transverse. Therefore their intersection is a closed 1-dimensional submanifold of~$L_\theta$, hence a union of parallel disjoint circles. The family~$(S_t)_{t\in[\theta, \theta']}$ provides an isotopy between $S_\theta$ and~$S_\theta'$. These (multi)-curves are therefore homologous.
\end{proof}

In the above context, the multicurve~$S_\theta$ is called a~\emph{stratum} of~$S$. For every $\gamma$-regular value~$\theta$, the stratum~$S_\theta$ is cooriented by the geodesic flow. By convention, we orient it so that the concatenation of the chosen orientation and the orientation of the flow gives the standard orientation on the torus~$L_\theta$. 
With this choice, the class~$[S_{\theta}]$ is a well-defined element of the group~$H_1(L_\theta; \Z)$, the latter being canonically identified with~$H_1(\T^2;\Z)$. Then, Lemma~\ref{L:RegularLevel} implies that $[S_{\theta}]$ is constant when~$\theta$ describe an interval of $\gamma$-regular values. Our goal now is to understand how $[S_{\theta}]$ evolves when $\theta$ passes a~$\gamma$-critical value.

\subsection{Packing into helix boxes}
\label{S:Boxes}

VanHorn-Morris~\cite{JVHM} constructed open book decompositions of the torus bundles over the circle by using special boxes and controlling how they match with each other. We use now the same elementary boxes for decomposing and describing the surfaces whose boundary is transverse to~$\GF{\T^2}$ around critical levels.

\begin{defi}
\label{D:HelixBox}
A~\emph{positive} (\resp \emph{negative}) \emph{helix box} is a cube containing an oriented surface isotopic to the surface depicted on Figure~\ref{F:HelixBox}, called the~\emph{helix}. The oriented boundary of the helix is made of seven oriented segments lying in the faces of the cube, plus one segment, called the~\emph{binding}, lying inside the cube and connecting two opposite faces of the cube.
\end{defi}

\begin{figure}[htb]
	\begin{center}
	\includegraphics[width=.8\textwidth]{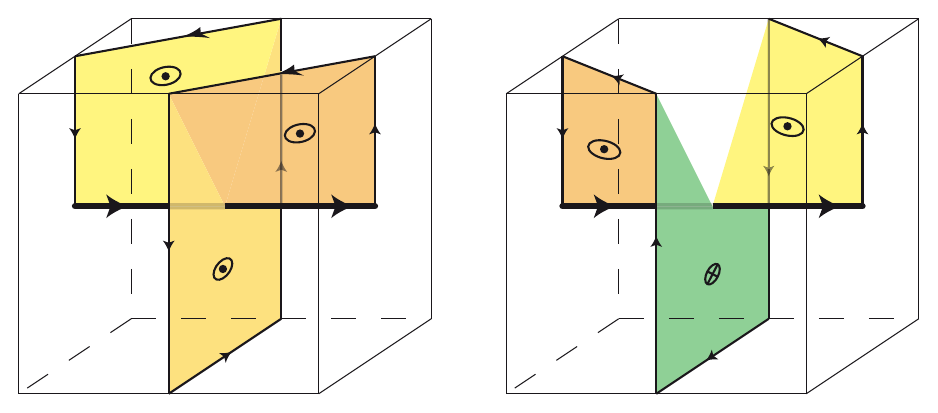}
	\end{center}
	\caption{A positive helix box on the left, a negative helix box. The bindings are in bold. The orientations of the helices are represented by dotted and crossed circles. The interiors of the helices are transverse to the direction of the binding, positively or negatively oriented according to the sign of the box.
} 
	\label{F:HelixBox}
\end{figure}


The next result asserts that almost every surface transverse to~$\GF{\T^2}$ is locally made of helices. 

\begin{lemma}
\label{L:BoxDecomposition}
(See Figures~\ref{F:HelixBox} and~\ref{F:HalfBox}.) Assume that $\gamma$ is a null-homologous collection of periodic orbits of~$\GF{\T^2}$ and $S$ is a surface with boundary~$\gamma$ whose interior is transverse to~$\GF{\T^2}$. Let~$\gamma_i$ be an element of~$\gamma$. Denote by~$L_{\theta_i}$ the $\gamma$-critical level containing~$\gamma_i$. Then there exists a small tubular neighbourhood~$N_{\gamma_i}$ of~$\gamma_i$ of the form~$]\gamma_i-\epsilon, \gamma_i+\epsilon[\times]\theta_i-\eta, \theta_i+\eta[$ in~$\T^2\times\Sph^1$ such that

$(i)$ if the interior of~$S$ does not intersect the level~$L_{\theta_i}$, then the surface $S$ is negatively transverse to~$\GF{\T^2}$ and is locally isotopic to~$\gamma_i\times[\theta, \theta+\epsilon[$ or to $\gamma_i\times]\theta - \epsilon, \theta]$;

$(ii)$ otherwise~$N_{\gamma_i}$ can be decomposed as the union of a positive number~$t_{\gamma_i}$ of helix boxes, which are all positive (\resp negative) if~$S$ is positively (\resp negatively) transverse to~$\GF{\T^2}$, and such that~$\gamma_i$ is identified with the union of the bindings, $S$ is the union of the helices, and the horizontal and vertical faces of~$N_{\gamma_i}$ are identified with the horizontal and vertical faces of the helix boxes.
\end{lemma}

\begin{figure}[htb]
	\begin{center}
	\includegraphics*[width=.6\textwidth]{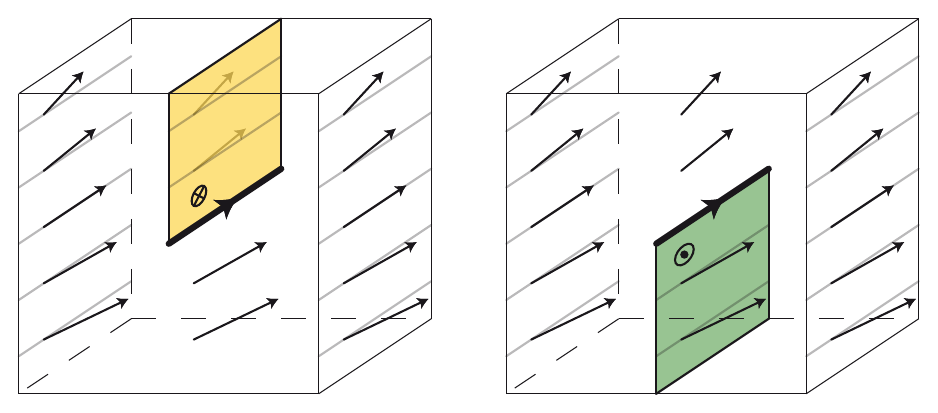}
	\end{center}
	\caption{Case $(i)$ of Lemma~\ref{L:BoxDecomposition} when the surface $S$ is negatively tranverse to the flow and the vector~$\vec{n_p}$ always points in the same half-space. The boundary~$\bord S$ is in bold.
} 
	\label{F:HalfBox}
\end{figure}

\begin{proof}
We write~$N_{\epsilon, \eta}$ for the tubular neighbourhood~$]\gamma_i-\epsilon, \gamma_i+\epsilon[\times]\theta_i-\eta, \theta_i+\eta[$ of~$\gamma_i$ in~$\T^2\times\Sph^1$. 
For every point~$p$ on~$\gamma_i$, we denote by~$\vec{n_p}$ the unique unit vector orthogonal to~$\gamma_i$, tangent to~$S$, and pointing inside~$S$. 
If $\epsilon$ and $\eta$ are small enough, then the intersection of~$S$ with~$N_{\epsilon, \eta}$ is isotopic to the surface generated by~$p+t\vec{n_p}$ when $p$ describes~$\gamma_i$ and $t$ is non-negative. We choose for~$N_{\gamma_i}$ such a neighbourhood. The surface~$L_{\theta_i}$ induces a trivialization of the unit normal bundle $\nu_p(\gamma_i)$ of~$\gamma_i$, so that we can define~$\psi(p)$ to be the angle between~$\vec{n_p}$ and~$L_{\theta_i}$. We then set $d_{\gamma_i}$ to be the degree of the map~$\psi:\gamma_i\simeq\Sph^1\to\nu_p(\gamma_i)\simeq\Sph^1$.

If $S$ is positively transverse to~$\GF{\T^2}$, then~$\psi(p)$ increases as $p$ describes the curve~$\gamma_i$. Therefore the degree $d_{\gamma_i}$ of $\psi$ is positive. We then obtain the helix boxes by cutting~$N_{\gamma_i}$ at each point where~$\vec{n_p}$ points upward. This happens~$d_{\gamma_i}$ times, thus yielding~$d_{\gamma_i}$ positive helix boxes. The result when $S$ is positively transverse follows with $t_{\gamma_i}=d_{\gamma_i}$.

If $S$ is negatively transverse to the flow, then~$\psi$ is a non-increasing function. Indeed, since the geodesic flow is not parallel to~$\gamma_i$, but rotates when level changes, the vector~$\vec{n_p}$ can be constant and the application~$\psi$ can be of degree~$0$, see Figure~\ref{F:HalfBox}. If so, the surface~$S$ lies on one side of~$L_{\theta_i}$ only. It is therefore isotopic to~$\gamma_i\times[\theta_i, \theta_i+\epsilon[$ or to $\gamma_i\times]\theta_i - \epsilon, \theta_i]$, and we are in case~$(i)$ of the statement. Otherwise, the degree $d_{\gamma_i}$ of $\psi$ is negative, and the situation is similar to that in the positive case. The only difference is that the negativity of the intersection of $S$ with~$\GF{\T^2}$ forces~$S$ to wind in the other direction, so that we obtain $-d_{\gamma_i}$~negative boxes. The result then follows with~$t_{\gamma_i}=-d_{\gamma_i}$.
\end{proof}

In the above context, the tubular neighbourhood~$N_{\gamma_i}$ of~$\gamma_i$ is called a~\emph{product-neighbourhood} of~$\gamma_i$. If the interior of~$S$ does not intersect the level~$L_{\theta_i}$ (case~$(ii)$), then $N_{\gamma_i}$ is assumed to be decomposed as a union of~$t_{\gamma_i}$ helix boxes.

Lemma~\ref{L:BoxDecomposition} gives the structure of a surface transverse to the flow around its boundary. The next result decomposes such a surface around an entire critical level.

\begin{lemma}
\label{L:LevelStructure}
Assume that $\gamma$ is a null-homologous collection of periodic orbits of~$\GF{\T^2}$ with combinatorial type~$((n_1, \theta_1, p_1, q_1), \ldots, (n_k, \theta_k, p_k, q_k))$, and $S$ is a surface with boundary~$\gamma$ whose interior is transverse to~$\GF{\T^2}$. Let~$i$ be an element of~$\{1, \ldots, k\}$. Call $\gamma_{i,1}, \ldots, \gamma_{i,n_i}$ the $n_k$ elements of~$\gamma$ lying in the $\gamma$-critical level~$L_{\theta_i}$, and suppose that~$N_{\gamma_{i,1}}, \ldots, N_{\gamma_{i,n_i}}$ are the associated product-neighbourhoods. Then all the curves $\gamma_{i,1}, \ldots, \gamma_{i,n}$ are parallel, and all the numbers~$t_{\gamma_{i,1}}, \ldots, t_{\gamma_{i,n_i}}$ are equal to some number, say~$t_{\theta_i}$. Moreover, if $t_{\theta_i}$ is not zero, there exists a neighbourhood of~$L_{\theta_i}$ of the form $]L_{\theta_i-\epsilon}, L_{\theta_i+\epsilon[}$ which is tiled by $n_i\times t_{\theta_i}$ helix boxes such that, in each helix box, the surface~$S$ coincides with the helix.
\end{lemma}

\begin{proof}
By definition of~$\GF{\T^2}$, every orbit in the level~$L_{\theta_i}$ has direction~$\theta_i$. At the expense of possibly restricting some of them, we can suppose that all rectangular tubular neighbourhoods~$N_{\gamma_{i,j}}$ have the same height~$2\eta$. Then the complement of their union~$N_{\gamma_{i,1}}\cup\cdots\cup N_{\gamma_{i,n_i}}$ in the horizontal thick torus~$]L_{\theta_i-\eta}, L_{\theta_i+\eta}[$ is also the union of~$n_i$ solid tori admitting a rectangular section. We denote these tori by~$M_{i,1}, \ldots, M_{i,n_i}$. At the expense of possibly permuting the names, we can suppose that, for every~$j$, the torus~$M_{i,j}$ lies between the tori~$N_{\gamma_{i,j}}$ and~$N_{\gamma_{i,j+1}}$. Since $S$ is transverse to the flow, its intersection with~$M_{i,j}$ is transverse to the direction~$\theta_i$. Therefore it is the union of a certain number, say~$s_{i,j}$, of discs whose boundaries are meridian circles in the solid torus~$M_{i,j}$. 

If, for some~$j$, the number $t_{\gamma_{i,j}}$ is zero, then the two vertical boundaries of~$N_{\gamma_{i,j}}$ do not intersect~$S$. Therefore, the intersection of~$S$ with~$M_{i,j}$ is empty, which implies $s_{i,j}=0$. Considering the other boundary of~$M_{i,j}$, we get~$t_{\gamma_{i,j+1}}=0$. By induction, we get $t_{\gamma_{i,j}}=0$ for every~$j$.

If, for some~$j$, the number $t_{\gamma_{i,j}}$ is not zero, then~$N_{\gamma_{i,j}}$ is tiled into~$t_{\gamma_{i,j}}$ helix boxes. Therefore the boundary between~$M_{i,j}$ and~$N_{\gamma_{i,j}}$ is an annulus that intersects~$S$ along~$t_{\gamma_{i,j}}$ vertical segments, and we deduce $s_{i,j}=t_{\gamma_{i,j}}$. Considering the other vertical boundary of~$M_{i,j}$, we get~$s_{i,j}=t_{\gamma_{i,j+1}}$, and therefore $t_{\gamma_{i,j+1}}=t_{\gamma_{i,j}}$. By induction, all numbers~$t_{\gamma_{i,j}}$ are equal to some fixed number, say~$t_{\theta_i}$. Finally, since the intersection of $S$ with $M_{i,j}$ consists of discs only, we can extend the solid tori~$N_{\gamma_{i,j}}$ so that their union covers the whole neighbourhood~$]L_{\theta_i-\eta}, L_{\theta_i+\eta}[$. Since every~$N_{\gamma_{i,j}}$ is tiled by~$t_{\theta_i}$ helix boxes, the thick torus~$]L_{\theta_i-\eta}, L_{\theta_i+\eta}[$ is tiled by~$n_i\times t_{\theta_i}$ helix boxes.
\end{proof}

Considering for a moment the angular parameter~$\theta$ as a (periodic) time, a surface transverse to~$\GF{\T^2}$ can be seen as the movie of its strata. By Lemma~\ref{L:RegularLevel}, the strata vary continuously as long as $\theta$ is regular. Using Lemma~\ref{L:LevelStructure}, we can now describe how the strata evolve when~$\theta$ crosses a~critical value.

\begin{lemma}
\label{L:RelHomol}
In the context of Lemma~\ref{L:LevelStructure}, if the surface $S$ is negatively transverse~$\GF{\T^2}$, then for every $\gamma$-critical angle~$\theta_i$, the homology classes of the strata~$S_{\theta_i-\eta}$ and $S_{\theta_i+\eta}$ are related by the srelation
\begin{equation}
\label{Eq:RelHomNeg}
[S_{\theta_i+\eta}]=[S_{\theta_i-\eta}] + n_i(p_i, q_i).
\end{equation}
If $S$ is positively transverse to~$\GF{\T^2}$, then we have
\begin{equation}
\label{Eq:RelHomPos}
[S_{\theta_i+\eta}]=[S_{\theta_i-\eta}] - n_i(p_i, q_i).
\end{equation}
\end{lemma} 

\begin{figure}[h]
	\begin{center}
	\begin{picture}(120,58)(0,0)
	\put(0,0){\includegraphics*[scale=0.84]{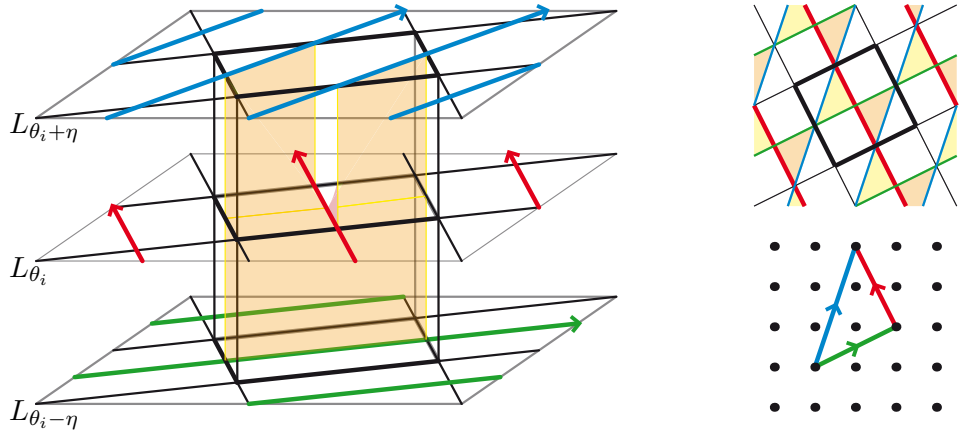}}
	\put(-2,1){$L_{\theta_i-\eta}$}
	\put(-2,20){$L_{\theta_i}$}
	\put(-2,39){$L_{\theta_i+\eta}$}
	\end{picture}
	\end{center}
	\caption{On the left, a surface~$S$ with boundary~$\gamma$ transverse to~$\GF{\T^2}$ in a neighbourhood of a $\gamma$-critical level~$L_{\theta_i}$ of the form~$]L_{\theta_i}-\eta, L_{\theta_i}+\eta[$. It is tiled by five negative helix boxes. Here, there is only one component, say~$\gamma_{i,1}$, of~$\gamma$ in~$L_{\theta_i}$ (in red), that is, we have~$n_i=1$. The intersection of $S$ with one of the five helix boxes is depicted. Its boundary consists of one fifth of the curve~$\gamma_{i,1}$, one fifth of the stratum~$S_{\theta_i+\eta}$ (on the top, in blue), one fifth of~$S_{\theta_i-\eta}$ (on the bottom, in green), and of vertical segments which are glued to the four other boxes. On the top right, the projection on a horizontal torus. On the bottom right, the homological relation between $n_i(p_i, q_i)$, $[S_{\theta_i-\eta}]$ and $[S_{\theta_i+\eta}]$ stated in Lemma~\ref{L:RelHomol}, here with $n_i=1$, $(p_i, q_i) = (-1, 2)$, $[S_{\theta_i-\eta}]=(2,1)$ and $[S_{\theta_i+\eta}]=(1,3)$. According to Lemma~\ref{L:PPolygon} $(ii)$, the area of this homological triangle $(5/2)$ is half the number of helix boxes involved in the tiling of the neighbourhood of the $\gamma$-critical level~$L_{\theta_i}$.}
	\label{F:CriticalLevel}
\end{figure}
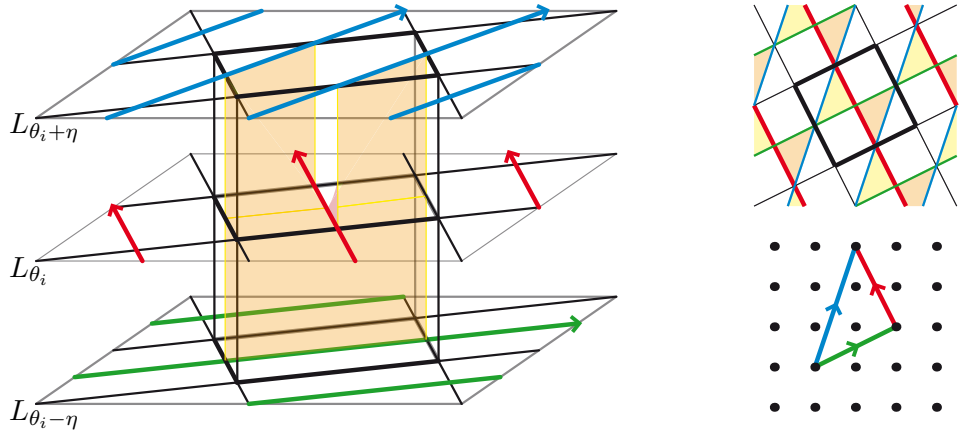

\begin{proof}
We continue with the notation of Lemma~\ref{L:LevelStructure}. In particular, we assume that the neighbourhood~$]L_{\theta_i-\eta}, L_{\theta_i+\eta}[$ of~$L_\theta$ is tiled with~$n_i\times t_{\theta_i}$ helix boxes. The boundary of the intersection of the surface~$S$ with~$]L_{\theta_i-\eta}, L_{\theta_i+\eta}[$ consists of of pieces of three types: the curves~$\gamma_{i,1}, \ldots, \gamma_{i,n_i}$, the stratum~$S_{\theta_i-\eta}$, and the stratum~$S_{\theta_i+\eta}$. Therefore, the sum of these curves, with the orientations induced by the surface~$S$, is null-homologous in~$\un\T^2$. After projection on~$\T^2$, this sum is still zero.

When $S$ is negatively transverse to~$\GF{\T^2}$, then the two orientations on~$S_{\theta_i-\eta}$ given by~$S$ and by~$\GF{\T^2}$ agree, whereas the two orientations on~$S_{\theta_i+\eta}$ are opposite. We thus get~$n_i(p_i, q_i) + [S_{\theta_i-\eta}] - [S_{\theta_i+\eta}] =0$. Similarly, when $S$ is positively transverse to~$\GF{\T^2}$, the two orientations on~$S_{\theta_i-\eta}$ are opposite, whereas the two orientations on~$S_{\theta_i+\eta}$ agree, yielding Equation~\eqref{Eq:RelHomPos}.
\end{proof}


\subsection{Correspondence between pointed polygons and transverse surfaces}
\label{S:PolSurface}

We can now associate with every surface transverse to~$\GF{\T^2}$ a polygon in the lattice~$H_1(\T^2; \Z)$ that encodes the homology classes of all strata simultaneously.

\begin{defi}
\label{D:PointedPolygon}
Assume that $\gamma$ is a null-homologous collection of periodic orbits of~$\GF{\T^2}$ and $S$ is a surface with boundary~$\gamma$ whose interior is transverse to~$\GF{\T^2}$. The~\emph{pointed polygon}~$\PPol{S}$ of~$S$ is the polygon of~$\R^2$ whose vertices are the points $[S_\theta]$ for $\theta$ a $\gamma$-regular angle. 
\end{defi}

\begin{lemma}
\label{L:PPolygon}
$(i)$ In the above context, let $((n_1, \theta_1, p_1, q_1), \ldots, (n_k, \theta_k, p_k, q_k))$ be the combinatorial type of~$\gamma$ and~$\Pol\gamma$ be the polygon associated with $\gamma$. If $S$ is negatively transverse to~$\GF{\T^2}$, then the polygon~$\PPol{S}$ is a pointed copy of~$\Pol\gamma$. If $S$ is positively transverse to~$\GF{\T^2}$, then~$\PPol{S}$ is obtained from~$\Pol\gamma$ by a reflection.

$(ii)$ For every~$\gamma$-critical angle~$\theta_i$, the number of helix boxes used for tessellation a neighbourhood~$]L_{\theta_i-\eta}, L_{\theta_i+\eta}[$ is equal to the area of the parallelogram spanned by the vectors~$[S_{\theta_i-\eta}]$ and $[S_{\theta_i+\eta}]$.
\end{lemma}

\begin{proof}
By Lemma~\ref{L:RegularLevel}, the polygon~$\PPol S$ has at most~$k$ vertices.
Now if $S$ is negatively transverse to~$\GF{\T^2}$, then \eqref{Eq:RelHomNeg} implies that, for every~$i$, the two vertices $[S_{\theta_i-\eta}]$ and $[S_{\theta_i+\eta}]$ of~$\PPol{S}$ differ by the vector~$n_i(p_i,q_i)$. On the other hand, if $S$ is positively transverse to~$\GF{\T^2}$, then $[S_{\theta_i-\eta}]$ and $[S_{\theta_i+\eta}]$ differ by~$-n_i(p_i,q_i)$. This proves~$(i)$.

For $(ii)$, we see on Figure~\ref{F:CriticalLevel} that, for every~$i$, every helix box used in the tiling of the neighbourhood~$]L_{\theta_i-\eta}, L_{\theta_i+\eta}[$ of~$L_{\theta_i}$ is above an intersection point between the projection of the curve~$S_{\theta_i+\eta}$ and the projection of one component of~$\gamma$ in~$L_{\theta_i}$. Therefore the number of helix boxes in the tiling is the absolute value of the intersection number of the classes $[S_{\theta_i+\eta}]$ and $n_i(p_i,q_i)$ in~$H_1(\T^2; \Z)$. As depicted on Figure~\ref{F:CriticalLevel} right, this number coincides with the absolute value of the intersection number of $[S_{\theta_i-\eta}]$ and $[S_{\theta_i+\eta}]$, which is the area of the parallelogram spanned by these two vectors.
\end{proof}

Assume that $S$ is a surface transverse to~$\GF{\T^2}$ and $L_\theta$ is a regular level of~$\un\T^2$ for~$S$, so that the stratum~$S_\theta$ is a smooth multicurve. Then we obtain another surface~$S'$ transverse to~$\GF{\T^2}$ by cutting $S$ along~$S_\theta$, gluing a copy of~$L_\theta$, and smoothing. We say that $S'$ is obtained from $S$ by~\emph{horizontal surgery}. It is easy to check that the polygons~$\PPol S$ and~$\PPol{S'}$ coincide although the surfaces~$S$ and~$S'$ are not isotopic.  Therefore pointed polygons do not encode all information about the isotopy type of transverse surfaces. Nevertheless, we will see that horizontal surgeries are the only freedom left by polygons. 

For~$\gamma$ a null-homologous collection of periodic orbits of~$\GF{\T^2}$ with associated polygon~$\Pol\gamma$, we write~$A(\gamma)$ for the area of~$\Pol\gamma$ (which is an integer by Pick's Formula) and $I(\gamma)$ for the number of integer points in the strict interior of~$\Pol\gamma$. We can now state and establish the main result. 

\begin{theo}
\label{T:Torus}
$(i)$ The map~$S\mapsto \PPol S$ induces a one-to-one correspondence between surfaces negatively transverse to the flow~$\GF{\T^2}$ with boundary made of periodic orbits, up to isotopy and horizontal surgeries, and convex polygons with integer vertices in~$\R^2$ containing the origin in their interior or on their boundary. 

$(ii)$ The map~$S\mapsto \PPol S$ induces a one-to-one correspondence between negative Birkhoff sections for the flow~$\GF{\T^2}$ and convex polygons with integer vertices in~$\R^2$ containing the origin their (strict) interior.

$(iii)$ There is no surface positively transverse to~$\GF{\T^2}$ with boundary made of periodic orbits.

$(iv)$ Assume that~$\gamma$ is a null-homologous collection of periodic orbits of~$\GF{\T^2}$ with associated polygon~$\Pol\gamma$. Then for every surface~$S$ transverse to~$\GF{\T^2}$ with boundary~$\gamma$, the Euler characteristic of~$S$ is~$-2A(\gamma)$ and the genus of~$S$ is~$I(\gamma)$.

$(v)$ Assume that $\gamma, \gamma'$ are two null-homologous collections of periodic orbits of~$\GF{\T^2}$. Then ther linking number $\lk(\gamma, \gamma')$ is equal to $A(\gamma) + A(\gamma') - A(\gamma\cup\gamma')$.
\end{theo}

\begin{proof}
$(i)$ (See Figure~\ref{F:PPol}.) Assume that~$S$ is a surface negatively transverse to~$\GF{\T^2}$. Let~$\gamma$ be its boundary and $((n_1, \theta_1, p_1, q_1), \ldots, (n_k, \theta_k, p_k, q_k))$ be the combinatorial type of~$\gamma$. For every~$\gamma$-regular angle~$\theta$, the stratum~$S_\theta$ is transverse to the direction~$\theta$. Therefore, if $S_\theta$ is non-empty and with the orientation of $S_\theta$ defined in Section~\ref{S:RegularLevels}, the basis formed by a vector tangent to~$S_\theta$ and a vector with direction~$\theta$ is direct. Hence the basis formed by the direction of~$[S_\theta]$ and the direction~$\theta$ is also direct. Let $D_\theta$ be the line oriented by~$\theta$ passing through the vertex~$S_\theta$ of~$\PPol S$. The previous observation implies that the point~$(0,0)$ is on the left of~$D_\theta$. Let $\theta_i$ be the smallest~$\gamma$-critical angle larger than~$\theta$. Then, when $\theta$ tends to~$\theta_i$, the line~$D_\theta$ tends to the line supporting the edge of~$\PPol S$ with direction~$\theta_i$. Therefore, the point~$(0,0)$ is also on the left of the edge of~$\PPol S$ with direction~$\theta_i$ (the boundary of~$\PPol S$ being oriented trigonometrically). If the stratum $S_\theta$ is empty, we have~$[S_\theta]=0$, and $(0,0)$ is also on the left the line~$D_\theta$. Taking all critical value of~$\theta$ into account, we deduce that the point~$(0,0)$ is on the left of all oriented edges of~$\PPol S$, and therefore lies in the interior or on the boundary of~$\PPol S$. Thus the map~$\PPol{}$ associates with every surface transverse to~$\GF{\T^2}$ a polygon in~$H_1(\T^2; \Z)$ containing~$(0,0)$ in its interior.

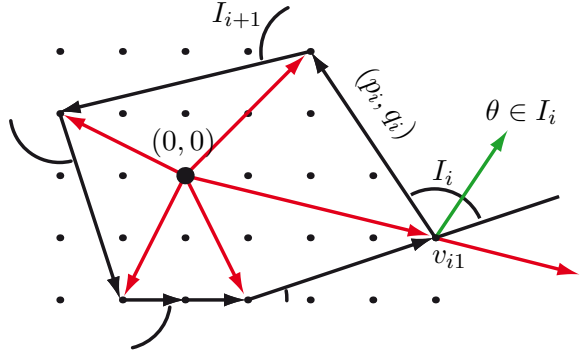
\begin{figure}[h]
	\begin{center}
	\begin{picture}(78,48)(0,0)
	\put(0,0){\includegraphics*[width=0.5\textwidth]{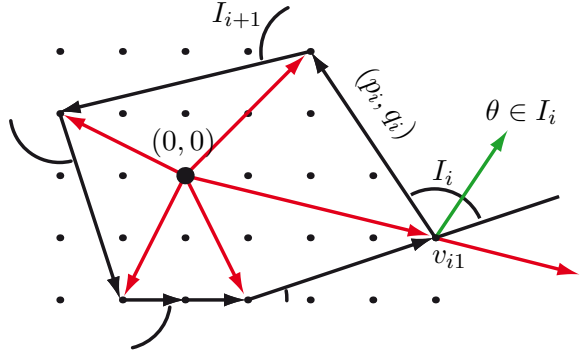}}
	\put(20,27){$(0,0)$}
	\put(64,31){$\theta\in I_i$}
	\put(57,23){$I_i$}
	\put(28,44){$I_{i+1}$}
	\put(57,12){$v_{i1}$}
	\put(47,37){\begin{rotate}{-55}{$(p_{i}, q_{i})$}\end{rotate}}
	\end{picture}
	\end{center}
	\caption{The Polygon~$\PPol S$ for a surface $S$ negatively transverse to the geodesic flow. For a $\gamma$-regular angle~$\theta$, the directions $[S_\theta]$ (indicated by a red arrow) and $\theta$ (indicated by a green arrow) form a direct basis. The point~$(0,0)$ is on the left of all edges of~$\PPol S$, and therefore in the interior or on the boundary of~$\PPol S$.}
	\label{F:PPol}
\end{figure}

For the surjectivity of the map~$\PPol{}$, suppose that a convex polygon~$P$ containing~$(0,0)$ is given. 
Let $\theta_0$ be a $\gamma$-regular angle. 
Let~$V$ be the unique vertex of~$P$ such that the line of slope~$\theta_0$ passing through~$V$ lies on the right of~$P$. 
Then construct a surface~$S$ a follows. Start with a stratum~$S_{\theta_0}$ that is transverse to the $\theta_0$-direction and whose homology class is~$V$. Let $(p_1, q_1)$ denote the edge of~$P$ starting at $V$. Then erect helix boxes whose bindings have direction $(p_1, q_1)$ so that their bottom faces match with~$S_{\theta_0}$. By Lemma~\ref{L:RelHomol}, the boundary of the helices in the top faces form a curve whose homology class is $V+(p_1,q_1)$, so that the stratum corresponds to the second vertex of~$P$. By continuing this procedure of gluing helix boxes whose direction is prescribed by the edges of~$P$ and whose number is dictated by the strata that have been constructed previously, we erect a surface which is negatively transverse to~$\GF{\T^2}$ and whose associated polygon is~$P$. 

For the injectivity, note that the surface~$S$ can be recovered from~$\PPol S$ by the above procedure. The only choice arises when $\theta$ has described the whole circle~$\Sph^1$ and comes back to~$\theta_0$: we have to glue the last floor of helix boxes to the stratum~$S_\theta$. This gluing is not unique, but two such gluings precisely differ by a horizontal surgery.

$(ii)$ Assume that~$S$ is a negative Birkhoff section for~$\GF{\T^2}$. As $S$ is transverse to the flow, we can apply the result of~$(i)$ and deduce that the polygon~$\PPol S$ contains $(0,0)$ in its interior or on its boundary. Since~$S$ is a Birkhoff section, it intersects all orbits of~$\GF{\T^2}$. In particular, this implies that for every $\gamma$-regular value of~$\theta$, the stratum~$S_\theta$ is non-empty. This excludes the case where~$(0,0)$ lies on the boundary of~$\PPol S$.

$(iii)$ Assume that~$S$ is a surface with boundary positively transverse to~$\GF{\T^2}$. Then we can apply the same argument as in the negative case~$(i)$. The only difference is that, for every $\gamma$-regular angle~$\theta$, the basis formed by~$[S_\theta]$ and $\theta$ is indirect. Therefore, the point~$(0,0)$ lies on the right of the line with direction~$\theta$ passing through the vertex~$[S_\theta]$. Thus~$(0,0)$ is on the right of all edges of~$\PPol S$, whereas the boundary is oriented trigonometrically, a contradiction.

$(iv)$ In every helix box, the helix surface consists of a topological disc, of eight edges, seven of them being on the boundary of the box, and of eight vertices, two of them being in the center of a face of the box and the six others in the middle of an edge of the box. Therefore, the contribution of a helix box to the Euler characteristic is $1-(1+7/2)+(2/2+6/4)=-1$. Assume that~$S$ is a surface transverse to~$\GF{\T^2}$ with boundary~$\gamma$, and let~$\PPol S$ be the associated polygon. Let $\theta, \theta'$ be two $\gamma$-regular angles such that there is exatly one $\gamma$-critical value in~$]\theta, \theta'[$. Then, according to Lemma~\ref{L:PPolygon}, the number of helix boxes used for tiling the thick torus lying between the two levels~$L_{\theta}$ and $L_{\theta'}$ is twice the area of the triangle whose vertices are~$(0,0)$, $[S_\theta]$ and~$[S_{\theta'}]$. By summing over all~$\gamma$-critical levels, we obtain that the total number of helix boxes in twice the area of~$\PPol S$, hence twice the area of~$\Pol\gamma$. As the genus of~$S$ is given by the formula $\chi(S) = 2-2g(S)-\sum{n_i}$, Pick's Formula for the area of a polygon with integral vertices gives the formula for the genus.

$(v)$ By definition, the linking number $\lk(\gamma, \gamma')$ is the intersection number of a surface with boundary~$\gamma$ and the collection~$\gamma'$. It is well-defined when~$\gamma$ is null-homologous, since, in this case, the intersection number does not depend on the choice of the surface. Here, let us pick a Birkhoff section for~$\GF{\T^2}$ with boundary~$\gamma$, and call it~$S_\gamma$. Let $((n'_1, \theta'_1, p'_1, q'_1)$, \ldots, $(n'_k, \theta'_k, p'_k, q'_k))$ be the combinatorial type of~$\gamma'$. Then the intersection number of~$S_\gamma$ with a periodic orbit of~$\GF{\T^2}$ of slope~$(p'_i, q'_i)$ is the opposite of the area of the parallelogram spanned by the vectors~$[S_{\theta'_i}]$ and $(p'_i, q'_i)$. Since the area of~$\PPol S$ equals the area of~$\Pol \gamma$, the jigsaw puzzle depicted on Figure~\ref{F:SumPolygons} shows that the sum of the areas of these parallelograms is equal to~$A(\gamma\cup\gamma')-A(\gamma)-A(\gamma')$. 
\end{proof}

\begin{figure}[h]
	\begin{center}
	\begin{picture}(140,55)(0,0)
	\put(0,0){\includegraphics*[width=0.85\textwidth]{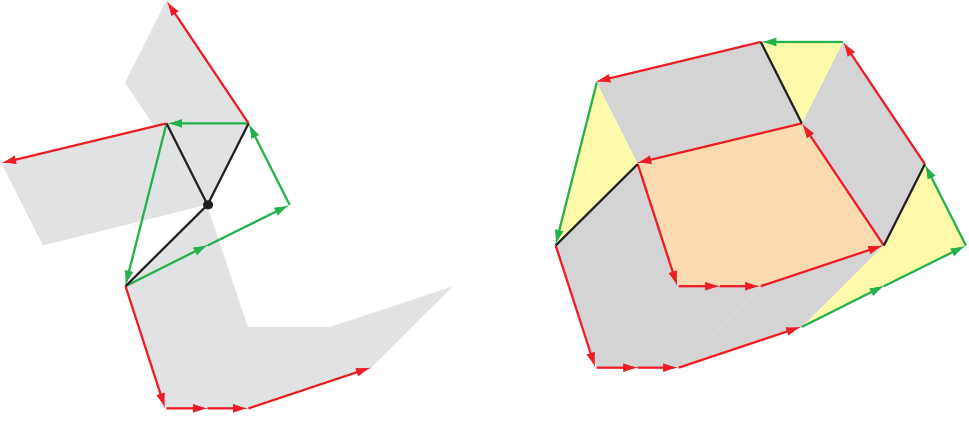}}
	\end{picture}
	\end{center}
	\caption{On the left, the polygon~$\PPol S$ associated to a Birkhoff section~$S$ of~$\GF{\T^2}$, and the homology classes of the elements of a family~$\gamma'$, each of them based at the vertex of~$\PPol S$ corresponding to the class of the intersected stratum of~$S$. In grey, the parallelograms whose areas add up to the intersection number of~$S$ with the collection~$\gamma'$. On the right, the polygon~$\Pol{\gamma\cup\gamma'}$ is decomposed into three parts whose areas respectively are the area of the grey zone, the area of~$\PPol{S}$, and the area of~$\Pol{\gamma'}$.}
	\label{F:SumPolygons}
\end{figure}

\begin{coro}[Theorem~B]
\label{C:Torus}
Assume that~$\Orb$ is a quotient of~$\T^2$ on which any two geodesics intersect. Then for every pair~$\gamma, \gamma'$ of periodic orbits of~$\GF{\Orb}$, the inequality $\lk(\gamma, \gamma') < 0$ holds.
\end{coro}

\begin{proof}
By Lemma~\ref{L:Covering}, it is enough to shows that the lifts~$\hat\gamma, \hat\gamma'$ of~$\gamma, \gamma'$ in~$\un\T^2$ have a negative linking number. As the projections of~$\gamma, \gamma'$ on~$\Orb$ intersect, the projections of~$\hat\gamma, \hat\gamma'$ on~$\T^2$ also intersect. Formula~$(v)$ in Theorem~\ref{T:Torus} shows that the linking number of two collections is zero if and only if the latter consist of parallel lifts of one geodesics on~$\T^2$. The hypothesis on the  intersection then discards this situation.
\end{proof}

Theorem~\ref{T:Torus}\,$(ii)$ implies that almost every null-homologous collection of periodic orbits of~$\GF{\T^2}$ bounds a Birkhoff section. The exceptions are the collections whose associated polygon contains no point with integral coordinates. 

For example, let~$\gamma$ be an unoriented periodic geodesics on~$\T^2$. Let $(p,q)$ be its code. Denote by~$\gamma_+, \gamma_-$ its two lifts in~$\un\T^2$ (one for each of the two possible orientations of~$\gamma$). Then $\gamma_+$ and~$\gamma_-$ are periodic orbits of~$\GF{\T^2}$, and their sum is null-homologous. The associated polygon is made of one segment with coordinates~$(p,q)$ only. As predicted by Theorem~B\,$(i)$, the union of $\gamma_+$ and $\gamma_-$ bounds two non-isotopic surfaces which are transverse to~$\GF{\T^2}$, namely the two vertical ribbons in~$\un\T^2$ consisting of the unit tangent vectors which are based on~$\gamma$ and which point into one of the two sides of~$\gamma$. None of these two ribbons is a Birkhoff section for~$\GF{\T^2}$ since each of them only intersect half of the orbits. 

For another example, consider the three orbits with respective slopes $(1,0), (0,1)$ and $(-1,-1)$. They bound three non-isotopic surfaces transverse to~$\GF{\T^2}$, but they do not bound any Birkhoff section, since the associated polygon is a triangle whose interior contains no point with integral coordinates. 

A last example, which was a surprise for us, is given by the four orbits with slopes~$(\pm 1,0)$ and $(0, \pm 1)$, in which case the associated polygon is the unit square, again containing no integral point inside.

As explained in the introduction, Birkhoff sections give rise to open book decompositions for the underlying 3-manifold, here for unit tangent bundle~$\un\T^2$, a 3-torus. Planar open book decompositions, that is, decompositions where the pages are of genus 0, have been often investigated. Theorem~B\,$(iv)$ implies that none of them comes from Birkhoff sections of the geodesic flow on the torus.

\begin{coro}
The geodesic flow on~$\un\T^2$ contains no Birkhoff section of genus~$0$.
\end{coro}

Since helix boxes contribute~$-1$ to the Euler characteristics, and since every helix box involves one boundary component, Birkhoff sections with genus~$1$ are very peculiar. 

\begin{coro}
A Birkhoff section of genus 1 for~$\GF{\T^2}$ is made of exactly one helix box per boundary component.
\end{coro}

In the article where he introduced the now called Birkhoff sections~\cite{Birkhoff}, Birkhoff gave examples by constructing sections for the geodesic flow on the unit tangent bundle of every surface~$\Sigma$. More precisely, a collection of periodic orbits of $\GF{\Sigma}$ is said~\emph{symmetric} if, for every element of the collection, the orbit corresponding to the opposite orientation of the underlying geodesics also belongs to the collection. Birkhoff showed that every large enough symmetric collection~$\gamma$ of periodic orbits of~$\GF{\T^2}$ bounds a section. In the case when $\Sigma$ is a torus with a flat metric, the symmetry hypothesis implies that the polygon~$\Pol \gamma$ is symmetric. The section constructed by Birkhoff corresponds to the surface~$S$ whose associated polygon~$\PPol S$ in pointed in the center, that is, contains~$(0,0)$ as symmetry center.

\section{Templates for the geodesic flow of a hyperbolic orbifold}
\label{S:Template}

We turn to hyperbolic orbifolds. The aim of this section is to show how the geodesic flow associated with an arbitrary hyperbolic 2-orbifold~$\Orb$ can be distorted onto a certain multitemplate (Definition~\ref{D:MultiTemplate}) lying inside~$\un\Orb$. 
The important property of this distortion is that its restriction to periodic orbits is an isotopy (Proposition~\ref{P:Template}), so that the topological properties of the periodic orbits of the geodesic flow can be studied using this multitemplate. 
What makes the construction possible is that distinct periodic geodesics on a hyperbolic orbifold never point in the same direction at infinity.
Our strategy is similar to Birman--Williams'~\cite{BW} who contract the stable direction of a hyperbolic flow. 
The characteristic here is that the explicit nature of the geodesic flow make it possible to perform the construction in full detail.

Let $\Orb$ be a good hyperbolic 2-orbifold, and let $\Gamma$ denote its fundamental group. Our strategy for constructing the template adapted to the geodesic flow~$\GF{\Orb}$ is as follows. We first choose an adapted tessellation of the universal cover~$\Hy^2$ of~$\Orb$, namely, a $\Gamma$-invariant tesselation such that every tile contains at most one point whose stabilizor has order larger than~2. We also choose in every tile a smooth immersed graph pairwise connecting the sides in such a way that the graphs associated with adjacent tiles match on their common side. We then distort all geodesics in the hyperbolic plane into quasi-geodesics consisting of edges of the graphs so constructed (\S\,\ref{S:Discretisation}). Next, we lift this deformation in the unit tangent bundle~$\un\Orb$ by forcing every tangent vector to always point toward its initial direction at infinity (\S\,\ref{S:Lifting}). Then the image of the deformation at time~$1$ provides the expected (multi)template. It naturally carries a flow, namely the image of the geodesic flow by the deformation (\S\,\ref{S:MultiTemplate}). 

\subsection{Discretisation of geodesics}
\label{S:Discretisation}

The construction starts with a tessellation of the hyperbolic plane that behaves nicely with respect to the orbifold.

\begin{defi}
\label{D:Tiling}
Assume that $\Gamma$ is a Fuchsian group. Let $\Orb$ denote the orbifold~$\Hy^2/\Gamma$. A tessellation $\Tiling$ of~$\Hy$ is \emph{adapted} to~$\Orb$ if 

$(i)$ $\Tiling$ is $\Gamma$-invariant;

$(ii)$ every tile of~$\Tiling$ is a convex polygon (with possibly some vertices on~$\bord \Hy^2$);

$(iii)$ every tile of~$\Tiling$ contains at most one singular point in its interior, and points of index at most~2 on its boundary;

$(iv)$ every tile of~$\Tiling$ has a finite stabilizor in~$\Gamma$;

$(v)$ if $\Tile, \Tile'$ are adjacent tiles of~$\Tiling$ separated by a side~$e_0$, then, for all other sides $e$ of $T$ and $e'$ of $T'$ not both adjacent to~$e_0$, the two geodesics respectively containing $e$ and $e'$ do not intersect.
\end{defi}

For example, assume that~$\Sigma$ is a hyperbolic compact surface. Consider a convex polygonal fundamental domain~$D$ for the action of~$\pi_1(\Sigma)$ on~$\Hy^2$. Then the tessellation formed by the images of $D$ under the action of~$\pi_1(\Sigma)$ is adapted to~$\Sigma$. Note that Condition $(iii)$ prevents fundamental domains from providing tessellations adapted to arbitrary orbifolds. However, it is easy to see that, when a Fuchsian group~$\Gamma$ and a $\Gamma$-invariant tessellation~$\Tiling$ are given, one can always subdivide~$\Tiling$ and adapt it to~$\Hy^2/\Gamma$.
Condition $(v)$ in Definition~\ref{D:Tiling} may look strange. It is nevertheless important in order to guarantee that the ribbons of the template we will subsequently construct do not intersect (Lemma~\ref{L:IsATemplate}).

It would be natural to add a sixth contraint, namely that no periodic geodesic goes through a vertex (see Definition~\ref{D:Discretisation} below). 
However this is not always possible, in particular for for the triangle groups we will be interested in in Sections~\ref{S:q} and~\ref{S:237}.

We now define, for every tessellation that is adapted to some orbifold, a graph that is \emph{dual} to the tessellation, and on which we will then distort the geodesics of~$\Hy^2$. We have to choose some additional data, namely to pick points on the sides of the tessellation and to choose edges connecting them, but the construction will not depend on these choices, {\it i.e.}, the templates we will eventually associate to two such choices will be isotopic. In the sequel, we use the word ``side'' when referring to the tiles of a tessellation, and ``edge'' when referring to a graph.

Assume that $P$ is a polygon in~$\Hy^2$ with finitely many sides $e_1, \ldots, e_n$, and suppose that $v_1, \ldots, v_n$ are points on $e_1, \ldots, e_n$ respectively. Let $\Gr_P^0$ be a complete unoriented graph with vertices $v_1, \ldots, v_n$, which is immersed in~$P$ so that its edges are orthogonal to the sides of~$P$, and such that two edges intersect at most once (see Figure~\ref{F:ImmersedGraph}). Call~\emph{internal graph} of $P$ associated to~$v_1, \ldots, v_n$ the orientation cover $\Gr_P$ of~$\Gr_P^0$, that is, the oriented graph with twice as many edges as $\Gr_P^0$, each of them corresponding to an edge of~$\Gr_P$ oriented in one of the two possible ways. If~$e_i, e_j$ are two distinct sides of~$\Polg$, we denote by $c_{e_i}^{e_j}$ the oriented edge of~$\Gr_P$ connecting $e_i$ to~$e_j$.

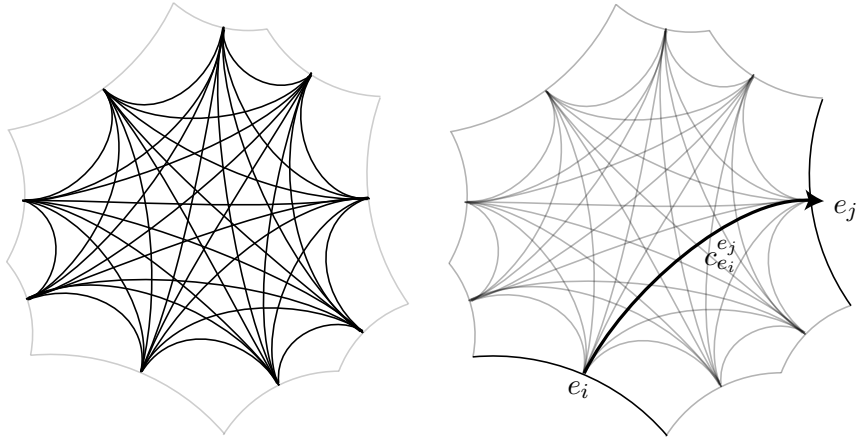
\begin{figure}[hbt]
	\begin{center}
	\begin{picture}(120,55)(0,0)
	\put(0,0){\includegraphics*[scale=.6]{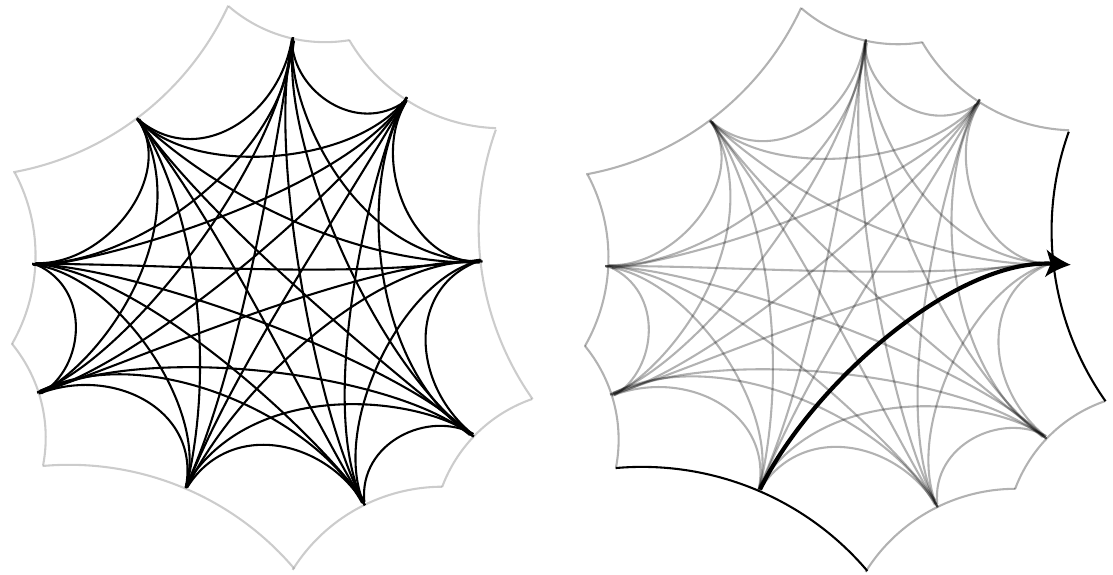}}
	\put(75,6){$e_i$}
	\put(110,30){$e_j$}
	\put(93,23){$c_{e_i}^{e_j}$}
	\end{picture}
	\end{center}
	\caption{\small An internal graph. On the left, the unoriented graph~$\Gr_P^0$. On the right, an oriented edge of~$\Gr_P$.} 
	\label{F:ImmersedGraph}
\end{figure}

Assume now that $\Gamma$ is a Fuchsian group, and that $\Tiling$ is a tessellation adapted to~$\Hy^2/\Gamma$. A set~$\mathcal V$ of points in~$\Hy^2$ is called a \emph{$\Tiling$-marking} if every point in $\mathcal V$ lies on the common boundary between two tiles of~$\Tiling$, every side between two tiles of~$\Tiling$ contains exactly one element of~$\mathcal V$, and $\mathcal V$ is $\Gamma$-invariant.

\begin{defi}
\label{D:Graph}
Assume that $\Gamma$ is a Fuchsian group, that $\Tiling$ is a tessellation of~$\Hy^2$ adapted to~
$\Hy^2/\Gamma$, and that $\mathcal V$ is a $\Tiling$-marking. Suppose that in every tile~$T$ of~$\Tiling$, there is an internal graph~$\Gr_T$ associated to~$\mathcal V$, and that the set of internal graphs is~$\Gamma$-invariant. Then the union~$\Gr_\Tiling$ of all internal graphs~$\Gr_T$ is said to be a graph \emph{dual to} $\Tiling$ and \emph{associated with}~$\mathcal V$.
\end{defi}

It is easy to see that dual graphs exist for every tessellation. In the sequel, we will omit to mention the set~$\mathcal V$  of marked points, since its choice does not influence the construction.
A graph dual is a sort of discretisation of the hyperbolic plane adapted to a given Fuchsian group. If the Fuchsian group is of the first kind, that is, when its limit set is the whole boundary at infinity~$\bord_\infty\Hy^2$, the limit set of any graph dual to any adapted tessellation is also the whole circle~$\bord_\infty\Hy^2$. 
We now introduce a procedure that distorts the geodesics of~$\Hy^2$ to curves included in the dual graph~$\Gr_\Tiling$. 

\begin{defi}
\label{D:Discretisation} 
Assume that~$\Gamma$ is a Fuchsian group of the first kind, and that $\Tiling$ is a tessellation adapted to~$\Hy/\Gamma$. Let $\Gr_\Tiling$ be a graph dual to~$\Tiling$. Then a \emph{discretisation of geodesics on~$\Gr_\Tiling$} is a family consisting, for every geodesics $\underline\gamma$ in~$\Hy^2$, of 

$(i)$ a curve~$\discr{\underline\gamma}{\Tiling}$ embedded in~$\Gr_\Tiling$ that crosses the same tiles of~$\Tiling$ as~$\underline\gamma$ (or a small perturbation of~$\underline\gamma$ in case $\underline\gamma$ goes through a vertex of~$\Gr_\Tiling$), 

$(ii)$ an isotopy $\isot{\underline\gamma}{\Tiling}:[0,1]\times \R\to\Hy^2$ between $\underline\gamma$ and $\discr{\underline\gamma}{\Tiling}$,  {\it i.e.}, a smooth map such that $\isot{\underline\gamma}{\Tiling}^0(t)$ describes $\underline\gamma$ when $t$ describes~$\R$, $\isot{\underline\gamma}{\Tiling}^1(t)$ describes $\discr{\underline\gamma}{\Tiling}$ when $t$ describes~$\R$, and, for every $s$ in~$[0,1]$, the curve $\isot{\underline\gamma}{\Tiling}^0(\R)$ is a smooth embedded curve in~$\Hy$.\\
In addition, the family is supposed to be $\Gamma$-invariant in the sense that, if $g(\underline\gamma)= \underline\gamma'$ holds for some~$g$ in $\Gamma$, then $g(\discr{\underline\gamma}{\Tiling})= \discr{\underline\gamma'}{\Tiling}$ and $g(\isot{\underline\gamma}{\Tiling}^s(t)) = \isot{\underline\gamma'}{\Tiling}^s(t)$ hold for every $(s, t)$ in $[0,1]\times\R$.
\end{defi}


The invariance condition implies in particular that, if $\underline\gamma$ is the lift of periodic geodesics on~$\Hy/\Gamma$, then $\discr{\underline\gamma}{\Tiling}$ projects on a periodic curve on~$\Hy/\Gamma$. More generally, it implies that all choices commute with the covering map~$\Hy\to\Hy/\Gamma$. 

Also, assume that a geodesic~$\underline\gamma$ enters a tile~$\Tile$ by a side~$e_i$ and leaves it by~$e_j$, then its discretisation~$\discr {\underline\gamma} \Tiling$ visits the same tiles as $\underline\gamma$ before and after~$\Tile$. Therefore~$\discr {\underline\gamma} \Tiling$ contains the edge~$c_{e_i}^{e_j}$ of~$\Gr_\Tiling$.

Given a hyperbolic 2-orbifold and an adapted tessellation, the existence of discretisation of geodesics easily follows from the definition.

A discretisation of geodesics contracts many geodesics together. Indeed, if two oriented geodesics $\underline\gamma, \underline\gamma'$ have one end in common, their discretisation will necessarily coincide on some neighbourhood of their positive end. Discretisation will nevertheless be useful for studying $\Gamma$-periodic geodesics.

\subsection{Lifting the discretisation to the unit tangent bundle}
\label{S:Lifting}

Given an orbifold~$\Hy^2/\Gamma$ and some additional data, using the discretisation procedure of Definition~\ref{D:Discretisation}, we distorted the geodesics of~$\Hy^2$ onto some discrete graph. We now lift this procedure to the unit tangent bundle, in view of subsequently constructing the expected template for~$\GF{\Hy/\Gamma}$. 

\begin{defi}
\label{D:Deformation}
Assume that $\Gamma$ is Fuchsian group of the first kind, that $\Tiling$ is an adapted tessellation of~$\Hy^2$, that $\Gr_\Tiling$ is a graph dual to~$\Tiling$ and that a discretisation of geodesics on $\Gr_\Tiling$ has been chosen. Then the associated \emph{tearing map} of the unit tangent bundle is the map $\deform\Tiling$ from $[0,1]\times \un\Hy^2$ to $\un\Hy^2$ defined as follows. For~$(p,v)$ in~$\un\Hy^2$, let $\underline\gamma$ denote the geodesics containing~$p$ and oriented by~$v$, let $\isot\Tiling{\underline\gamma}$ denote the associated isotopy, let $t_p$ be the real parameter such that $p=\isot\Tiling{\underline\gamma}^0(t_p)$, and let $\gamma_+$ be the positive extremity of~$\gamma$ in~$\bord\Hy^2$. Then $\deform\Tiling^s(p,v)$ is defined to be the unique unit tangent vector based at $\isot\Tiling{\underline\gamma}^s(t_p)$ and pointing in the direction of~${\underline\gamma}_+$.
\end{defi}

Note that a tearing map is not continuous. Indeed, since the graph~$\Gr_\Tiling$ is discrete, there are pairs of arbitrarily close tangent vectors that are mapped to different edges of~$\Gr_\Tiling$. Also, a tearing map can be injective when the time~$s$ is close to~$0$, but its time~$1$ map may, for instance, collapse some horocyle. For these two reasons, a tearing map is not an isotopy. Nevertheless, if we restrict to $\Gamma$-periodic geodesics, that is, to geodesics which are $g$-invariant for some $g$ in~$\Gamma$, we have

\begin{lemma}
\label{L:GeodDontCross}
In the above context, the restriction of $\deform\Tiling$ to vectors tangent to $\Gamma$-periodic geodesics is an isotopy. 
\end{lemma}

\begin{proof}
Suppose that $\deform\Tiling^s(p_1, v_1)=\deform\Tiling^s(p_2, v_2)$ holds for some $s$ in~$[0,1]$. Let $\underline\gamma_1, \underline\gamma_2$ denote the two geodesics tangent to $v_1$ and $v_2$ at $p_1$ and $p_2$ respectively. As the vectors  $\deform\Tiling^s(p_1, v_1)$ and $\deform\Tiling^s(p_2, v_2)$ point in the directions $(\underline\gamma_1)_+$ and $(\underline\gamma_2)_+$, their equality implies $(\underline\gamma_1)_+=(\underline\gamma_2)_+$. Therefore $\underline\gamma_1$ and $\underline\gamma_2$ get closer with an exponential rate. By hypothesis, both are $\Gamma$-periodic, hence compact in~$\Hy/\Gamma$. Therefore they coincide. By definition of~$\deform\Tiling$, the equality $\deform\Tiling^s(p_1, v_1)=\deform\Tiling^s(p_2, v_2)$ implies $\isot{\underline\gamma_1}\Tiling^s(p_1)=\isot{\underline\gamma_1}\Tiling^s(p_2)$. Since $\isot{\underline\gamma_1}\Tiling$ is an isotopy, we deduce $p_1=p_2$. Finally, since the vectors $v_1, v_2$ point in the same direction, they also coincide.
\end{proof}

\subsection{Multitemplates for the geodesic flow}
\label{S:MultiTemplate}

We have now constructed a deformation of the unit tangent bundle that preserves the topology of periodic geodesics. Our task is now to determine the image of the deformation. In particular, we want to show that it lies inside some specific object that we call a \emph{multitemplate}.

\begin{defi}
\label{D:MultiTemplate}
(See Figure~\ref{F:Template}.) Assume that $M$ is a 3-manifold. A \emph{ribbon} in~$M$ is an embedded surface in~$M$ diffeomorphic to~$[0,1]^2$  equipped with the horizontal flow generated by~$\frac\bord{\bord x}$. If $\Rib$ is a ribbon, we denote by~$X_\Rib$ the vector field on it.

A \emph{multitemplate}~$S$ in~$M$ is a branched surface equipped with a vector field~$X_S$, that is locally a union of finitely many ribbons, and is such that

$(i)$ two distinct ribbons $\Rib_1, \Rib_2$ of~$S$ can only intersect along their vertical edges, which are then called \emph{branching segments},

$(ii)$ at every point on a branching segment, there are finitely many ribbons, and the associated vector fields all coincide,

$(iii)$ for every ribbon~$\Rib$ of~$S$, the vector field $X_\Rib$ coincide with~$X_S$ on~$\Rib$.

An \emph{orbit} of a multitemplate~$S$ is a complete immersion of the real line~$\R$ is~$S$ that is everywhere tangent to~$X_S$.
\end{defi}

\begin{figure}[hbt]
	\begin{center}
	\includegraphics*[scale=.8]{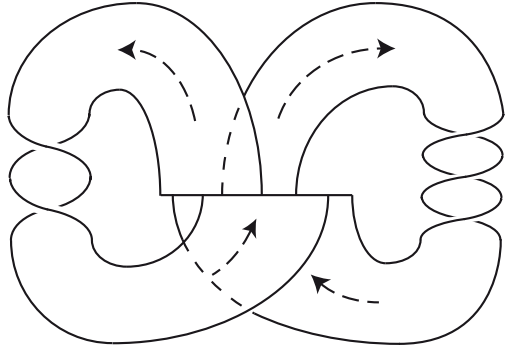}
	\end{center}
	\caption{\small A multitemplate in~$\R^3$. Along the branching segment, there are, from left to right, 1, then 2, then 1, and then 0 escaping ribbons.} 
	\label{F:Template}
\end{figure}

The difference with the usual notion of a template~\cite{BW, GHS} is that there is no uniquely defined semi-flow, but a multiflow. Indeed, at a point of a branching segment, there may be several escaping ribbons, and therefore several possible futures. If there were at most one escaping ribbon at every branching point, we would speak of a template. This will only happen in our construction when the starting tessellation consists of ideal polygons. Note also that there may be points that are visited by no orbit of the multitemplate, as for instance the points on the right of the branching segment on Figure~\ref{F:Template}.

Let us go back to the construction. In order to specify the ribbons making the expected multitemplate, we describe the set of directions at infinity that are pointed in by elements in the image of a deformation~$\deform\Tiling^1$.

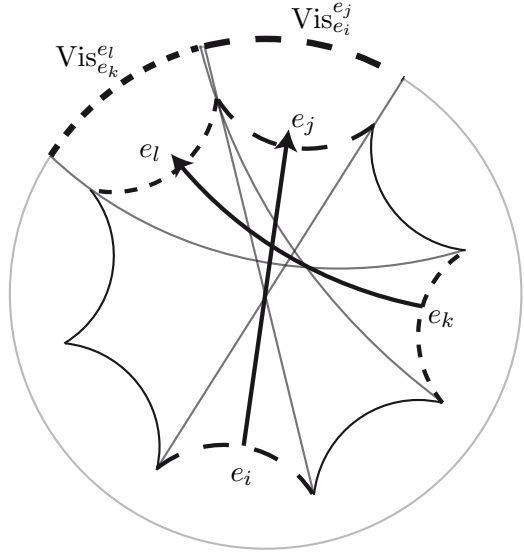
\begin{figure}[hbt]
	\begin{center}
	\begin{picture}(75,70)(0,0)
	\put(0,0){\includegraphics*[scale=.8]{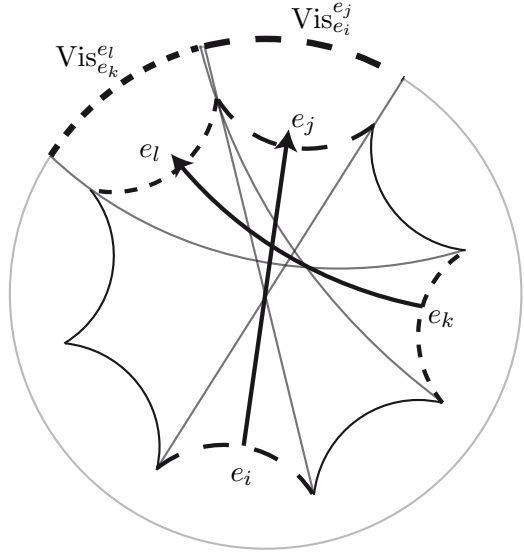}}
	\put(30,10){$e_i$}
	\put(38,57){$e_j$}
	\put(18,53){$e_l$}
	\put(56,31){$e_k$}
	\put(7,65){$\Vis{k}{l}$}
	\put(38,70){$\Vis{i}{j}$}
	\end{picture}
	\end{center}
	\caption{\small The visual intervals $\Vis{i}{j}, \Vis{k}{l}$ associated to two intersecting edges of~$\Gr_\Tiling$. Since $\Vis{i}{j}$ and $\Vis{k}{l}$  are disjoint, the associated ribbons~$\Rib_{e_i}^{e_j},\Rib_{e_k}^{e_l}$ do not intersect.}
	\label{F:VisInt}
\end{figure}

\begin{defi}
\label{D:TheTemplate}
(See Figure~\ref{F:VisInt}.) 
Assume that $\Gamma$ is Fuchsian group of the first kind, that $\Tiling$ is an adapted tessellation of~$\Hy^2$, and that $\Gr_\Tiling$ is a graph dual to~$\Tiling$. Let $\Tile$ be a tile of~$\Tiling$, and $e_i, e_j$ be two sides of~$T$. Then the~\emph{visual interval} associated to~$(e_i, e_j)$ is the interval consisting of the positive extremities of geodesics connecting a point of~$e_i$ to a point~$e_j$ in~$\bord_\infty\Hy^2$. We denote it by~$\Vis{i}{j}$. The associated \emph{product-ribbon} is the product of the oriented edge $c_{e_i}^{e_j}$ connecting $e_i$ to $e_j$ in~$\Gr_\Tiling$ by the interval $\Vis{i}{j}$ in~$\un\Hy^2$, seen as the product~$\Hy^2\times\bord_\infty\Hy^2$. We denote it by~$\Rib_{e_i}^{e_j}$. It is equipped with the horizontal vector field whose flow goes along the curves $c_{e_i}^{e_j}\times\{*\}$ at speed~1.

In the above context, we denote by~$\prot_\Tiling$ the union in~$\un\Hy^2$ of the product-ribbons associated with all oriented edges of~$\Gr_\Tiling$. Its quotient under the action of~$\Gamma$ is denoted by~$\temp_{\Gamma,\Tiling}$.
\end{defi}

\begin{lemma}
\label{L:IsATemplate}
In the context of Definition~\ref{D:TheTemplate}, $\prot_\Tiling$ is a multitemplate in~$\un\Hy^2$. 
\end{lemma}

\begin{proof}
By definition, the set~$\prot_\Tiling$ is the union of several ribbons, which are in one-to-one correspondence with the oriented edges of the graph~$\Gr_\Tiling$. Let~$\Tile_1, \Tile_2$  be two adjacent tiles of~$\Tiling$. Call $e$ the common side of~$\Tile_1$ and $\Tile_2$, and let $p$ be a vertex of~$\Gr_\Tiling$ lying on~$e$. Since the tiles of~$\Tiling$ are supposed to have finitely many sides, there are finitely many ribbons that intersect the fiber~$\un\{p\}$ of~$p$. Since all the edges of~$\Gr_\Tiling$ with~$p$ as an extremity are orthogonal to~$e$, the associated product-ribbons all are tangent in~$\un\{p\}$. Now the product-ribbons that have an extremity in~$\un\{p\}$ decompose into four classes depending on whether they lie above~$\Tile_1$ or above~$\Tile_2$, and on whether they correspond to edges of~$\Gr_\Tiling$ starting at~$p$ or ending at~$p$. Let $(p,v)$ be a tangent vector based at~$p$. Suppose that $v$ points into~$\Tile_1$. Then the only ribbons that may contain~$(p,v)$ are those coming from geodesics with a positive extremity on the same side of~$e$ as $\Tile_2$. In this case, the vector field on any such ribbon at~$(p,v)$ is the unit vector orthogonal to~$e$, and pointing into~$\Tile_1$. Therefore the vector field on all such ribbons coincide. Similarly, if $v$ points into~$\Tile_2$, the vector fields of all ribbons that contain~$(p,v)$ are equal at~$(p,v)$ with the unit vector orthogonal to~$e$, and pointing into~$\Tile_2$.

There remains to show that product-ribbons are disjoint outside the fiber of the vertices of~$\Gr_\Tiling$. Since product-ribbons are in the fibers of edges of~$\Gr_\Tiling$, this is equivalent to showing that, if two edges~$c_{e_i}^{e_j},c_{e_k}^{e_l}$ of~$\Gr_\Tiling$ intersect inside a tile, say~$\Tile$, of~$\Tiling$, then the associated visual intervals $\Vis{i}{j}$ and $\Vis{k}{l}$ are disjoint. Indeed, in this situation, at the expense of possibly exchanging the indices and performing a symmetry, we can suppose that the edges $e_i, e_k, e_j, e_l$ are cyclically ordered. Let $\underline\gamma_{i,j}^l$ be the geodesics joining the right extremity of~$e_i$ to the left extremity of~$e_j$, and $\underline\gamma_{i,j}^r$ be the geodesics connecting the left extremity of~$e_i$ to the right extremity of~$e_j$. Define $\underline\gamma_{k,l}^l$ and $\underline\gamma_{k,l}^r$ similarly. Then $\Vis{i}{j}$ is the interval $[(\underline\gamma_{i,j}^r)_+, (\underline\gamma_{i,j}^l)_+]$, and $\Vis{k}{l}$ is $[(\underline\gamma_{k,l}^r)_+, (\underline\gamma_{k,l}^l)_+]$ (see Figure~\ref{F:VisInt}). The geodesics $\underline\gamma_{i,j}^l$ and $\gamma_{k,l}^r$ intersect inside~$\Tile$, so that $(\gamma_{k,l}^r)_+$ lies on the left of $(\underline\gamma_{i,j}^l)_+$ on~$\bord_\infty\Hy^2$. Therefore $\Vis{i}{j}$ and $\Vis{k}{l}$ are disjoint.
\end{proof}

\begin{figure}[hbt]
	\begin{center}
	\includegraphics*[scale=.75]{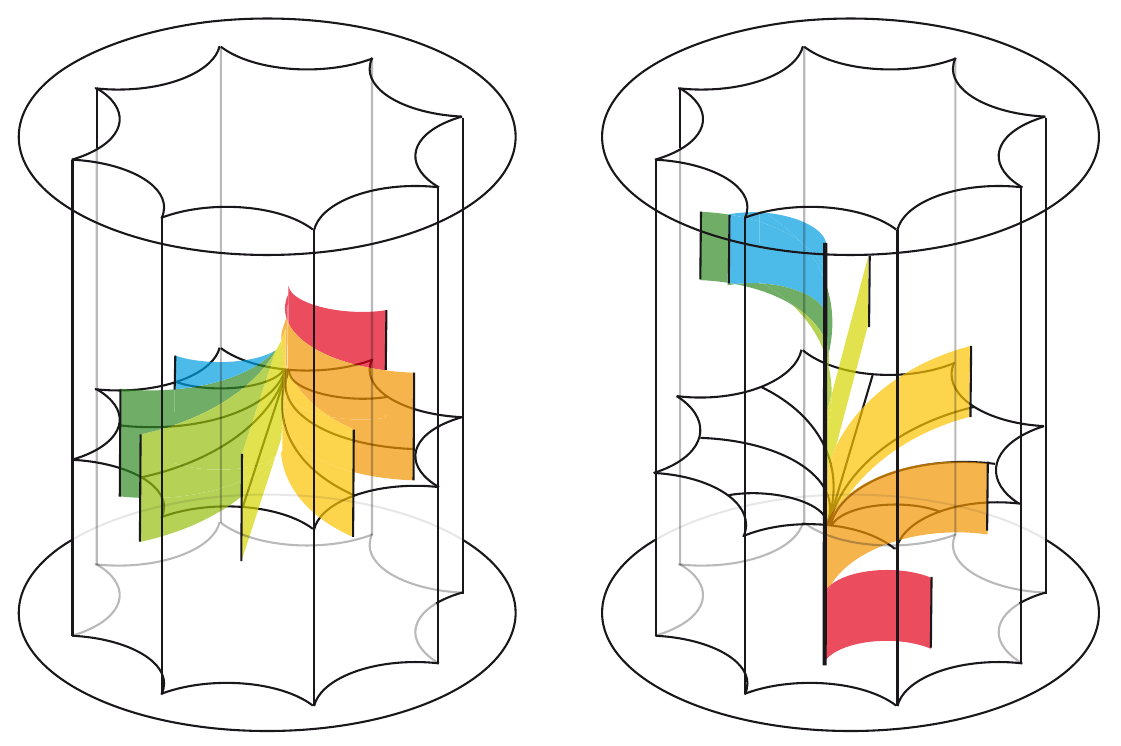}
	\end{center}
	\caption{\small Some ribbons of a template above a tile that is not an ideal polygon. Both incoming ribbons (on the left) and outgoing ribbons (on the right) overlap, since the associated visual intervals overlap.} 
	\label{F:TemplateNonIdeal}
\end{figure}

Assume now that the tiles of $\Tiling$ all are ideal polygons. Let $p$ be a vertex of~$\Gr_\Tiling$. Then all visual intervals associated with the edges of~$\Gr_\Tiling$ ending at~$p$ are disjoint. Hence, for every tangent vector~$v$ in the fiber~$\un\{p\}$, there is at most one escaping ribbon. Therefore~$\prot_\Tiling$ is a template, and so does $\temp_{\Gamma, \Tiling}$.

\begin{figure}[htb]
	\begin{center}
 	\includegraphics*[scale=0.8]{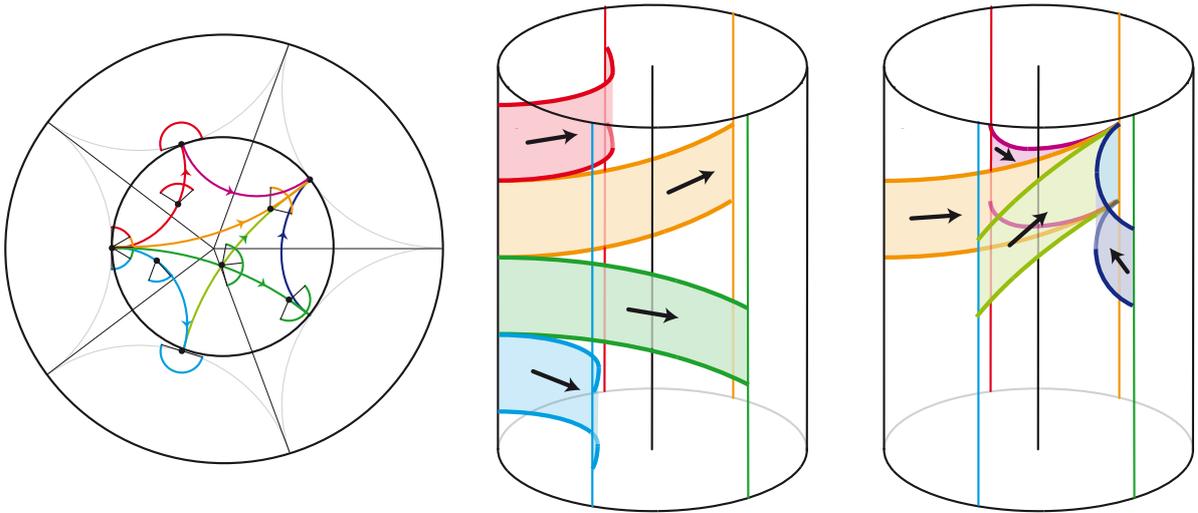}
	\end{center}
	\caption{\small Some ribbons of a template above a tile that is an ideal polygon. The ribbons emerging from the same side do not overlap, since the associated visual intervals are disjoint.} 
	\label{F:TemplateIdeal}
\end{figure}

In the above context, since all steps in the construction of~$\prot_\Tiling$ are $\Gamma$-invariant, the quotient~$\temp_{\Gamma,\Tiling}$ is also a multitemplate. We can now state the main result of this section.

\begin{theo}
\label{P:Template}
Assume that $\Gamma$ is Fuchsian group of the first kind, $\Tiling$ is an adapted tessellation of~$\Hy^2$, $\Gr_\Tiling$ is a graph dual to~$\Tiling$, and a discretisation of geodesics on $\Gr_\Tiling$ has been chosen. Let $\deform\Tiling$ denote the associated tearing map of~$\un\Hy^2$, and $\temp_{\Gamma,\Tiling}$ denote the associated multitemplate in~$\un\Hy^2/\Gamma$. Then the action of~$\deform\Tiling$ on~$\un\Hy/\Gamma$ induces an isotopy of every collection of periodic orbits of the geodesic flow~$\GF{\Hy^2/\Gamma}$ onto a collection of periodic orbits of~$\temp_\Tiling$. Moreover, if all tiles of~$\Tiling$ are ideal polygons, then~$\temp_{\Gamma,\Tiling}$ is a template, and $\deform\Tiling$ is a one-to-one correspondence between the periodic orbits of~$\GF{\Hy^2/\Gamma}$ and the periodic orbits of~$\temp_{\Gamma,\Tiling}$ that do not lie in the boundary of~$\temp_{\Gamma,\Tiling}$.
\end{theo}

\begin{proof}
Let $\underline\gamma$ be a geodesic of $\Hy^2$. Then its discretisation~$\discr {\underline\gamma}\Tiling$ is included in the graph~$\Gr_\Tiling$. Let $p$ be point on~$\underline\gamma$ and $v$ be the tangent vector to $\underline\gamma$ at $p$. Then the vector $\deform{\Tiling}^1((p,v))$ lies in a fiber over~$\Gr_\Tiling$. By construction, for every edge~$c_{e_i}^{e_j}$ of~$\Gr_\Tiling$ contained in~$\discr{\underline\gamma}\Tiling$, the direction~$\underline\gamma_+$ belongs to the visual interval~$\Vis i j$, so that the part of the curve $\deform{\Tiling}^1(\underline\gamma)$ above~$c_{e_i}^{e_j}$ lies in the ribbon~$\Rib_{e_i}^{e_j}$ and points towards~$\gamma_+$. Therefore $\deform{\Tiling}^1(\gamma)$ sits in the multitemplate~$\prot_{\Tiling}$, and is everywhere tangent to the vector field~$X_{\prot_{\Tiling}}$. By Lemma~\ref{L:GeodDontCross}, the restriction of $\deform\Tiling$ to $\Gamma$-invariant geodesics is an isotopy. Since everything commutes with the action of~$\Gamma$, we can mod out by~$\Gamma$, so that the projection of~$\deform\Tiling$ realizes an isotopy between the periodic orbits of the geodesic flow~$\GF{\Hy^2/\Gamma}$ and their images.

Suppose now that all tiles of~$\Tiling$ are ideal polygons. 
Let $\gamma(t)$ be an orbit of $X_{\prot_{\Tiling}}$ not lying in the boundary of~$\prot_{\Tiling}$ and $g$-invariant for some~$g$ in~$\Gamma$. Let $\gamma_0$ be its projection on~$\Hy^2$. It is a $g$-invariant curve in~$\Gr_\Tiling$. Since all tiles of~$\Tiling$ are ideal polygons, $\gamma_0$ is a simple curve. The assumption that $\gamma(t)$ does not lie in the boundary of~$\prot_{\Tiling}$ implies that the two extremities of $\gamma_0$ are distinct. Let $\gamma_1$ be the unique geodesics in~$\Hy$ connecting $(\gamma_0)_-$ to $(\gamma_0)_+$. Then $\gamma_1$ is also~$g$-invariant. It turns out that $\gamma_0$ is then the discretisation of~$\gamma_1$. Therefore, $\deform\Tiling$ maps the vectors that are tangent to~$\gamma_1$ to vectors that are tangent to~$\gamma(t)$.
\end{proof}

\begin{figure}[hbt]
	\begin{center}
	\includegraphics*[width=.8\textwidth]{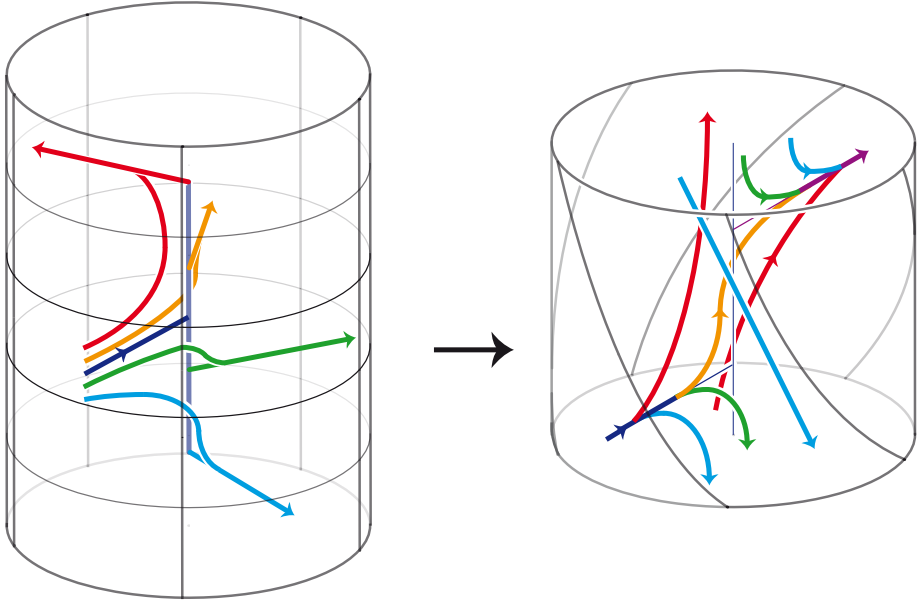}
	\end{center}
	\caption{\small 
	A simplified version of some ribbons of a template above some tile with an order~5 symmetry. On the left we displayed only 4 of the 20 ribbons, and only one orbit per ribbon; the other ribbons are obtained by iterating a screw-motion (remember Figure~\ref{F:UTBundle}). 
	On the right, the quotient of the unit tangent bundle by the order~5 symmetry, in the storey model. This is what we call an~\emph{elevator}. For example the red ribbon on the left goes one floor up (that is, it goes in the next fundamental domain for the storey model), so we see it crossing once the horizontal disc in the story model. 
	}
	\label{F:TemplateQuotient}
\end{figure}

To conclude this section, we introduce some terminology that will be useful when Theorem~\ref{P:Template} is applied in the sequel. Assume that $\Tiling$ is a tessellation of the hyperbolic plane, and that~$\Gr_\Tiling$ is an associated graph. Let $\Tile_0$ be a tile of~$\Tiling$ with $n$ sides. Then the part of the template~$\prot_\Tiling$ that lies above~$\Tile_0$, that is, the intersection of~$\prot_\Tiling$ with $\un\Tile_0$, consists of $n(n-1)$~ribbons, as depicted on Figure~\ref{F:TemplateIdeal} and Figure~\ref{F:TemplateNonIdeal}. In particular, there are $n$ branching segments on which the template flow enters the solid torus~$\un\Tile_0$, which we call~\emph{incoming segments}, and $n$ branching segments where the template flow escapes~$\un\Tile_0$, which we call \emph{outgoing segments}. We call such a part of a template a~\emph{switch tower}. If~$\Tile_0$ has a trivial stabilizor in~$\Gamma$, then the part of the template~$\temp_{\Gamma, \Tiling}$ above the quotient of~$\Tile_0$ by $\Gamma$ is also a complete star. 

Suppose now that~$\Tile_0$ has a non-trivial stabilizor, say~$\Gamma_{\Tile_0}$, in~$\Gamma$. Then the part of $\temp_{\Gamma, \Tiling}$ above~$\Tile_0/\Gamma$ is the quotient of~$\prot_\Tiling$ by~$\Gamma_{\Tile_0}$. If~$\Gamma_{\Tile_0}$ has order~$d$, then the part of the template has $n(n-1)/d$ ribbons. In particular, if $\Tile_0$ is a regular $n$-gon and if its stabilizor~$\Gamma_{\Tile_0}$ is of order~$n$, then there are only~$n-1$ ribbons in the quotient, all of them joining a unique incoming segment to a unique outgoing segment, see Figure~\ref{F:TemplateQuotient}. We call such a part of a template an~\emph{elevator}.

\section{Geodesic flow for the orbifolds of type~$(p, q, \infty)$}
\label{S:q}

We now turn to the linking properties of orbits associated with hyperbolic orbifolds of type~$(2, q, \infty)$ with $q\ge 3$. The goal of this section is to prove the first case of Theorem~A, that is, to prove that the linking number of every two orbits of~$\GF{\Orb_{2,q,\infty}}$ is negative. 

The idea is to apply the construction of Section~\ref{S:Template}, thus obtaining a template that describes the topology of the periodic orbits of~$\GF{\Orb_{2,q,\infty}}$ (\S\,\ref{S:Tpq}), and then to compute the linking number of a pair of periodic orbits. Actually, we do more and first compactify the unit tangent bundle into a lens space (\S\,\ref{S:Compactification}). As a lens space is a rational homology sphere, the linking number is defined for every pair of links. 
We then show that the linking number of every pair of periodic orbits of the template is negative (Proposition~\ref{P:2q} and case $(a)$ of Theorem~A). By the way, we consider a slightly more general context and construct a template for every orbifold of type~$(p,q,\infty)$ with~$p\ge 2$. The advantage of this approach is to also provide a precise formula for the linking number of a periodic orbit of~$\GF{\Orbpq}$ with the fiber of the cusp in the unit tangent bundle, that is, with the link that has been added for the compactification (Proposition~\ref{P:LkInfty}). 

\subsection{A template for~$\GF{\Orbpq}$}
\label{S:Tpq}

Here we introduce orbifolds of type~$(p, q, \infty)$, choose adapted tessellations of the hyperbolic plane, and describe the associated templates. 
As we will recall, the space~$\un\Orbpq$ is obtained by gluing two solid tori along their boundary, and what we will do is two describe a template that lies in a neighborhood of the gluing torus. 
In the case~$p=2, q=3$, we recover Ghys' template for the geodesic flow on the modular surface~\cite{GhysMadrid}, and, in the more general case~$p=2, q\ge 3$, we recover Pinsky's template~\cite{Tali}. 

Until the end of Section~\ref{S:q}, we assume that $p, q$ are fixed integers satisfying  $p\ge 2$ and~$q \ge3$. 
Since~$1/p+1/q<1$ holds, there exists a hyperbolic triangle~$PQZ$ in~$\Hy^2$, with the two vertices~$P,Q$ inside~$\Hy^2$ with respective angles~$2\pi/p$ and $2\pi/q$, and the vertex~$Z$ lying on~$\bord \Hy^2$. 
For convenience, we also suppose $P,Q,Z$ trigonometrically ordered. Let $\Gamma_{p, q}^*$ be the group generated by the symmetries around the sides of~$PQZ$, and let $\Gamma_{p, q}$ be its index~2 subgroup consisting of orientation preserving isometries, often called the \emph{Hecke triangular group}. The group~$\Gamma_{p, q}$ acts properly and discontinuously on~$\Hy^2$. The action is not free since, for example, $P$ and $Q$ have stabilizors of order $p$ and~$q$ respectively. The quotient~$\Hy^2/\Gamma_{p,q}$ is then an orbifold, with two singular points of order~$p$ and~$q$, and one cusp. We call it~$\Orbpq$.

\begin{figure}[hbt]
	\begin{center}
 	\includegraphics*[scale=.7]{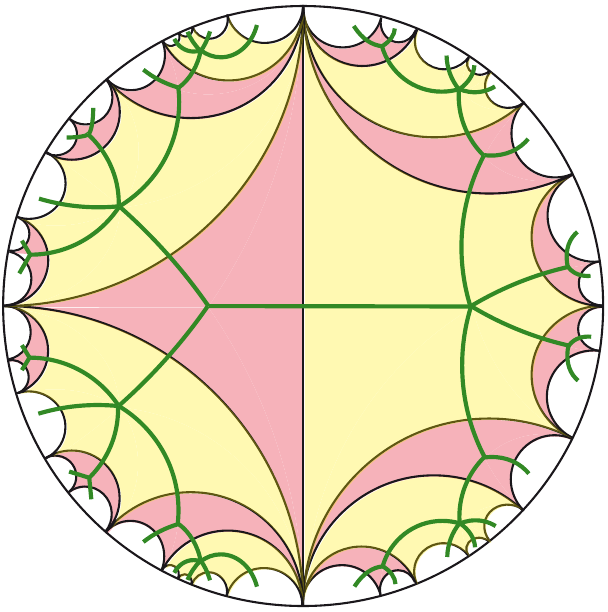}
	\end{center}
	\caption{\small The tessellation of~$\Hy^2$ by copies of the ideal polygons $\Delta_P$ et $\Delta_Q$, here with $p=3$ and $q=5$.} 
	\label{F:PTessel}
\end{figure}

For $k=1, \dots, {q-1}$, let $Z^k_q$ be the image of~$Z$ by a rotation of center~$Q$ and of angle~$2k\pi/q$ (see Figure~\ref{F:Spq8}). Then $Z, Z^1_q, \ldots, Z^{q-1}_q$ are the vertices of an ideal $q$-gon, say $\Delta_Q$. Let $\Gamma_Q$ be the stabilizor of~$Q$ in~$\Gamma_{p, q}$. Then $\Delta_Q$ is invariant under the action of~$\Gamma_Q$. 

Assume now $p>2$. Define similarly the points $Z^1_p, \ldots, Z^{p-1}_p$ on~$\bord_\infty\Hy^2$ and the polygon~$\Delta_P$. Note that the points $Z^1_q$ and $Z^{p-1}_p$ coincide. Call~$\Side$ the geodesics~$ZZ^1_q$. Then the polygons~$\Delta_P, \Delta_Q$ lie on different sides of~$\Side$, hence they are distinct. One easily sees that the images of~$\Delta_P$ and $\Delta_Q$ under $\Gamma_{p,q}$ cover the whole hyperbolic plane, and therefore form a tessellation (Figure~\ref{F:PTessel}). We denote it by~$\Tiling_{P,Q}$. The sides of the tiles of~$\Tiling_{P,Q}$ exactly are the images of~$e$ under $\Gamma_{p,q}$. Since all tiles are ideal polygons, no two sides in the tessellation intersect inside~$\Hy^2$. Also, every tile is a copy of either~$\Delta_P$ or~$\Delta_Q$, and therefore contains exactly one singular point in its interior. 

The unit tangent bundles to~$\Delta_Q/\Gamma_Q$ and $\Delta_P/\Gamma_P$ are both non-compact solid tori (remember Figure~\ref{F:UTBundle}). The unit tangent bundle $\un\Orbpq$ is then obtained by identifying the tangent vectors that constitute the boundaries of the unit tangent bundles to~$\Delta_Q/\Gamma_Q$ and~$\Delta_P/\Gamma_P$. These are exactly the images in the quotient of the tangent vectors based on~$\Side$, that is, the image of~$\un e$ in the orbifold~$\Orbpq$.

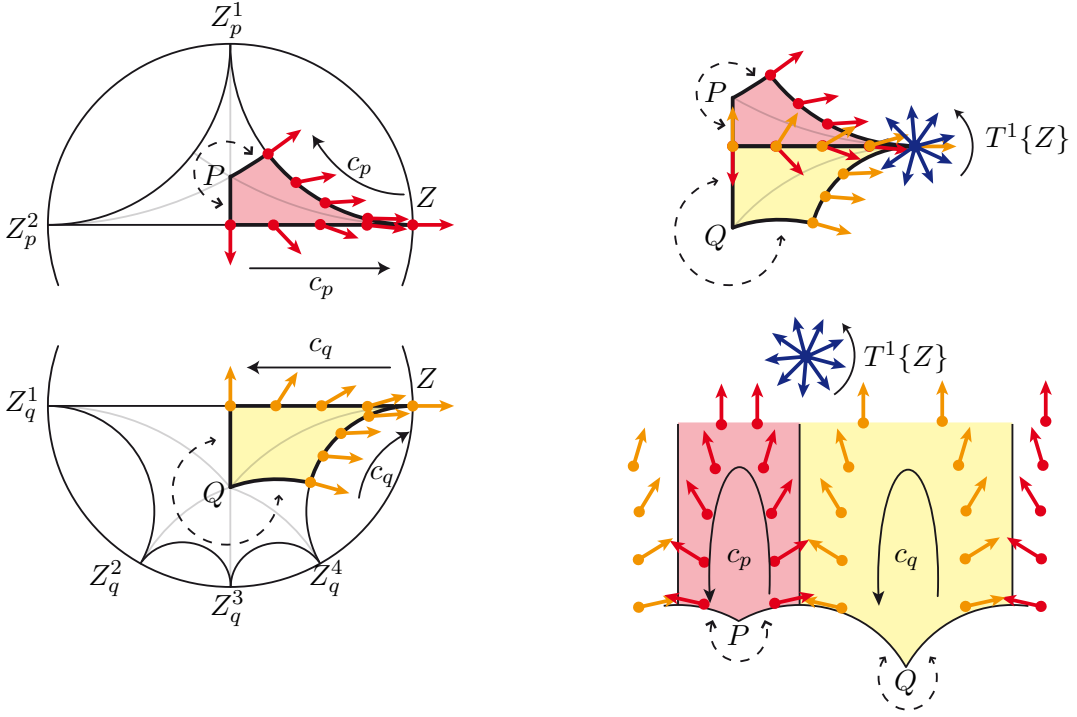
\begin{figure}[hbt]
	\begin{center}
	\begin{picture}(140,92)(0,0)
	\put(5,-1){\includegraphics*[scale=.7]{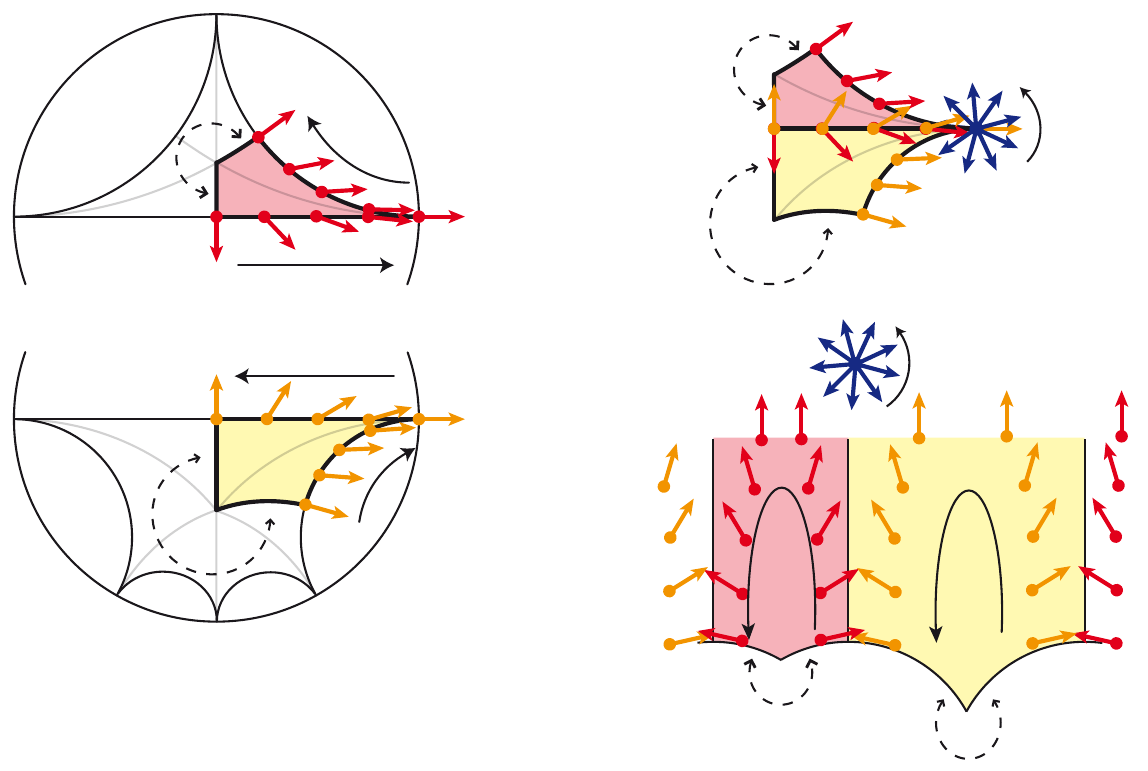}}
	\put(27,70){$P$}
	\put(27,28){$Q$}
	\put(1, 63){$Z^2_p$}
	\put(55,67){$Z$}
	\put(28,91){$Z^1_p$}
	\put(55,43){$Z$}
	\put(1, 40){$Z^1_q$}
	\put(12,17){$Z^2_q$}
	\put(28,13){$Z^3_q$}
	\put(41,17){$Z^4_q$}
	\put(41,56){$c_p$}
	\put(46,72){$c_p$}
	\put(41,48){$c_q$}
	\put(49,31){$c_q$}
	\put(93,81){$P$}
	\put(93,62){$Q$}
	\put(96,9){$P$}
	\put(118,3){$Q$}
	\put(96,20){$c_p$}
	\put(118,20){$c_q$}
	\put(130,75){$\un\{Z\}$}
	\put(114,46){$\un\{Z\}$}
	\end{picture}
	\end{center}
	\caption{\small On the left, the tiles~$\Delta_{P}$ and $\Delta_Q$, with $p=3$ et $q=5$. On the right a fundamental domain for the action of~$\Gamma_{p,q}$ on~$\Hy^2$. The curves $c_p$, $c_q$ and $\un\{Z\}$ 
	are also depicted. They lie on the common boundary~$\BoundTorus$ of the two solid tori $\barre{\un\Delta_P}/\Gamma_P$ and~$\barre{\un\Delta_Q}/\Gamma_Q$.} 
	\label{F:Spq8}
\end{figure}


We still assume $p>2$. Let $M$ be the intersection of the segment~$PQ$ with~$\Side$, let $\BranchPQ$ be the set of all tangent vectors at~$M$ pointing into~$\Delta_Q$, and let $\BranchQP$ be the set of all tangent vectors at~$M$ pointing into~$\Delta_P$. Then the template~$\Tpq$ given by Proposition~\ref{P:Template} consists of two parts: one elevator (Figure~\ref{F:TemplateQuotient}) sitting inside the solid torus~$\un\Delta_Q/\Gamma_Q$ with~$q-1$ ribbons, say~$\Rib_q^1, \ldots, \Rib_q^{q-1}$, all connecting~$\BranchPQ$ to~$\BranchQP$, and one elevator sitting inside~$\un\Delta_P/\Gamma_P$ with~$p-1$ ribbons, say~$\Rib_p^1, \ldots, \Rib_p^{p-1}$, all connecting~$\BranchQP$ to~$\BranchPQ$.

If $p=2$, then, with the above definition, $\Delta_P$ is a bigon with an empty interior. In this case, the tessellation~$\Tiling_{P,Q}$ consists of copies of~$\Delta_Q$ only. In the quotient of~$\Hy^2$ by~$\Gamma_{p,q}$, the edges of~$\Delta_Q$ are quotiented by order~2 rotations, so that the unit tangent bundle of~$\Orbdq$ is obtained by considering the solid torus~$\un\Delta_Q/\Gamma_Q$, and identifying pairs of points on the boundary with the order~$2$ rotation around~$P$.

For convenience (especially in view of the pictures in Section~\ref{S:LkPQ}), we slightly modify the tessellation in this case. We consider a tile~$\Delta'_P$ which is the~$\epsilon$-neighbourhood of~$\Delta_P$, and we change~$\Delta_Q$ accordingly. If~$\epsilon$ is positive, the sides of the tiles are no longer geodesic, so that the construction of Section~\ref{S:Template} does not apply. We rather see~$\epsilon$ as infinitely small. The unit tangent bundle of~$\Orbdq$ is then the union of~$\un\Delta'_Q/\Gamma_Q$, which is infinitesimally smaller than~$\un\Delta_Q/\Gamma_Q$, with the infinitesimally small solid torus~$\un\Delta'_P/\Gamma_P$. The role of the latter solid torus is to identify pairs of points on the boundary of~$\un\Delta'_Q/\Gamma_Q$. 

Mimicking the case $p>2$, we denote by~$M$ the point on the segment~$[PQ]$ that is at distance~$\epsilon$ from~$P$, by~$\BranchPQ$ the set of all tangent vectors at~$M$ pointing into~$\Delta'_Q$ and by~$\BranchQP$ the set of all tangent vectors at~$M$ pointing into~$\Delta'_P$. Then~$\Tdq$ consists of one elevator in~$\un\Delta'_Q/\Gamma_Q$ with~$q-1$ ribbons, say~$\Rib_q^1, \ldots, \Rib_q^{q-1}$, all connecting~$\BranchPQ$ to~$\BranchQP$, and one ribbon in~$\un\Delta'_P/\Gamma_P$ connecting~$\BranchQP$ to~$\BranchPQ$. 

In the sequel, it will be imported to visualize how the ribbons~$\Rib_p^1, \ldots, \Rib_p^{p-1}$ can be distorted on the torus~$\bord\un\Delta_P/\Gamma_P = \bord\un\Delta_Q/\Gamma_Q$ (and similarly for~$\Rib_q^1, \ldots, \Rib_q^{q-1}$). 
Figure~\ref{F:TemplateDef} shows two ways of deforming every such ribbon by pushing it to the left or to the right. 

\begin{figure}[hbt]
	\begin{center}
 	\begin{picture}(130,90)(0,0)
	\put(0,0){\includegraphics*[width=.45\textwidth]{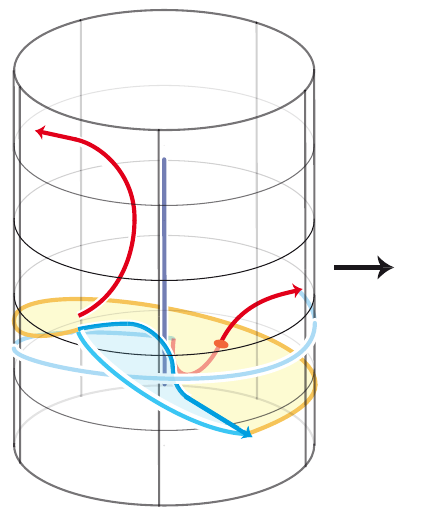}}
	\put(70,15){\includegraphics*[width=.35\textwidth]{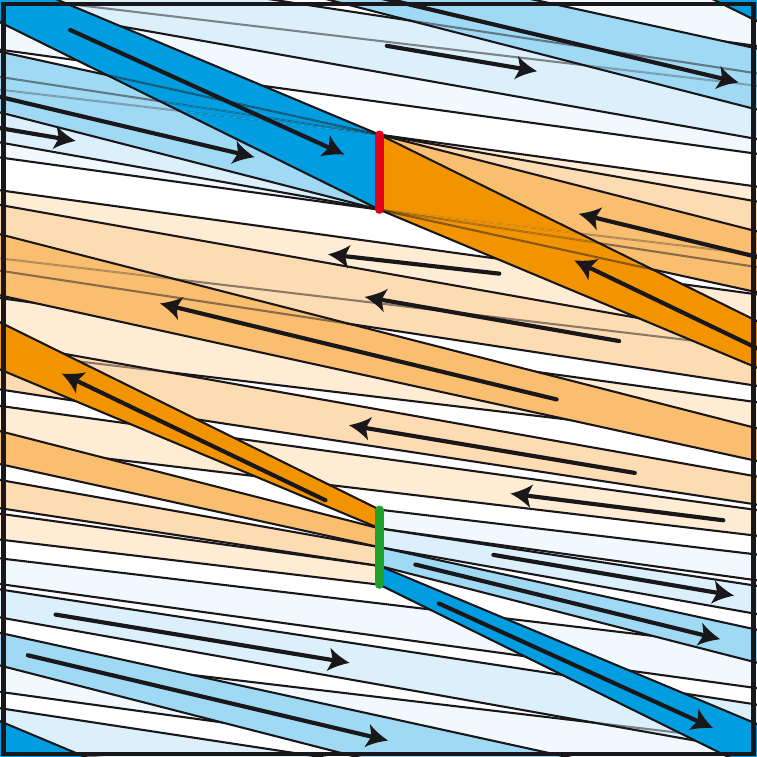}}
	\put(27,10){$\Rib_5^{4,r}$}
	\put(-8,32){$\Rib_5^{4,l}$}
	\put(-8,26){$\Rib_5^{1,r}$}
	\put(14,47){$\Rib_5^{1}$}
	\put(37,37){$\Rib_5^{1}$}
	\put(125,15){$\Rib_5^{4,r}$}
	\put(126.5,27){$\Rib_5^{1,r}$}
	\put(125,31){$\Rib_5^{4,l}$}
	\put(125,44){$\Rib_5^{1,l}$}
	\end{picture}
	\end{center}
	\caption{\small The two possible deformations $\Rib_p^{i,l}$ and $\Rib_p^{i,r}$ of a ribbon of type~$\Rib_p^i$ on~$\bord\un\Delta_P/\Gamma_P$. 
	On the left, with $p=5$, the ribbon~$\Rib_p^{p-1}$ (in blue) can be pushed in~$\un\Delta_P$ either to the right (in light blue) or to the left (in orange). Another ribbon (here $\Rib_p^{1}$ in red) and one of its images under the $\Z/p\Z$-action. Since the rightmost ribbon (here the blue one) goes the lowest, when distorting the ribbons on~$\bord\un\Delta_P$, all crossings that appear are positive. 
	On the right, the two projections of all ribbons~$\Rib_p^1, \ldots, \Rib_p^{p-1}$ on~$\bord\un\Delta_P/\Gamma_P$, seen in the slice-of-cake model (the one whose fundamental domain is the space located between two vertical walls). 
	The blue ribbons are obtained when pushing to the right, and the orange ones and when pushing to the left.
		}
	\label{F:TemplateDef}
\end{figure}

\subsection{Compactification and coordinates}
\label{S:Compactification}

The unit tangent bundle~$\un\Orbpq$ is a non-compact 3-manifold with first homology group~$\Z$. 
This can be seen in the previous discussion by considering a loop of tangent vectors based along a horocycle centered at~$Z$, and checking that this loop is not null-homologous. For addressing Question~\ref{Q:OneHandedKnot}, we want to compute linking numbers in~$\un\Orbpq$. 

As stated in the introduction, we will make a more general computation by first compactifying~$\un\Orbpq$ into a rational homology sphere, and then compute linking in the resulting manifold. 
Since~$\Orbpq$ has one cusp, a natural compactification that does not change the homology type consists in adding a boundary-circle. 
For~$\un\Orbpq$, this corresponds to the addition of a boundary-torus. 
As we want a compactification with trivial first rational homology group, we need to fill this torus. 
A natural choice is to fill the boundary-circle with a disc and to lift this filling. 
But this choice is not appropriate for the hyperbolic structure, and a more adapted choice is to force all vectors tangent to a given horocycle to bound a disc in the compactification. 
These two compactifications are defined according to whether we see the boundary circle as a hole or a cusp. 
Actually, there is one filling of the torus, and therefore one compatification by a circle for every choice of a Euler number, thus leading to a fiber bundle with the chosen Euler class (as explained by Pinsky~\cite{Tali}). 
The hole-like filling corresponds to Euler number~$0$, while the cusp-like filling has Euler number~$-1$. This leads to the following

\begin{defi}
\label{D:Compactification}
The \emph{hyperbolic compactification}~$\barre{\un\Orbpq}$ of~$\un\Orbpq$ is obtained by adding a fiber associated to the cusp~$Z$, that is, by considering the topology induced by the compactification of~$\Orbpq$ in the hyperbolic disc.
\end{defi}

The compactification~$\barre{\un\Orbpq}$ is obtained by gluing the two solid tori~$\uQ$ and $\uP$ (with $\Delta'$ instead of~$\Delta$ in the case $p=2$) along their boundaries. 
It is then a lens space. 
In order to describe it, let us introduce some notation (see Figure~\ref{F:Spq8}). 
We write~$\BoundTorus$ for the 2-torus that is the boundary between~$\uQ$ and~$\uP$. 
We define~$\tinf$ to be the loop in~$\BoundTorus$ describing the fiber~$\un\{Z\}$ with the trigonometric orientation, and~$c_P$ to be the curve consisting of tangent vector based on~$\barre\Side$ and oriented by the geodesics going through~$P$. We define~$c_Q$ in the same way.
We also consider the set~$\hat D_P$ of all vectors based on points of~$\Delta_P$ and pointing in the direction of~$Z$, and its quotient~$D_P$ under the projection~$\barre{\un\Delta_P}\to\barre{\un\Delta_P/\Gamma_P}$ with the induced orientation (see Figure~\ref{F:MeridianDisc}). We write~$\bord D_P$ for the oriented boundary of~$D_P$. We define~$D_Q$ and~$\bord D_Q$ in the same way.

\begin{lemma}
\label{L:BoundBasis}
$(i)$ The set~$D_P$ is a meridian disc of the solid torus~$\uP$.

$(ii)$ The homology classes~$[c_P]$ and $[c_Q]$ form a basis of~$H_1(\BoundTorus; \Z)$. In this basis, we have the decompositions $[\tinf] = (1,1)$, $[\bord D_P]=(p-1,-1)$ and $[\bord D_Q]=(-1,q-1)$.
\end{lemma}

\begin{proof}
$(i)$ The disc~$\hat D_P$ is contractible in~$\barre{\un\Delta_P}$ and its boundary belongs to the boundary~$\barre{\un(\bord\Delta_P})$. 
Therefore its quotient~$D_P$ is also contractible in~$\barre{\un\Delta_P/\Gamma_P}$, and its boundary~$\bord D_P$ belongs to the boundary~$\barre{\un(\bord\Delta_P/\Gamma_P)}$, which is, by definition, the 2-torus~$\BoundTorus$. 
The loop~$\bord D_P$ is not contractible in~$\BoundTorus$ because its projection on the basis is not. 
Therefore $D_P$ is a meridian disc in~$\BoundTorus$.

\begin{figure}[hbt]
	\begin{center}
	\begin{picture}(135, 80)(0,0)
 	\put(0,0){\includegraphics*[scale=1]{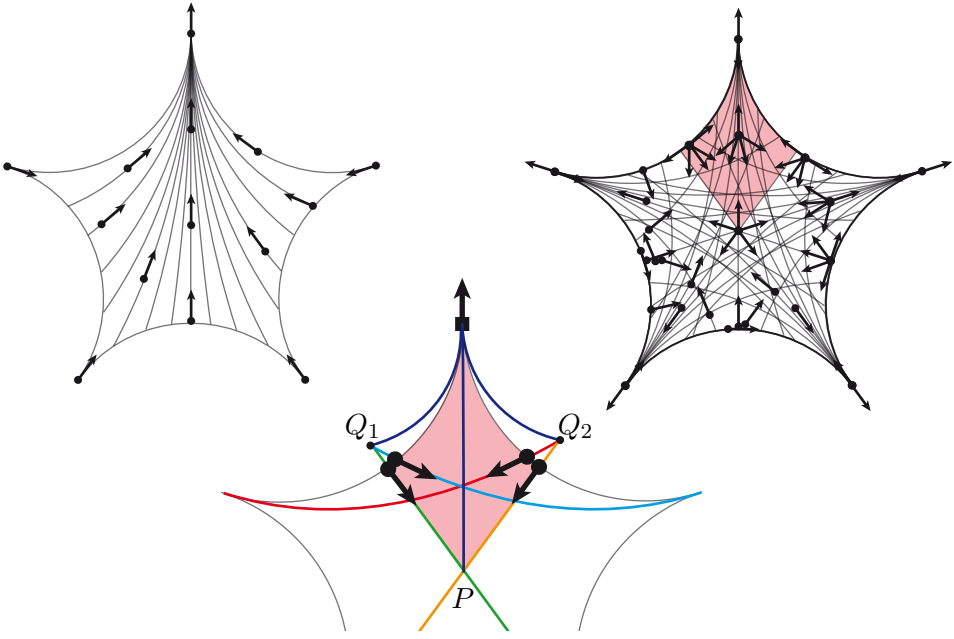}}
	\put(46,26){$Q_1$}
	\put(74,26){$Q_2$}
	\put(60,3){$P$}
	\end{picture}
	\end{center}
	\caption{\small On the left, the meridian disc~$\hat D_P$, with $p=5$. It is the set of all vectors pointing at~$Z$. On the right, the union of its iterated images under the rotation of angle~$2\pi/p$. It is the set of all vectors pointing at one of the $p$ vertices of~$\Delta_P$. The meridian disc~$D_P$ (Lemma~\ref{L:BoundBasis}) is obtained by restricting to a fundamental domain, for example the tinged part. On the bottom, the $p-1$ intersection points between~$c_Q$ and $\bord D_P$ (the leftmost and the rightmost vectors are identified in the quotient). The unique intersection point between~$c_P$ and $\bord D_P$ is the vector based at~$Z$ with a squared origin.}
	\label{F:MeridianDisc}
\end{figure}
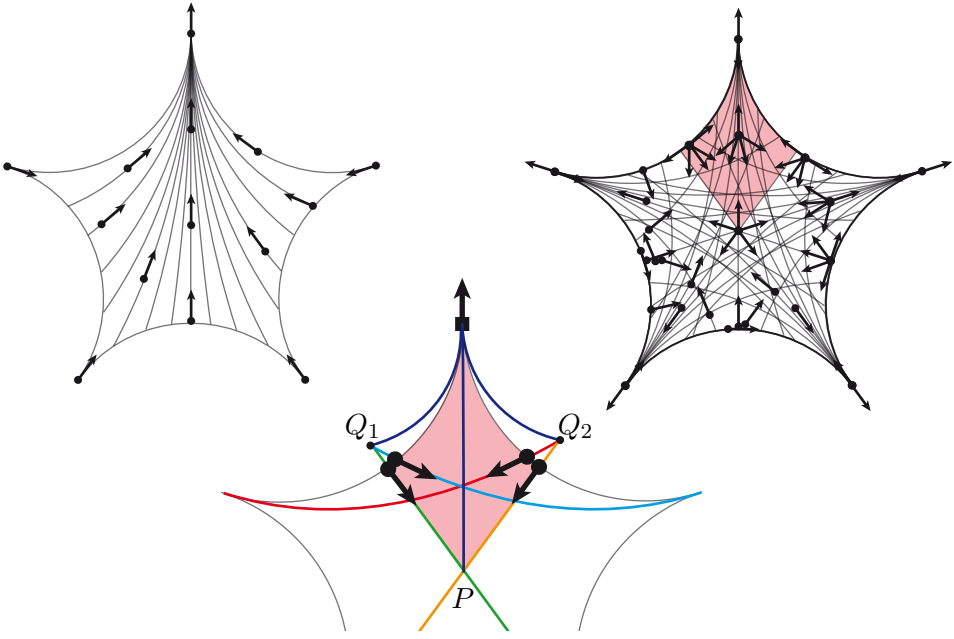

$(ii)$ We write~$\cdot$ for the intersection form on the torus~$H_1(\BoundTorus; \Z)$. By definition, and as indicated on Figure~\ref{F:Spq8}, the three curves~$c_P, c_Q$ and $\tinf$ have one point in common, namely the unit tangent vector based at~$Z$ and oriented by outgoing geodesics. Therefore, we have
$$\vert[c_P]\cdot[c_Q]\vert=\vert[\tinf]\cdot[c_P]\vert=\vert[\tinf]\cdot[c_Q]\vert = 1,$$
so that the classes~$[c_P]$ and $[c_Q]$ form a basis of~$H_1(\BoundTorus; \Z)$. We orient the 2-torus~$\BoundTorus$ in such a way that the basis~$([c_P], [c_Q])$ is positive. The signs of the intersections~$[\tinf]\cdot [c_P]$ and $[\tinf]\cdot [c_P]$ can be determined by checking that the concatenation of the loops~$c_P$ and $c_Q$ is homotopic to~$\tinf$, so that, in the basis~$([c_P], [c_Q])$, we have~$[\tinf] = (1,1)$.

In order to determine the coordinates of~$[\bord D_P]$, we compute the intersection numbers with the basis vectors. 
For the intersection between~$[\bord D_P]$ and~$[c_P]$, we see on Figure~\ref{F:MeridianDisc} that there is only one vector in~$[\bord D_P]\cap[c_P]$, namely the vector based at~$Z$ and oriented by outgoing geodesics. 
For the intersection between~$[\bord D_P]$ and~$[c_Q]$, we have to count the vectors emerging from~$Q$ and pointing into one of the vertices~$Z_p^k$. 
There are $p-1$ such vectors, depicted on Figure~\ref{F:MeridianDisc}. Once again, the signs can be determined by checking that the loop $\bord D_P$ is isotopic to the concatenation of $p$ times~$c_P$ and one time~$\tinf$, taken backwards, whence the relation~$[\bord D_P] = (p-1, -1)$.

The coordinates of~$[\bord D_Q]$ are determined in the same way.
\end{proof}

\begin{figure}[hbt]
	\begin{center}
	\begin{picture}(65, 55)(0,0)
 	\put(0,0){\includegraphics*[scale=1]{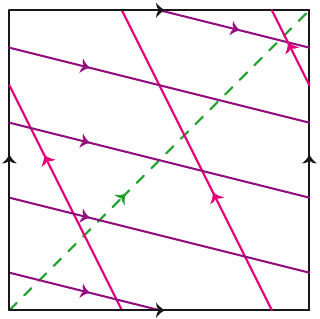}}
	\put(27,-2){$c_p$}
	\put(-4,28){$c_q$}
	\put(22,20){$\tinf$}
	\put(12,38){$\bord D_P$}
	\put(39,16){$\bord D_Q$}
	\end{picture}
	\end{center}
	\caption{\small The median torus~$\BoundTorus$, in the basis $(c_p, c_q)$, here with $p=5$ and $q=3$. The hyperbolic compactification~$\barre{\un\Orbpq}$ is obtained by gluing two solid tori~$\uP$ and~$\uQ$, with respective meridian~$\bord D_P$ and~$\bord D_Q$, along~$\BoundTorus$. The fibers of the points of~$e$ are the curves on~$\BoundTorus$ that are parallel to~$\tinf$.}
	\label{F:MedianTorus}
\end{figure}

We can now deduce the topology of~$\un\Orbpq$. 

\begin{lemma}
\label{L:Topology}
The hyperbolic compactification~$\barre{\un\Orbpq}$ of the unit tangent bundle to the orbifold~$\Orbpq$ is diffeomorphic to the lens space~$L_{pq-p-q, p-1}$, the circle added when compactifying being a $(p,q)$-torus knot drawn on a median torus of~$L_{pq-p-q, p-1}$.
\end{lemma}

\begin{proof}
(See Figure~\ref{F:MedianTorus}.)
We continue with the same notation. Since~$\barre{\un\Orbpq}$ is obtained by gluing the two solid tori~$\barre{\un\Delta_P/\Gamma_P}$ and $\barre{\un\Delta_Q/\Gamma_Q}$, it is a lens space. 
By Lemma~\ref{L:BoundBasis}, the two curves~$\bord D_P, \bord D_Q$ are respective meridians in the two solid tori. 
Using their coordinates, we deduce that their intersection number is~$\left| \begin{matrix} p-1 & -1 \\ -1 & q-1 \end{matrix} \right|= pq-p-q$. 
As the curve~$c_Q$ intersects~$\bord D_Q$ once, it is a parallel for the solid torus~$\barre{\un\Delta_Q}$. As $c_Q$ intersects $p-1$ times~$\bord D_P$, the 3-manifold~$\barre{\un\Orbpq}$ is the lens space~$L_{pq-p-q,p-1}$.

The circle that has been added when compactifying is the fiber~$\tinf$ of the point~$Z$. By Lemma~$\ref{L:BoundBasis}$, it intersects $p$ times the circle $\bord D_P$, and $q$~times $\bord D_Q$. 
Therefore it is a $(p,q)$-torus knot
\end{proof}

\begin{remark}
Since $(p-1)(q-1)\equiv 1\mod pq-p-q$, Brody's theorem asserts that the lens spaces $L_{pq-p-q,p-1}$ and $L_{pq-p-q, q-1}$ are diffeomorphic. 
This can be seen in the above proof by exchanging~$p$ and $q$.
\end{remark}

\begin{remark}
One can check that the alternative compactifications of~$\un\Orbpq$ associated with other Euler numbers can be obtained by cutting along~$\BoundTorus$, making a transvection along the curve~$\tinf$, and gluing back. This changes the manifold into~$L_{kpq-p-q, kp-1}$ for some $k$ in~$\Z$ (see Pinsky~\cite{Tali} for more detail).
\end{remark}

\begin{figure}[p]
	\begin{center}
 	\begin{picture}(60,63)(0,0)
 	\put(0,0){\includegraphics*[width=.38\textwidth]{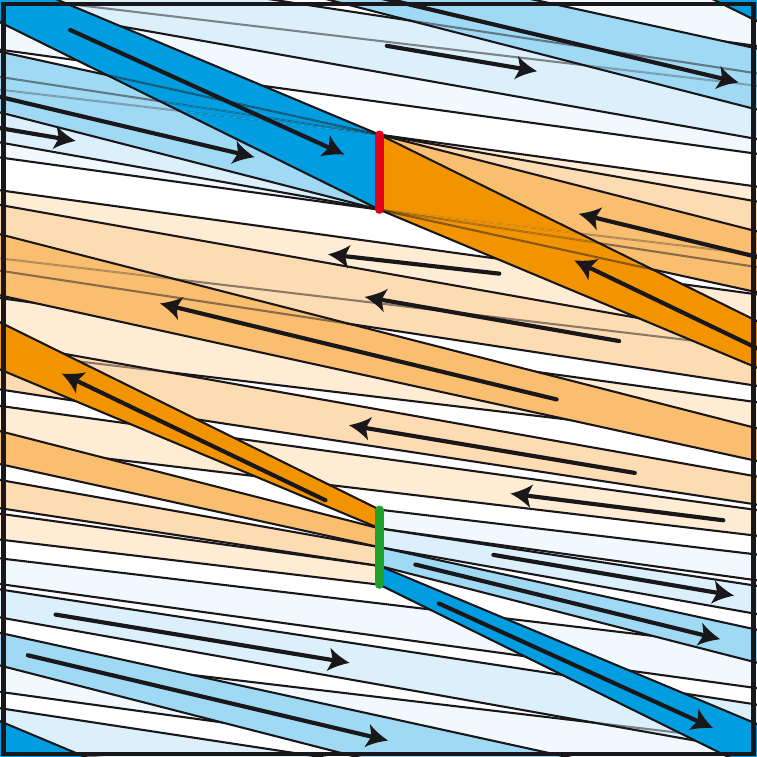}}
	\end{picture}
	\hspace{5mm}	
 	\begin{picture}(60,63)(0,0)
	\put(0,0){\includegraphics*[width=.38\textwidth]{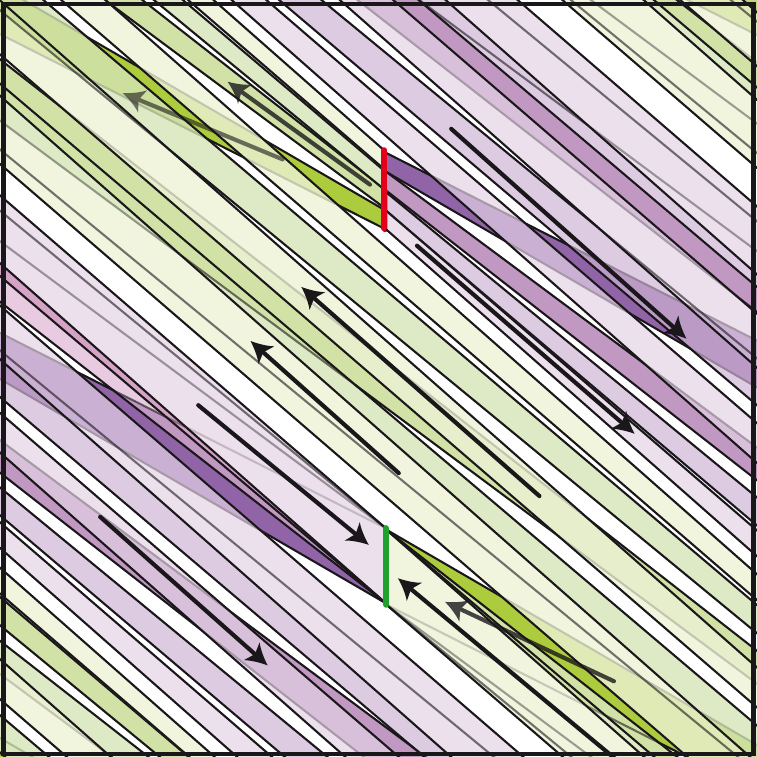}}
	\end{picture}
	\vspace{3mm}

 	\begin{picture}(60,63)(0,0)
	\put(0,0){\includegraphics*[width=.38\textwidth]{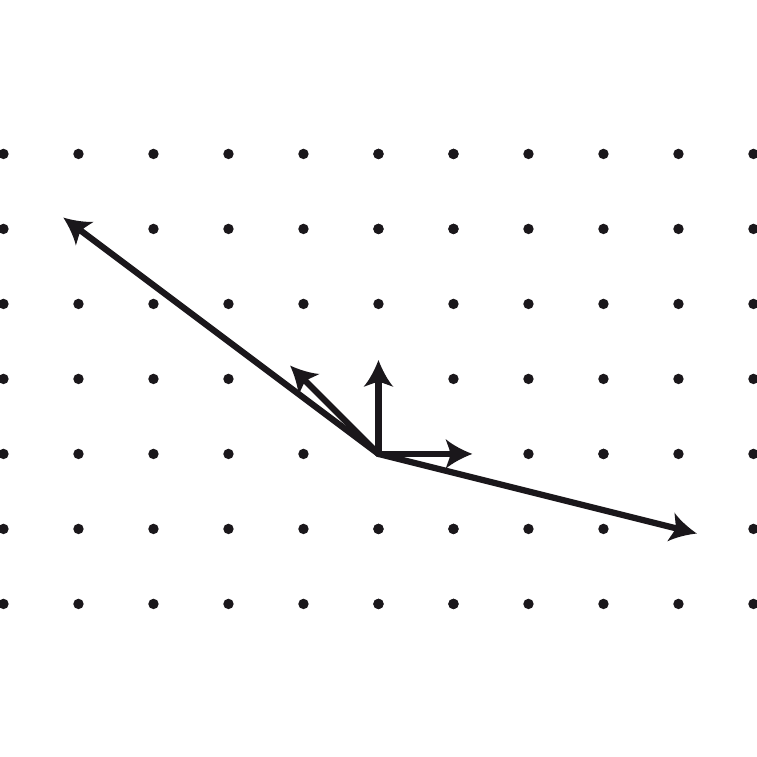}} 
	\put(50,13){$\bord D_Q$}
	\put(3,43){$\bord D_P$}
	\put(30,32){$\tinf$}
	\put(36,25){$c_Q$}
	\put(23,32 ){$c_P$}
	\end{picture}
	\hspace{3.5mm}	
 	\begin{picture}(60,63)(0,0)
	\put(0,0){\includegraphics*[width=.38\textwidth]{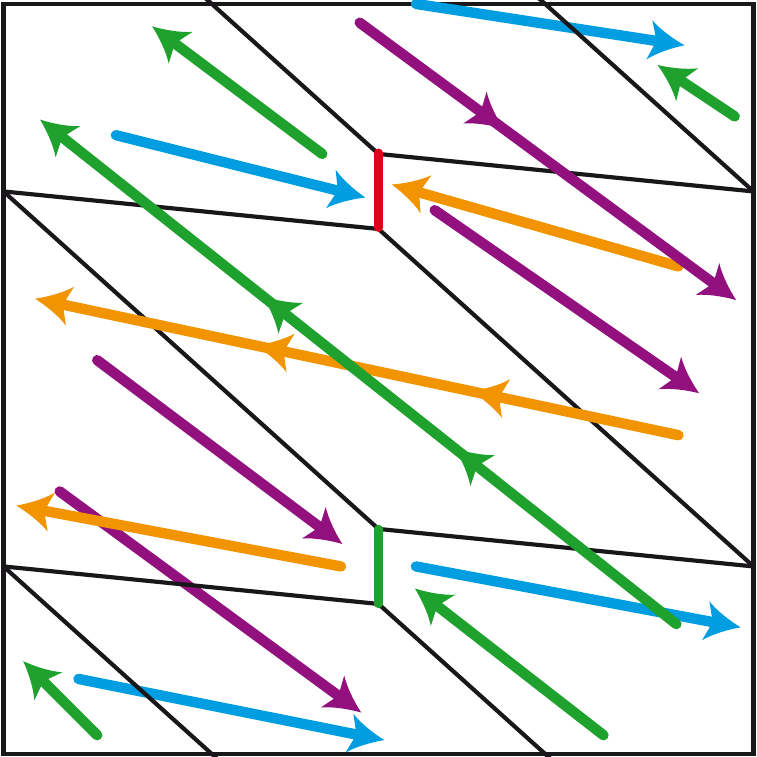}}
	\put(24,23){\begin{rotate}{-17}{$\BranchPQ$}\end{rotate}}
	\put(24,52){\begin{rotate}{-17}{$\BranchQP$}\end{rotate}}
	\end{picture}
	\end{center}
	\caption{\small The projection of the template~$\Tpq$ on the 2-torus~$\BoundTorus$, here with $p=q=4$. The sources of the projection are the fibers~$\un\{P\}$ for the part that lies inside the solid torus~$\uP$, and~$\un\{Q\}$ for what lies inside~$\uQ$. On the top left, the part of~$\Tpq$ lying inside~$\uQ$. On the top right, the part of~$\Tpq$ lying inside~$\uP$. The two pictures differ by a transvection. This is due to the choice of the compactification. Changing the compactification of~$\un\Orbpq$ leads to another transvection for the identification. Since the two solid tori are glued outgoing normal vs.\ incoming normal, the two pictures have opposite orientations, namely the front/back order of the ribbons is reversed. On the bottom left, the vectors~$[c_P], [c_Q], [\tinf], [\bord D_P]$ and $[\bord D_Q]$ in $H_1(\BoundTorus; \Z)$. The slope~$-1$ of $[c_P]$ explains the transvection on the top right picture. On the bottom right, the directions of the two possible deformations of the ribbons that constitute~$\Tpq$ on~$\BoundTorus$. The four colors correspond to the four types~$\Rib_p^{i, g}$, $\Rib_p^{i, d}$, $\Rib_q^{i, g}$ and $\Rib_q^{i, d}$. The key point for proving the negativity of linking numbers (Proposition~\ref{P:2q}) is that, in each of the two vertical intervals between~$\BranchPQ$ and~$\BranchQP$, all ribbons go in the same direction.}
	\label{F:ProjTemp}
\end{figure}

We now have a full description of the template~$\Tpq$ and of how it embeds into~$\barre{\un\Orbpq}$ (see also Figure~\ref{F:ProjTemp}). It is worth noting that in the case $p=2, q=3$, the compactification~$\barre{\un\Orbpq}$ is the 3-sphere, the fiber~$\tinf$ of the cusp is a trefoil knot, and the template~$\Tpq$ is Lorenz' template, as stated by Ghys~\cite{GhysMadrid}.

\subsection{Linking with the fiber of the cusp}

For $p\ge2, q\ge 3$, we use now the template~$\Tpq$ for computing the linking number in~$\barre{\un\Orbpq}$ between a periodic orbit of~$\GF{\Orbpq}$ and the $(p,q)$-torus knot~$\tinf$ that has been added when compactifying~$\un\Orbpq$ (Proposition~\ref{P:LkInfty}). This computation has been done in the case~$p=2, q=3$ by Ghys~\cite{GhysMadrid}. In this case, the linking number equals the Rademacher function of the underlying geodesics---a function of interest in number theory~\cite{Ogg}. 
As before, we assume that we are given a triangle~$PQZ$ in~$\Hy^2$, that~$\Gamma_{p,q}$ is the associated Hecke triangular group, that~$\Tiling_{P,Q}$ is the associated adapted tessellation of~$\Hy^2$, and that~$\Tpq$ is the associated template. 

Let~$\underline\gamma$ be a geodesic of~$\Hy$ whose extremities are not lifts of the cusp of~$\Orbpq$, that is, $\underline\gamma_+$ and $\underline\gamma_-$ are not in the orbit~$\Gamma_{p,q}(Z)$. Then picking an arbitrary starting point on it, $\underline\gamma$ is determined by a starting tile~$\Tile_0$ and a bi-infinite code~$\ldots u^{i_{-1}}v^{j_{-1}}u^{i_{0}}v^{j_{0}}u^{i_{1}}v^{j_{1}}\ldots$ describing how $\underline\gamma$ behaves in each tile of the tessellation~$\Tiling_{P,Q}$. Precisely, if $\gamma$ enters a copy of~$\Delta_P$ by a side, and goes out by another side that is obtained from the entering one by a rotation of angle~$2i\pi/p$, then the corresponding letter is~$u^i$. Similarly, when $\underline\gamma$ enters a copy of~$\Delta_Q$, the corresponding letter~$v^j$ describes how to pass from the entering side to the outgoing side. As $\Delta_P$ has~$p$ sides, every index~$i_k$ is between $1$ and $p-1$. Similarly, every index $j_k$ lies between $1$ and $q-1$. Considering another starting tile induces a shift of the code. If two geodesics are obtained one from the other by the action of an element~$g$ of~$\Gamma_{p,q}$, then their starting tiles are also obtained from one another by~$g$, and their codes coincide. Therefore, there is a one-to-one correspondence between codes up to shift and geodesics on the orbifold~$\Orbpq$ not pointing into the cusp. Moreover, if a geodesic on~$\Orbpq$ is periodic, then its code is periodic, that is, of the form~$(u^{i_1}\ldots v^{j_m})^\Z$. In this case, we call the word~$u^{i_1}\ldots v^{j_m}$, which is assumed to be of minimal possible length, a~\emph{reduced code} of the periodic geodesic. Different reduced codes for a given periodic geodesic differ by a cyclic permutation of the letters. 

We now define an invariant of periodic geodesics that will be useful for expressing the linking number of their liftings in~$\un\Orbpq$ with the fiber of the cusp. Assume that $\underline\gamma$ is a geodesic in~$\Hy^2$ with code~$\ldots u^{i_{-1}}v^{j_{-1}}u^{i_{0}}v^{j_{0}}u^{i_{1}}v^{j_{1}}\ldots$. For a more symmetric expression, we set~$i'_k = i_k-p/2$ and $j'_k = j_k- q/2$. Then the discretisation~$\discr{\underline\gamma}{\Tiling_{P,Q}}$ of~$\underline\gamma$ lies in the tree depicted in Figure~\ref{F:PTessel}. By definition, for every index~$k$, the discretisation~$\discr{\underline\gamma}{\Tiling_{P,Q}}$ turns by an angle~$2\pi i'_k/p$ in the corresponding copy of~$\Delta_P$ and by an angle~$2\pi j'_k/q$ in the corresponding copy of~$\Delta_Q$.

\begin{defi}
\label{D:WheelTurn}
Assume that $\underline\gamma$ is a periodic geodesic on~$\Orbpq$. Let $u^{i_1}v^{j_1}u^{i_2}\ldots v^{j_m}$ be a reduced code of~$\underline\gamma$. Then the \emph{wheel turn}~$\WT{\underline\gamma}$ of~$\underline\gamma$ is the rational number $\sum_{i=1}^m i'_k/p + j'_k/q$.
\end{defi}

Here is the expected evaluation of the linking number between a geodesic of~$\GF{\Orbpq}$ and the fiber of the cusp in terms of an analog of the Rademacher function. 

\begin{prop}
\label{P:LkInfty}
Assume $p \ge 2, q \ge 3$. Then, for every periodic orbit~$\gamma$ of the geodesic flow~$\GF{\Orbpq}$, we have
$$\lk(\gamma, \tinf) = \frac{pq}{pq-p-q} \WT{\underline\gamma},$$
where $\underline\gamma$ is the projection of~$\gamma$ on~$\barre\Orbpq$.
\end{prop}

The principle of the proof is as follows. Write $r$ for the number~$pq-p-q$. Since the first homology group of~$L_{pq-p-q, p-1}$ is~$\Z/r\Z$, we know that for every element~$[c]$ of~$H_1(L_{pq-p-q, p-1}; \Z$), the cycle~$r[c]$ is a boundary of an integral 2-chain. The idea will be to construct a 2-chain with boundary~$r[\gamma]$ which is transverse to~$\tinf$, and then to count the intersection number with~$\tinf$. Since~$\gamma$ is isotopic in the complement of~$\tinf$ in an orbit of the template~$\Tpq$, we can then make use of the available information about the position of the latter in~$\barre{\un\Orbpq}$.

In order to implement the argument, let us write~$h$ for the orbit of~$\Tpq$ whose code is~$(u^1v^1)^\Z$. 
Note that~$h$ is one of the two periodic orbits of~$\Tpq$ that is not isotopic to a periodic orbit of the geodesic flow, but to a periodic orbit of the horocyclic flow. Write~$a_P$ for the curve that describes the fiber~$\un P$. It is the core of the solid torus~$\uP$. Similarly, write~$a_Q$ for curve describing the fiber~$\un Q$. 
We begin with a preliminary computation. 
Remember that $u^{i_1}v^{j_1}u^{i_2}\ldots v^{j_m}$ denotes a reduced code of~$\gamma$.

\begin{lemma}
\label{L:HomolPQ}
In the above context, the cycle~$[\gamma]$ is homologous in~$\un\Orbpq$ to the 1-cycle
$$\sum_{k=1}^m \bigg([h]+ (i_k-1)[a_P] + (j_k-1)[a_Q]\bigg).$$
\end{lemma}

\begin{proof}
Let $\gamma^1$ be the image of~$\gamma$ under the deformation~$\deform{\Tiling_{P,Q}}^1$. Then $\gamma^1$ is an orbit of the template~$\Tpq$. Suppose that~$\gamma$ first travels along the ribbon~$\Rib_{p}^{1}$, and then along~$\Rib_{q}^{1}$. 
Then it is homologous to~$h$ in~$\un\Orbpq$ during the corresponding interval of time, and its code starts with~$u^1v^1$. 
Otherwise, the homology class of~$\gamma$ in the complement of~$\tinf$ during one period is obtained by adding to~$h$ the cycles consisting in traveling along~$\Rib_{p}^{i}$ backwards and then along~$\Rib_{p}^{i+1}$ frontwards, for every $i$ between $1$ and $i_k$, and by also adding the cycles consisting in traveling along~$\Rib_{q}^{j}$ backwards and then along~$\Rib_{q}^{j+1}$ forwards, for every $j$ between $1$ and $j_k$. 
Every cycle in the first category is actually equal to~$[a_P]$. 
Indeed, the ribbons are not the same, but the annuli of the form $\Rib_{p}^{j+1}-\Rib_{p}^{j}$ are homologous in the quotient~$\barre{\un\Orbpq}$: they correspond to curves turning once around the point~$P$ on~$\Orbpq$.
So $[a_P]$ is added $i_k-1$~times. 
Similarly, every cycle in the second category is equal to~$[a_Q]$, so $[a_Q]$ is added $j_k-1$~times. 
\end{proof}

We can then complete the argument.

\begin{proof}[Proof of Proposition~\ref{P:LkInfty}]
In~$\barre{\un\Orbpq}$, the cycle~$h$ bounds a disc whose intersection number with~$\tinf$ equals~$-1$. Indeed, since~$h$ is homologous to a horocyle, the latter bounds a horodisc, say~$d_h$, which is foliated by horocycles parallel to~$h$. By definition of the compactification, the family of all vectors tangent to these horocycles extends to the fiber of the cups, and therefore form a disc in~$\barre{\un\Orbpq}$ that intersects~$\tinf$ in exactly one point, namely the limit of the tangent vectors. Hence we have~$\lk(h, \tinf) = -1/r$.

Let us turn to~$\lk(h, \tinf)$. Write~$c^+_P, \bord D^+_p$ and $\bord D^+_q$ for the curves $c_P, \bord D_p$ and $\bord D_q$ slightly pushed away from~$\BoundTorus$ in~$\uP$, so that they do not intersect~$\tinf$. As~$c_Q$ is a parallel for the solid torus~$\uP$, the cycle~$r[a_P]$ is homologous in~$\un\Orbpq$ to~$r[c^+_P]$. The latter has coordinates~$(pq-p-q,0)$ in the basis~$([c_P], [c_Q])$ of~$H_1(\BoundTorus; \Z)$. By Lemma~\ref{L:BoundBasis}, the cycles~$[\bord D^+_p]$ and $[\bord D^+_q]$ have coordinates~$(p-1, -1)$ and $(-1, q-1)$, so that $r[c^+_P]$ equals $(1-q)[\bord D^+_p] - [\bord D^+_q]$. On the one hand, $[\bord D^+_p]$ bounds a meridian disc for~$\uP$ that does not intersect~$\tinf$. On the other hand, since $\bord D_q$ and~$\tinf$ intersect $q$ times on~$\BoundTorus$, the curve~$\bord D^+_q$ bounds a meridian disc for~$\uQ$ that intersect $-q$ times~$\tinf$. Therefore we have~$\lk(a_P, \tinf) = q/r$. We obtain in the same way~$\lk(a_Q, \tinf) = p/r$. The result then follows from Lemma~\ref{L:HomolPQ}.
\end{proof}

\subsection{Linking number between collections of geodesics}
\label{S:LkPQ}

We now restrict to the case~$p=\nobreak2$, and study the linking number between two collections of periodic geodesics of~$\GF{\Orbdq}$. Our goal is to show 

\begin{prop}[case $(a)$ of Theorem~A]
\label{P:2q}
Assume $q\ge 3$. Then, for all collections of periodic orbits~$\gamma, \gamma'$ of the geodesic flow~$\GF{\Orbdq}$ in~$\barre{\un\Orbdq}$, the linking number between~$\gamma$ and $\gamma'$ is negative.
\end{prop}

The proof of this statement will occupy the rest of Section~\ref{S:q}. The strategy is as follows. Owing to Proposition~\ref{P:Template}, it is enough to show that the linking number of every pair~$\gamma, \gamma'$ of collections of periodic orbits of the template~$\Tdq$ is negative. By Lemma~\ref{L:Topology}, the first homology group of~$\barre{\un\Orbdq}$ is $\Z/(q-2)\Z$, so that the $1$-cycle~$(q-2)[\gamma]$ is the boundary of some 2-chain. What we shall do is  to explicitly construct a 2-chain~$S$ whose boundary is~$(q-2)[\gamma]$, and to show that the intersection number of~$S$ with~$\gamma'$ is negative. As the family~$\gamma'$ lies in the template~$\Tdq$, working with the 1-skeleton of the template as in the proof of Proposition~\ref{P:LkInfty} is impossible. Instead, we shall choose a particular projection of~$\Tdq$ on~$\BoundTorus$ and reduce the problem to computing intersection numbers on~$\BoundTorus$. Practically, we shall construct the $2$-chain~$S$ as the union of three parts, namely a 2-chain~$\Sq$ lying inside the solid torus~$\un\Delta'_Q/\Gamma_Q$, a 2-chain~$\Sp$ lying inside the solid torus~$\un\Delta'_P/\Gamma_Q$, and a 2-chain~$\Sb$ in the torus~$\BoundTorus$. Then we shall show that the intersection number between~$\Sq$ and~$\gamma'$ is slightly positive, that the intersection number between~$\Sp$ and~$\gamma'$ is zero, and that the intersection number between~$\Sb$ and $\gamma'$ is very negative, so that the sum of these three numbers is negative, as expected.

Let us turn to the construction of the 2-chains~$\Sp$, $\Sq$, and~$\Sb$. They will be defined by glueing discs whose boundaries will consist of \emph{elementary arcs}, some particular segments drawn inside the ribbons of the multitemplate~$\Tdq$. 

As depicted on Figure~\ref{F:ProjTemp}, every ribbon~$\Rib_q^j$ of~$\Tdq$ can be distorted in two ways on~$\BoundTorus$, according to whether the ribbon is pushed on its right or on its left (see Figure~\ref{F:TemplateQuotient}). We denote by $\Rib_q^{i, l}$ and $\Rib_q^{i, r}$ the two ribbons in~$\BoundTorus$ produced that way. Similarly, the ribbon~$\Rib_p$ can be pushed on the right or on the left, and can thus be distorted on two ribbons on~$\BoundTorus$. We denote them by $\Rib_p^l$ and $\Rib_p^r$. What we shall do is to decompose the orbits of~$\Tdq$ into pieces lying inside a ribbon, and choose for every such piece a combination of the two possible projections, so that the sum of these projections is null-homologous in~$\BoundTorus$ (see Figure~\ref{F:ProjOrbit}). Here is the precise notion. 

\goodbreak

\begin{defi}
\label{D:ProjRibbon}
We say that $\alpha$ is an \emph{elementary arc} (of~$\Tdq$) if $\alpha$ is 

- (type 1) either a segment of an orbit in~$\Tdq$ that goes from a point~$A_0$ of $\BranchPQ$ to a point~$A_1$ of $\BranchQP$ and travels along the ribbon~$\Rib_q^i$ for some $i$ between $1$ and~$q-1$; then we write~$\alpha^l$ and $\alpha^r$ for the segments of~$\BoundTorus$ that connect~$A_0$ to $A_1$ and are orbits in the ribbons~$\Rib_q^{i, l}, \Rib_q^{i, r}$ respectively,

- (type 2) or a segment of an orbit in~$\Tdq$ that goes from~$\BranchQP$ to~$\BranchPQ$ by travelling along~$\Rib_p$; then we write~$\alpha^l$ and $\alpha^r$ for the deformations of~$\alpha$ that are orbits of the ribbons~$\Rib_p^l$ and~$\Rib_p^r$.
\end{defi}

\begin{figure}[hbt]
	\begin{center}
	\begin{picture}(90,90)(0,0)
 	\put(0,0){\includegraphics*[width=0.52\textwidth]{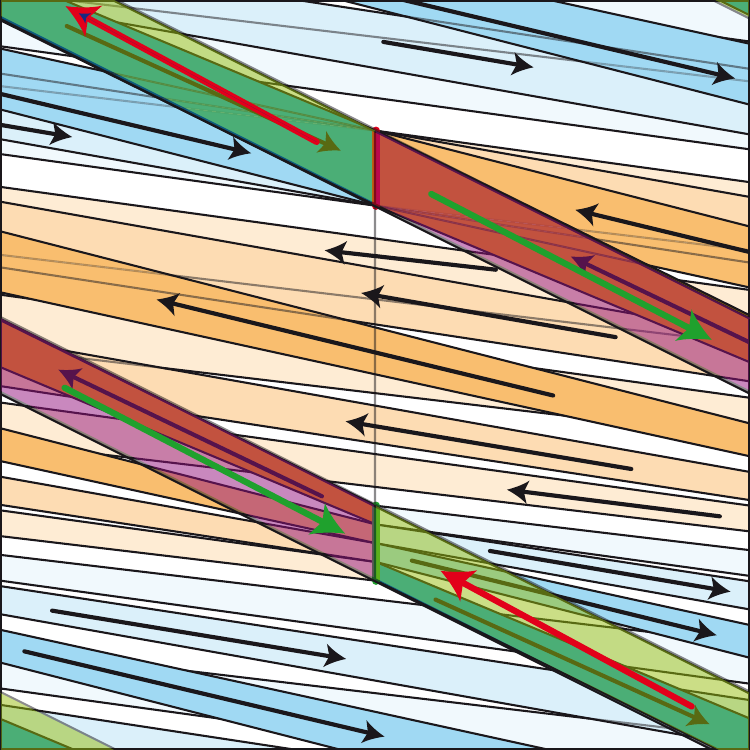}}
	\put(85,15){$\Rib^{i,d}_q$}
	\put(85,57){$\Rib^{j,g}_q$}
	\end{picture}
	\end{center}
	\caption{\small The two projections of~$\Tdq$ on~$\BoundTorus$, with $q=5$, in the slide-of-cake model. The curve~$a_Z$ is the vertical boundary of the depicted square, while the curve~$c_Q$ is the horizontal boundary. The solid torus~$\un\Delta'_P/\Gamma_P$ is in front of the picture, so that we see on the front the two projections of the ribbon~$\Rib_2^1$. On the back, the two projections of each of the four ribbons~$\Rib_q^1, \ldots, \Rib_q^4$.}
	\label{F:ProjOrbit}
\end{figure}

We now choose a canonical projection of every elementary arc to a convenient multicurve. So assume that $\alpha, \alpha'$ are elementary of type~$1$ and $2$ respectively, and that the end of~$\alpha$ coincides with the origin of~$\alpha'$. Note that the condition about the ends implies that $\alpha'$ is uniquely determined by~$\alpha$. Then we denote by~$\alpha_\pi$ the multicurve consisting of $i$ times~$\alpha^l$ and $q-2-i$ times~$\alpha^r$, followed by $i$ times~$\alpha'^l$ and $q-2-i$ times~$\alpha'^r$. The reason for this particular choice is the following 

\begin{lemma}
\label{L:ProjNulle}
Let $\alpha_1, \alpha'_1 \ldots, \alpha_n, \alpha'_n$ be the decomposition of~$\gamma$ into a concatenation of elementary arcs of type 1 and 2 alternately. 
Then the union~$\gamma_\pi$ of the multicurves~$(\alpha_1)_\pi, (\alpha'_1)_\pi$, \dots, $(\alpha_n)_\pi$, $(\alpha'_n)_\pi$ is a multicurve on~$\BoundTorus$ that is trivial in homology.
\end{lemma}

\begin{proof}
We see on Figure~\ref{F:ProjOrbit} that, for every~$i$, the ribbon~$\Rib_q^{i,r}$ (blue on the picture) cuts the curve~$\tinf$ (the vertical boundary on the picture) $i$ times and the curve~$c_Q$ (the horizontal boundary) $-1$ times. Similarly, $\Rib_q^{i,l}$ (orange) cuts~$\tinf$ (vertical) $q-i$ times and~$c_Q$ zero time. In the same way, $\Rib_p^r$ cuts~$\tinf$ minus one times and~$c_Q$ one time, whereas $\Rib_p^l$ cuts~$\tinf$ one time and~$c_Q$ zero time.

Suppose that~$\alpha, \alpha'$ are two consecutive elementary arcs of~$\gamma$, with~$\alpha$ lying on the ribbon~$\Rib_q^i$ for some~$i$ and~$\alpha'$ lying on~$\Rib_p$. Then the above remark implies that the (non-close) multi-curve~$\alpha\cup\alpha'$ has zero-intersection with both~$\tinf$ and~$c_Q$. By adding the contributions of all elementary arcs of~$\gamma$, we deduce that~$\gamma$ is null-homologous in~$\BoundTorus$.
\end{proof}

We are now going to define the 2-chains~$\Sp$ and~$\Sq$. If $\alpha$ is an elementary arc of type~1, we denote by~$S_q^\alpha$ the 2-cycle consisting of $i$ times a disc in~$\barre{\un\Delta'_Q/\Gamma_Q}$ with boundary~$\alpha\cup -\alpha^l$ plus $q-2-i$ times a disc with boundary~$\alpha\cup -\alpha^r$. Symmetrically, if~$\alpha$ is of type~2, we denote by~$S_p^\alpha$ the $2$-cycle consisting of $i$ times a disc in~$\barre{\un\Delta'_P/\Gamma_P}$ with boundary~$\alpha \cup -\alpha^l$ plus $q-2-i$ times a disc with boundary~$\alpha \cup -\alpha ^r$.

\begin{defi}
With the above notation, we define~$\Sq$ to be the union of the 2-cycles $S_q^{\alpha_1}, \ldots, S_q^{\alpha_n}$, and $\Sp$ to be the union of the 2-cycles~ $S_p^{\alpha'_1}, \ldots, S_p^{\alpha'_n}$. 
\end{defi}

The next step is to complete~$\Sp\cup\Sq$ into a 2-chain with boundary~$\gamma$. Owing to Lemma~\ref{L:ProjNulle}, this can be done inside~$\BoundTorus$.
Indeed, the multi-curve~$\gamma_\pi$ divides~$\BoundTorus$ into a finite number of regions, say~$R_1, \ldots, R_n$, that can be seen as 2-chains. Since~$[\gamma_\pi]$ is zero in~$H_1(\BoundTorus; \Z)$, there exists an integral linear combination~$\sum \lambda_k [R_K]$ with boundary~$\gamma_\pi$. In fact, the coefficients~$\lambda_k$ are defined up to a constant only. With our particular choice of the projection~$\gamma_\pi$, at every point of~$\BranchPQ$ or~$\BranchQP$, the number of segments of~$\gamma_\pi$ that come from the left (\resp right) equals the number of segments that leave to the left (\resp right). 

\begin{defi}
Let us choose numbers~$\lambda_k$ so that, for every region~$R_i$ intersecting~$\BranchQP$, the associated coefficients~$\lambda_i$ is zero. Then we define~$\Sb$ to be the 2-chain $\sum\lambda_k R_k$.
\end{defi}

Note that, by construction, the boundary of the 2-chain~$\Sb$ is the multicurve~$\gamma_\pi$. 

At this point, we have associated with the first collection of periodic orbits~$\gamma$ a certain 2-chain~$\Sp\cup\Sq\cup\Sb$ that, by construction, has boundary~$\gamma$. Let us now consider the second collection of periodic orbits~$\gamma'$, which is assumed to be disjoint from~$\gamma$. We shall estimate the intersection number between~$\gamma'$ and each of the 2-chains $\Sp, \Sq$ and~$\Sb$, and prove that their sum is negative. For this, we introduce specific combinatorial data encoding the position of the collections~$\gamma$ and~$\gamma'$ inside the template~$\Tdq$.

\begin{lemma}
\label{L:EnlSpSq}
The collection~$\gamma'$ does not intersect the 2-chain~$\Sp$, and the intersection number between~$\gamma'$ and~$\Sq$ is at most
\begin{equation}
\label{Eq:EnlSq}
\sum_{1\le i<j\le q-1} \left(\left\lfloor\frac{j-i}2\right\rfloor + 1\right) (i-1) \, b_i b'_j
+ \sum_{1\le j<i\le q-1} \left(\left\lfloor\frac{i-j}2\right\rfloor + 1\right) (q-1-i) \, b_i b'_j,
\end{equation}
where, for every~$i$ between~$1$ and $q-1$, $b_i$ (\resp $b'_i$) is the number of elementary arcs of~$\gamma$ (\resp $\gamma'$) lying in the ribbon~$\Rib_q^i$.
\end{lemma}

\begin{figure}[hbt]
	\begin{center}
 	\includegraphics*[scale=0.8]{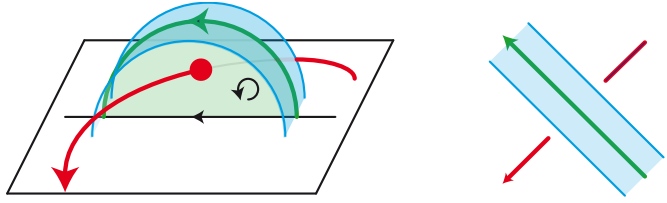}
	\end{center}
	\caption{\small An elementary arc of $\gamma'$ (red) may intersect the 2-chain $\Sq$ (green) only if it lies on a ribbon under the ribbon containing~$\gamma$ (blue). Since the projection of this intersection point corresponds to a positive crossing (see Figure~\ref{F:ProjOrbit}), the intersection number is~$+1$.}	
	\label{F:PositiveIntersection}
\end{figure}

\begin{proof}
Every intersection point between~$\gamma'$ and the 2-chains~$\Sp, \Sq$ is the intersection between one of the discs bounded by an elementary arc, say~$\alpha$, of~$\gamma$ and one of its two projections~$\alpha^l$ or~$\alpha^r$, and an elementary arc, say~$\alpha'$, of~$\gamma'$. This implies (see Figure~\ref{F:PositiveIntersection}) that $\alpha, \alpha'$ project on~$\BoundTorus$ on a double point, and that the ribbon containing~$\alpha'$ between~$\BoundTorus$ and the ribbon containing~$\alpha$. In particular, these two ribbons have to be different.

On Figure~\ref{F:ProjOrbit}, one sees that all intersections between projected ribbons on~$\BoundTorus$ correspond to two ribbons of type~$\Rib_q^{i,l}$ and~$\Rib_q^{j,l}$, or to two ribbons of type~$\Rib_q^{i,r}$ and~$\Rib_q^{j,r}$. Therefore no intersection point comes from~$\Rib_p^l$ or $\Rib_p^r$, so that $\gamma'$ does not intersect~$\Sp$. We also see that, for every $i,j$, the two projected ribbons~$\Rib_q^{i,l}$ and~$\Rib_q^{j,l}$ intersect~$\lfloor\vert i - j\vert / 2\rfloor$ times transversely, and overlap just before the gluing segment~$\BranchQP$.

The collection of the numbers~$b_i$ does not determine the position of the orbit~$\gamma$ on~$\Tdq$ completely. In particular, it does not say whether two orbits on~$\Rib_q^{i,l}$ and~$\Rib_q^{j,l}$ respectively will overlap before~$\BranchQP$. Nevertheless, since all projected crossings are positive, we obtain an upper bound for the intersection number when assuming that two such elementary arcs always overlap before~$\BranchQP$. 

By construction, there are $i\,b_i$ elementary arcs of~$\gamma_\pi$ on~$\Rib_q^{i,l}$, and $(q-2-i)\,b_i$ on~$\Rib_q^{i,r}$. Each elementary arc of type~1 yields at most~$(\lfloor (j-i)/2\rfloor +1)b'_j$ intersection points with elementary arcs of~$\gamma'$ lying on~$\Rib_q^{j}$ if $j>i$, and no intersection point of~$j\le i$. Similarly, for $j<i$, each elementary arc of type~2 yields at most~$(\lfloor (j-i)/2\rfloor +1)b'_j$ intersection points with elementary arcs of~$\gamma'$ lying on~$\Rib_q^{j}$. All intersection points are positive, and \eqref{Eq:EnlSq} follows.
\end{proof}

We now compute the contribution of~$\Sb$ to the linking number of~$\gamma$ and~$\gamma'$. For convenience, we set $\Delta = \sum_{1\le i\le q-1} (i-1)(q-1-i)\,b_i$.

\begin{lemma}
\label{L:EnlSb}
The intersection number between~$\gamma'$ and~$\Sb$ is at most
\begin{equation}
\label{Eq:EnlSf}
\sum_{j\le q/2}(-\Delta+\sum_{k\le j} (k-1) \, b_k) b'_j
+ \sum_{j > q/2}(-\Delta+\sum_{k> j} (q-1-k) \, b_k) b'_j.
\end{equation}
\end{lemma}

\begin{proof}
Since~$\gamma'$ intersects the torus~$\BoundTorus$ on~$\BranchPQ$ and~$\BranchQP$ only, we have to estimate the coefficients~$\lambda_k$ of the associeted regions in the 2-chain~$\Sb$. By definition of~$\Sb$, the coefficient of every region intersecting~$\BranchQP$ is zero. Since every elementary arc of~$\gamma'$ that intersects~$\BranchPQ$ goes from the solid torus~$\un\Delta'_Q/\Gamma_Q$ into~$\un\Delta'_P/\Gamma_P$, the intersection number between~$\gamma'$ and~$\Sb$ is exactly the sum of the levels of the intersection points of~$\gamma'$ with~$\BranchPQ$.  Let us cut the segment~$\BranchPQ$ into~$q-1$ segments, say~$[M_1M_2], \ldots, [M_{q-1}M_q]$, corresponding to the origins of the ribbons~$\Rib_q^1, \ldots, \Rib_q^{q-1}$.

We claim that the level of the points~$M_1$ and~$M_q$ is~$-\Delta$. Indeed, starting from~$M_q$ (the top point in the segment~$\BranchPQ$ on Figure~\ref{F:ProjOrbit}), and following the fiber until we reach~$\BranchQP$, we intersect the projections of all several ribbons of type~$\Rib_q^{i,l}$. For every~$i$, there are~$q-1-i$ such intersections, all being positive. Since the ribbon~$\Rib_q^{i,l}$ contains $i\,b_i$ elementary arcs of~$\gamma$, we cross~$\gamma$ exactly~$\Delta$ times along the path. The same argument works for~$M_1$.

Now we claim that, for~$i\le q/2$, the level at every point of~$[M_iM_{i+1}]$ is at most~$\Delta+\sum_{k\le i} (k-1) \, b_k$. Indeed, when starting from~$M_1$ and following~$\BranchPQ$, the level changes when we cross an intersection point of~$\gamma$ with~$\BranchPQ$. Let $B$ be such a point. Then there are~$q-2$ elementary arcs of~$\gamma$ arriving at~$B$ from the ribbon~$\Rib_p$. Depending on the ribbon~$\Rib_q^j$ followed before~$\Rib_p$, the~$q-2$ elementary arcs of the projection~$\gamma_\pi$ arriving at~$B$ decompose into $l-1$ of them arriving from the left along~$\Rib_q^{j,r}$, and $q-1-l$ arriving from the right along~$\Rib_q^{j,l}$. Similarly, since $\gamma$ leaves $B$ along~$\Rib_q^i$, there are~$q-2$ elementary arcs of~$\gamma_\pi$ that leave~$B$, $i-1$ of them on the left along~$\Rib_q^{i,l}$, and $q-1-i$ of them on the right along~$\Rib_q^{i,r}$. Therefore the difference of level under and above~$B$ is~$i-l$. In particular, it is at most~$i$. Using an induction on $i$, we deduce that the level is at most~$\Delta+\sum_{k\le i} (k-1) \, b_k$ at~$M_{i+1}$, and \emph{a fortiori} at every point on~$[M_iM_{i+1}]$. We get a similar result for $i$ larger that~$q/2$. Equation~\eqref{Eq:EnlSf} easily follows.
\end{proof}

We are now able to complete the argument.

\begin{proof}[Proof of Proposition~\ref{P:2q} (case $(a)$ of Theorem~A)] 
We continue with the same notations. Equation~\eqref{Eq:EnlSf} bounding the intersection number between~$\gamma'$ and~$\Sb$ expands into
\begin{eqnarray*}
&-&\sum_{1\le j\le q/2}\big(\sum_{i\le j} (i-1)(q-2-i)\,b_i + \sum_{i > j} (i-1)(q-1-i)\,b_i\big) b'_j \\
&-&\sum_{q/2 < j \le q-1}\big(\sum_{i< j} (i-1)(q-1-i)\,b_i + \sum_{i \ge j} (i-2)(q-1-i)\,b_i\big) b'_j.
\end{eqnarray*}

By adding Equation~\eqref{Eq:EnlSq}, we obtain the following expression
\begin{eqnarray*}
	\sum_{1\le j\le q/2}
	(&\sum_{i< j}& -(i-1) \big(q-2-i - \left\lfloor({j-i})/2\right\rfloor - 1\big)\,b_i\, \\
	+ &\sum_{i > j}& -(q-1-i) \big(i - \left\lfloor({i-j})/2\right\rfloor - 1 \big)\,b_i \quad )\quad b'_j \\
	+ \sum_{q/2 < j \le q-2}(
	&\sum_{i< j}& -(i-1)\big(q-1-i-\left\lfloor({j-i})/2\right\rfloor - 1\big)\,b_i \\
	+ &\sum_{i > j}& -(q-1-i)\big( i-2 -\left\lfloor({i-j})/2\right\rfloor -1 \big)\,b_i\quad) \quad b'_j,
\end{eqnarray*}
plus some terms in~$b_ib'_j$ whose coefficients all are negative. Therefore, the intersection number between~$\gamma'$ and~$S$ is bounded from above by a quadratic form in the families $(b_{i,j}), (b'_{i,j})$, all of whose coefficients are negative. Therefore, the linking number~$\lk(\gamma,\gamma')$ is negative.
\end{proof}

 
\section{Surfaces and orbifolds of type~$(2,3,4g{+}2)$ }
\label{S:237}

We now turn to the hyperbolic 2-orbifolds~$\Orbg$ and to case~$(b)$ in Theorem~A, namely the result that every two collections of periodic orbits of the geodesic flow on~$\Orbg$ are negatively linked. We recall from the introduction that, as the unit tangent bundle~$\un\Orbg$ is a quotient of the unit tangent bundle of a specific hyperbolic surface~$\Sigma_g$ of genus~$g$, our strategy will be to lift the question to~$\un\Sigma_g$, estimate the linking number between lifts of orbits of~$\GF{\Orbg}$, and eventually use Lemma~\ref{L:Covering}. 

In the whole section, $g$ denotes a fixed integer larger than or equal to 2.  The successive steps are as follows. We start in~\S\,\ref{S:Template3} with a $4g{+}2$-gon  in the hyperbolic plane and consider the multitemplate~$\temp_{4g+2}$  provided by Proposition~\ref{P:Template}. Mimicking the method of the previous section, we bound in~\S\,\ref{S:Bounds3} the linking number of a pair of collections of periodic orbits of~$\temp_{4g+2}$ by a quadratic form~$Q_{4g+2}$  in terms of the number of arcs that travel along every ribbon of~$\temp_{4g+2}$. The form~$Q_{4g+2}$ is not negative on the cone of admissible coordinates for geodesics on~$\Sigma_g$, but, using symmetries to reduce the set of possible coordinates, we introduce a refined form~$S_{4g+2}$ in~\S\,\ref{S:Lk237}, and show that the linking form is negative on the reduced cone.

\subsection{A template for $\GF{\Sigma_g}$ }
\label{S:Template3}

From now on, we fix a regular~$4g{+}2$-gon~$\Polg$  in the hyperbolic plane whose angles all are equal to~$2\pi/{(2g{+}1)}$ . We write~$e_1, \dots, e_{4g+2}$ for the sides of~$\Polg$. For every side~$e_i$, we write~$e_i^l$ for its left extremity (when looking at $e_i$ from inside~$\Polg$), and~$e_i^r$ for its right extremity. We also write $e_{\cci}$ for the side opposite to~$e_i$ (that is, we set~$\cc i = i{+}2g{+}1\mod 4g{+}2$). We call~$\Sigma_g$ the genus~$g$-surface obtained by identifying opposite sides of~$\Polg$. The vertices of type~$e_{2k}^r$ then project to one point of~$\Sigma_g$, say~$V_0$. Similarly, the vertices of type~$e_{2k}^l$ project to one point, say~$V_1$. The unit tangent bundle to~$\Polg$ is the product~$\Polg\times\bord_\infty\Hy^2$, where a tangent vector is identified with its direction on~$\bord_\infty\Hy^2$. Then~$\un\Polg$ is a solid torus whose boundary is made of the 4g{+}2 annuli~$\un e_1, \dots, \un e_{4g+2}$. The unit tangent bundle~$\un\Sigma_g$ is obtained from~$\un\Polg$ by identifying opposite annuli \emph{via} homographies of~$\Hy^2$. Precisely, if $g_{i,\cci}$ denotes the isometry that maps~$e_i$  to~$e_{\cci}$, then $g_{i,\cci}$ extends to~$\bord_\infty\Hy^2$, and the fibers of two paired points of~$e_i$ and $e_{\cci}$ are identified using the extension of~$g_{i, \cci}$ to~$\bord_\infty\Hy^2$.
We also introduce two small discs~$D_0, D_1$ on~$\Sigma_3$ centered at~$V_0, V_1$ respectively. We write~$\Polgint$ for the complement of~$D_0\cup D_1$ in~$\Polg$. This is a domain whose boundary is made of $4g{+}2$ geodesic segments and $4g{+}2$ arcs of circle are small radius.

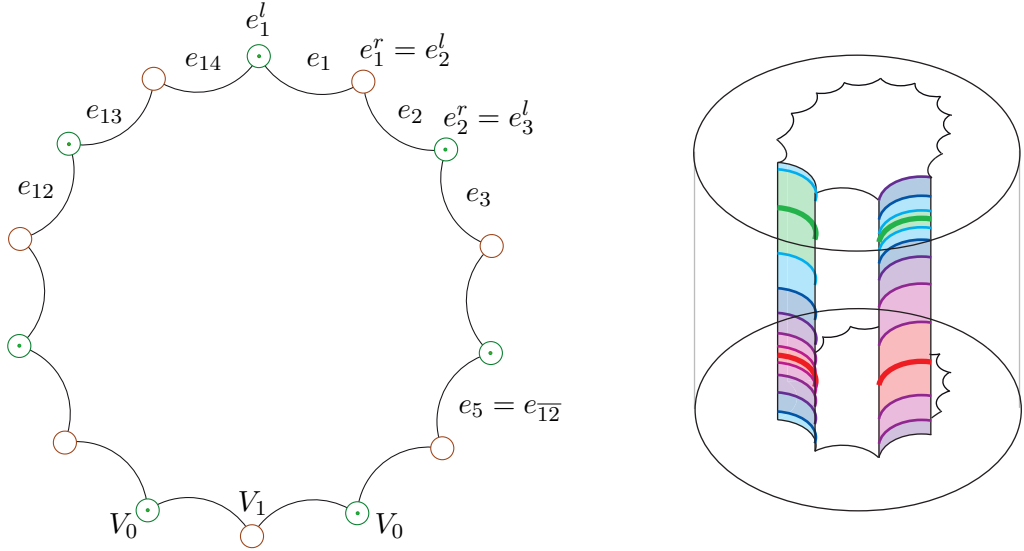
\begin{figure}[hbt]
	\begin{center}
	\begin{picture}(135,70)(0,0)
 	\put(0,0){\includegraphics*[scale=.4]{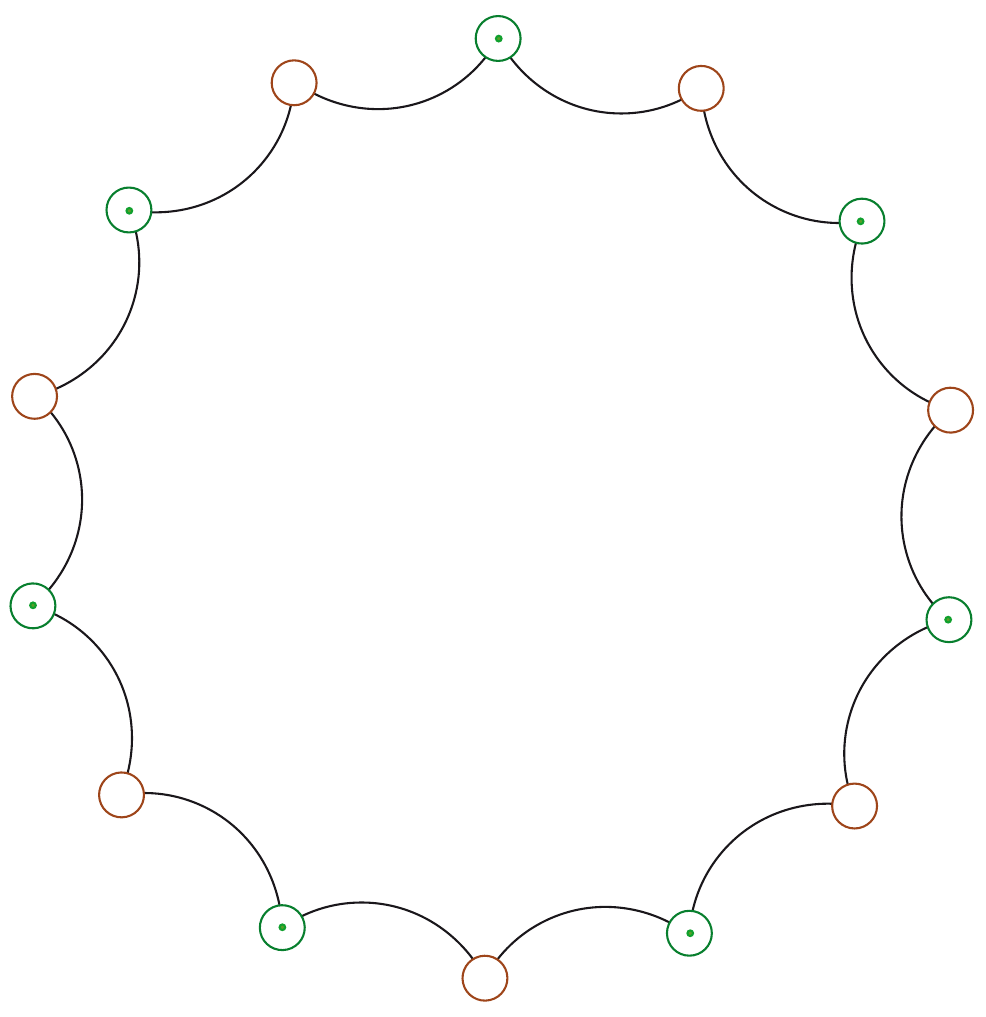}}
	\put(40,65){$e_1$}
	\put(52,58){$e_2$}
	\put(61,47){$e_3$}
	\put(60,19){$e_{5} = e_{\overline{12}}$}
	\put(2,48){$e_{12}$}
	\put(11,58){$e_{13}$}
	\put(24,65){$e_{14}$}
	\put(32,70){$e_1^l$}
	\put(47,66){$e_1^r = e_2^l$}
	\put(58,57){$e_2^r = e_3^l$}
	\put(14,3){$V_0$}
	\put(49,3){$V_0$}
	\put(31,6){$V_1$}
 	\put(90,5){\includegraphics*[scale=.5]{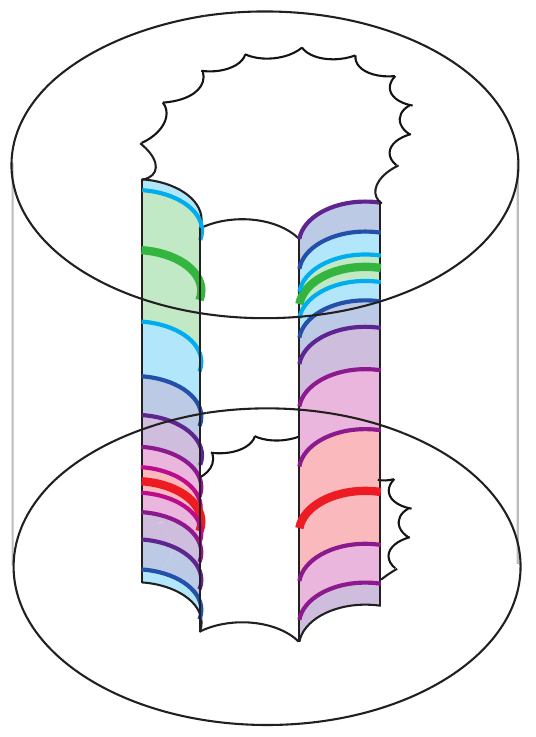}}
	\end{picture}
	\end{center}
	\caption{\small On the left, the regular 14-gon~$P_{14}$. The surface~$\Sigma_3$ is obtained by identifying opposite sides. On the right, the unit tangent bundle is obtained in the standard coordinates $\Polg\times\bord_\infty\Hy^2$ by gluing opposite walls using homographies.} 
	\label{F:237}
\end{figure}

The fundamental group~$\pi_1(\Sigma_g)$ is generated by the isometries~$g_{1,\bar 1}, \dots, g_{4g+1, \bar{4g+1}}$. We write~$\Tiling_g$ for the tessellation of~$\Hy^2$ induced by the images of~$\Polg$ under~$\pi_1(\Sigma_g)$. One easily checks that~$\Tiling_g$ is adapted to~$\Sigma_g$ (only point~$(v)$ in Definition~\ref{D:Tiling} requires  some attention).
Finally, we choose a graph $\Gr_{\Tiling_g}$ dual to~$\Tiling_g$ and an associated discretisation of geodesics.
We then write~$\temp_{4g+2}$ for the corresponding template in~$\un\Sigma_g$.
By definition, it consists of $(4g{+}2)(4g{+}1)$  ribbons connecting every pair of distinct boundary annuli. For every~$i,j$, we denote by~$\Rib_{i,j}$ the ribbon that connects $\un e_i$ to~$\un e_j$. Above every side of~$\Polg$, there are two branching segments, corresponding to geodesics crossing the side in both directions. The length of each branching segment is half the length of the fiber. Since we are interested in the topology of~$\temp_{4g+2}$ only, we can distort it using an isotopy, so that each branching segment has a small length, say~$\epsilon$, and consists of vectors that are almost orthogonal to~$e_i$. We then obtain a template similar to the one depicted on Figure~\ref{F:ProjTemplate}. For every edge~$e_i$ of~$\Polg$, we denote by~$\BS_{i, \cci}$ the branching segment that contains the orbits arriving on the side~$e_i$ and leaving from the side~$e_{\cci}$, and by~$\BS_{\cci, i}$ the other branching segment that contains the orbits arriving on the side~$e_{\cci}$ and leaving from the side~$e_{i}$. 

\begin{figure}[hbt]
	\includegraphics*[scale=.45]{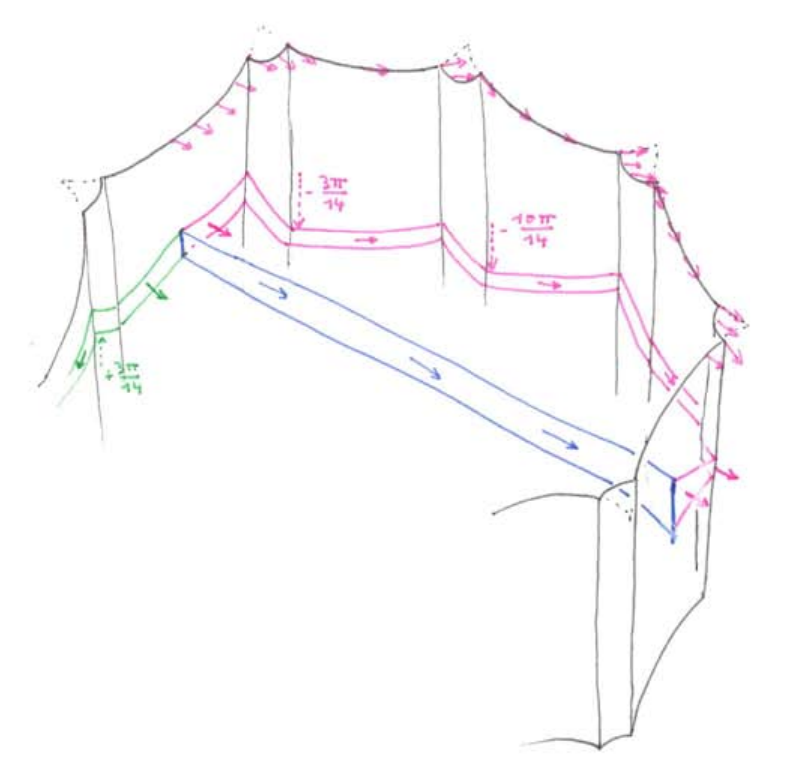}
	\hspace{1cm}
	\includegraphics*[scale=.45]{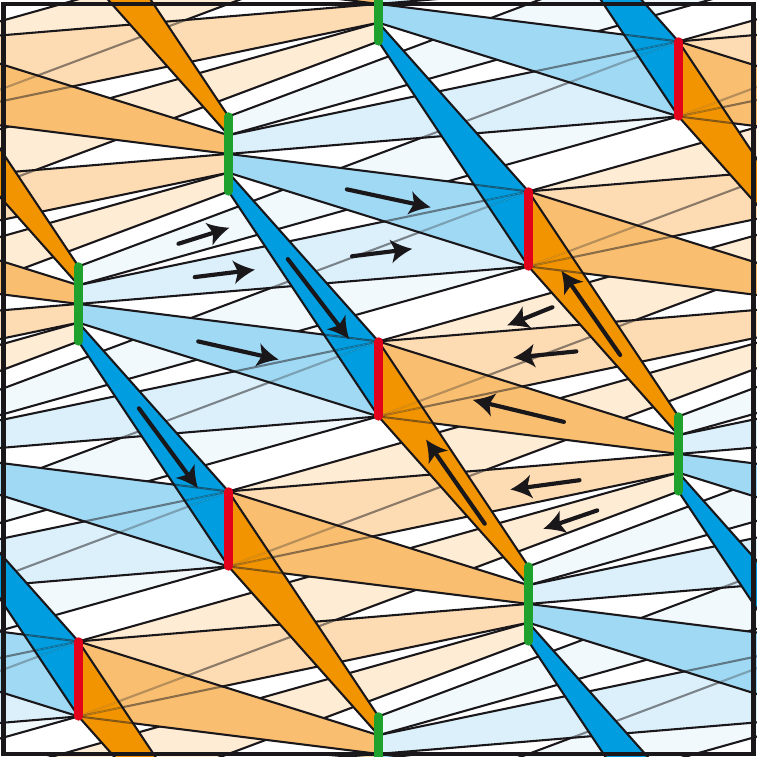}
	\caption{\small On the left, a ribbon~$\Rib_{i,j}^l$ and its projection~$\Rib_{i,j}^l$ on~$\un\bord\Polg$. By definition, it is horizontal in every wall of type~$\un e_i$. As stated in Lemma~\ref{L:ToursFibres}, it goes down around every vertex. This long descent (which is reminiscent of a picture by Escher) stems from the observation that a vector travelling along the left boundary of~$\Polg$ and staying tangent to~$\bord\Polg$ has to turn right at every vertex. On the right, the two projections of the whole template~$\temp_{4g+2}$ on~$\un\bord\Polg$ (with 5 instead of~$4g{+}2$): every ribbon has a blue and an orange projection.} 
	\label{F:ProjTemplate}
\end{figure}

In the sequel, we need two particular deformations of~$\temp_{4g+2}$ on~$\un\bord\Polgint$ that we describe now. Assume that~$\Rib_{i,j}$ is a ribbon of~$\temp_{4g+2}$. We isotope~$\Rib_{i,j}$ to the boundary of~$\un\Polgint$ without changing the extremities in two ways. For the first deformation, we push~$\Rib_{i,j}$ to the left until reaching~$\un\Polgint$ (see Figure~\ref{F:ProjTemplate} left). The image is denoted by~$\Rib_{i,j}^l$. Moreover, we choose the isotopy so that\\
- the part of~$\Rib_{i,j}^l$ lying in~$\un e_i$ has width~$\epsilon$ and consists of vectors almost orthogonal to~$e_i$,\\
- for every $c$ satisfying $j<c<i$ in the cyclic order, the part of~$\Rib_{i,j}^l$ lying in~$\un e_c$ has width~$\epsilon$ and consists of vectors almost parallel to~$e_c$,\\
- the part of~$\Rib_{i,j}^l$ lying in~$\un e_j$ has width~$\epsilon$ and consists of vectors almost orthogonal to~$e_i$.\\
We construct $\Rib_{i,j}^r$ similarly by pushing~$\Rib_{i,j}$ to the right in~$\un\Polgint$. We write~$\temp_{4g+2}^l$ for the union of all left projections of ribbons of~$\temp_{4g+2}$, and $\temp_{4g+2}^r$ for the union of all right projections (see Figure~\ref{F:ProjTemplate} right).

\subsection{Bounds for the linking number of orbits in $\GF{\Sigma_g}$}
\label{S:Bounds3}

Our goal is now to estimate and to bound the linking number between two null-homologous collections of periodic orbits of~$\temp_{4g+2}$. We will do that by considering the number of times the given collections travel along every ribbon of~$\temp_{4g+2}$. The formula may look convoluted, but hopefully the meaning of every term should be clear from the proof. The key point is that the bound we establish is bilinear in the number of times each collection travels along every ribbon, so that it can be easily estimated. We use Knuth's convention and write $\{\cdot \}$ for the characteristic function of a property. Also the inequality signs refer to the cyclic order in~$\Z/(4g{+}2)\Z$. The functions~$v_0, v_1,h_0, h_1$ will be defined in Definitions~\ref{D:v} and~\ref{D:h} below.

\begin{defi}
\label{D:BoundLk}
For every $i,j,k,l$ in $\{1, \ldots, 4g{+}2\}$ with $i\neq j$ and $k\neq l$, we define the real number~$q_{i,j,k,l}$ by 
\begin{eqnarray*}
\label{Eq:TotalLk}
& & \frac 1 2\big(\{i < k < l \le j\} + \{k < i < j \le l\}\big) \\
&-& \frac 1 8\big(\{k\neq i, j\} + \{k\neq \bar i, \bar j\} \big) \\
&+& \vzero i j \hzero k l + \vun i j \hun k l + \frac 1 {2g-2} (\vzero i j + \vun i j) (\vzero k l + \vun k l);
\end{eqnarray*}
we write~$Q_{4g+2}$ for the bilinear form on~$\R^{(4g+2)(4g+1)}$ whose coefficients are the~$q_{i,j,k,l}$.
\end{defi}

\begin{defi}
\label{D:LinearCode}
Assume that $\gamma$ is a null-homologous collection of periodic orbits of the template~$\temp_{4g+2}$. For every $i,j$ in $\{1, \ldots, 4g{+}2\}$, let~$b_{i,j}$ denote the number of arcs of~$\gamma$ that travel along the ribbon~$\Rib_{i,j}$,  respectively. The family $(b_{i,j})_{1\le i\neq j\le 4g+2}$ consists of $(4g{+}2)(4g{+}1)$ non-negative integers, it is called the \emph{linear code} of~$\gamma$.
\end{defi}
 
\begin{prop}
\label{P:Lk3}
Assume that $\gamma, \gamma'$ are two null-homologous collections of periodic orbits of the template~$\temp_{4g+2}$. Denote by $(b_{i,j})$ and $(b'_{i,j})$ their linear codes. Then the linking number $\lk(\gamma, \gamma')$ is at most $\sum_{1\le i,j,k,l\le 4g+2} q_{i,j,k,l}b_{i,j}b'_{k,l}$.
\end{prop}

Note that, in the expression for $q_{i,j,k,l}$ given in Definition~\ref{D:BoundLk}, the roles of~$\gamma$ and~$\gamma'$ are not symmetric. This is connected with our subsequent choice of a particular 2-chain, and with the fact that the coefficients~$b_{i,j}$ satisfy some linear constraints, so that the above formula is one among many other possible expressions.

The idea of the proof of Proposition~\ref{P:Lk3} is to construct a rational 2-chain~$S^\gamma$ with boundary~$\gamma$, and to bound its intersection number with~$\gamma'$. The 2-chain~$S^\gamma$ will consist of four parts, denoted by~$\Spi, \Sd, \Ss$ and $\Sss$, each being a combination of several rational 2-cells. 

We now establish several intermediate results consisting in evaluating various intersection numbers. First, we consider the above defined projections $\temp^l_{4g+2}$ and $\temp^r_{4g+2}$ of~$\temp_{4g+2}$. We write~$\gamma^l_\pi$ for the image of~$\gamma$ that lies in $\temp^l_{4g+2}$, and~$\gamma^r_\pi$ for the image that lies in~$\temp^r_{4g+2}$. 

\begin{defi}
\label{D:Spi}
Let~$\gamma_\pi$ be the combination~$\frac 1 2 \gamma^l_\pi + \frac 1 2 \gamma^r_\pi$. Then we define~$\Spi$ to be the sum, for each elementary arc~$\alpha$ of~$\gamma$, of a (rational) disc~$d_\alpha^l$ with boundary~$\frac1 2(\alpha\cup-\alpha_\pi^l)$ and of a (rational) disc~$d_\alpha^r$ with boundary~$\frac1 2(\alpha\cup-\alpha_\pi^r)$.
\end{defi}

It follows from the definition that $\Spi$ connects~$\gamma$ to~$\gamma_\pi$.

\begin{lemma}
\label{L:LkSpi}
The intersection number between the collection~$\gamma'$ and the rational 2-chain~$\Spi$ is at most $\sum_{i,j,k,l} \frac 1 2(\{i<k<l\le j\}+\{k<i<j\le l\}) \, b_{i,j}b'_{k,l}$.
\end{lemma}

\begin{proof}
We have to estimate, for every pair of elementary arcs~$(\alpha, \alpha')$ of~$\gamma$ and $\gamma'$ respectively, whether~$\alpha'$ intersects the discs~$d_\alpha^l$ and~$d_\alpha^r$ defined above, how many times it possibly does, and what is the sign of the intersection points. 
Let~$\Rib_{i,j}$ denote the ribbon containing~$\alpha$, and let $\Rib_{k,l}$ the ribbon containing~$\alpha'$. 

First, suppose $i\neq k$ and $j\neq l$. Figure~\ref{F:ProjTemplate} right then shows that $\alpha'$ intersects~$d_\alpha^r$ if and only if~$i<k<l<j$ in the cyclic order. In this case, there is only one intersection point, and its sign is positive (Figure~\ref{F:PositiveIntersection} is also relevant here). Since the disc~$d_\alpha^r$  has a coefficient~$\frac 1 2$, the contribution of this intersection point to the total intersection number is~$+\frac 1 2$. Similarly,  $\alpha'$ intersects~$d_\alpha^l$ if and only if~$i<j<l<k$ in cyclic order, and the contribution is then~$+\frac 1 2$.

Second, suppose $i\neq k$ and $j=l$. Then $\alpha'$ may intersect~$d_\alpha^l$ or~$d_\alpha^r$ or not, depending on which arc is above the other on~$\BS_{j, \cc j}$, and which arc comes from the right or the left before reaching~$\un e_j$. Since we look for an upper bound on the linking number, and since the sign of the intersection, if any, is positive, we can assume that there is always an intersection, so that the contribution is~$+\frac 1 2$. This happens if~$i<k<l=j$ or $i<j=l<k$ in the cyclic order. (Note that this is the only approximation that makes our computation of the linking number not exact. It will be refined for symmetric collections of orbits in the next section.)

Third, suppose $i=k$ and $j\neq l$. Then, as in the previous case, the arc $\alpha'$ may intersect~$d_\alpha^l$ or~$d_\alpha^r$ or not. But, unlike the previous case, we can ignore this potential intersection point. Indeed, let $A, B$ denote the respective starting points of~$\alpha$ and~$\alpha'$, which are located in the branching segment~$\BS_{\cci, i}$. Then there is an intersection point if $A$ is under $B$ and at the same time we have $j>l$, or if $A$ is above $B$ and we have $j<l$. At the expense of possibly performing a symmetry, we may restrict to the first case. $A$ under~$B$ means that~$\alpha$ points on the right of~$\alpha'$ on~$\bord_\infty\Hy^2$, whereas $j>l$ means that $\alpha$ escapes on the left of~$\alpha'$. This is possible for $j = l+1$, but this implies that the geodesics of~$\Hy^2$ that have been distorted onto~$\underline\gamma$ and~$\underline\gamma'$ intersect after crossing~$e_i$. As they are geodesics, they cannot intersect twice, so that they did not intersect before crossing~$e_i$. 
Therefore there was a pair of arcs that lie before~$\alpha$ and $\alpha'$ on~$\gamma$ and~$\gamma'$ that was counted in the previous paragraph (since $i=k$) and should not have.
So we can compensate this factor~$+\frac 1 2$ by ignoring the current intersection.

Fourth, suppose $i=k$ and $j=l$. Then $\alpha, \alpha'$ lie on the same ribbon, and $\alpha'$ does not intersect the discs~$d_\alpha^r$ and~$d_\alpha^l$.

Summing up, we obtain the announced upper bound. 
\end{proof}

The second part of~$S^\gamma$ will lie in the $2g{+}1$ annuli~$\un e_c$ with  $1\le c\le 2g{+}1$ (we recall that~$e_c$ is identified with~$e_{c+2g+1}$). Its boundary will be made of~$\gamma_\pi$ plus some curves lying in~$\un\bord D_0$ and~$\un\bord D_1$. Before describing it, we must describe~$\gamma_\pi$ in more detail.

\goodbreak

\begin{lemma}
\label{L:CompoGPi} (See Figure~\ref{F:ProjTemplate} left.)
Let $e_c$ be a side of~$\Polg$. Then the part of~$\gamma_\pi$ that lies in~$\un e_c$ consists of

$(i)$ $\frac 1 2\sum_{i\neq c} b_{i,c}$ arcs joining the fiber $\un e^l_c$ to the branching segment~$\BS_{c,\cc c}$, plus $\frac 1 2\sum_{i\neq c} b_{i,c}$ arcs joining~$\un e^r_c$ to~$\BS_{c,\cc c}$, plus $\frac 1 2\sum_{j\neq \cc c} b_{c,j}$ arcs joining~$\BS_{c,\cc c}$ to~$\un e^l_c$, plus $\frac 1 2\sum_{j\neq \cc c} b_{c,j}$ arcs joining~$\BS_{c,\cc c}$ to~$\un e^r_c$, all these arcs lying at a height that corresponds to vectors escaping from~$\Polg$ almost orthogonally,

$(ii)$ $\frac 1 2\sum_{i\neq \cc c} b_{i,\cc c}$ arcs joining~$\un e^l_c$ to~$\BS_{\cc c,c}$, plus $\frac 1 2\sum_{i\neq \cc c} b_{i,\cc c}$ arcs joining~$\un e^r_c$ to~$\BS_{\cc c, c}$, plus $\frac 1 2\sum_{j\neq c} b_{\cc c,j}$ joining~$\BS_{\cc c, c}$ to~$\un e^l_c$, plus $\frac 1 2\sum_{j\neq c} b_{\cc c,j}$ arcs joining~$\BS_{\cc c, c}$ to~$\un e^r_c$, all these arcs lying at a height that corresponds to vectors entering~$\Polg$ almost orthogonally,

$(iii)$ $\frac 1 2(\sum_{i<c<j<i} b_{i,j} + \sum_{i<j<\cc c<i} b_{i,j})$ arcs joining the fiber $\un e^r_c$ to~$\un e^l_c$
, all these arcs lying at a height that corresponds to vectors almost tangent to $e_c$ and pointing toward~$e^l_c$,

$(iii)$ $\frac 1 2(\sum_{i<j<c<i} b_{i,j} + \sum_{i<\cc c<j<i} b_{i,j})$ arcs joining the fiber $\un e^l_c$ to~$\un e^r_c$
, all these arcs lying at a height that corresponds to vectors almost tangent to $e_c$ and pointing toward~$e^r_c$.
\end{lemma}

\begin{figure}[hbt]
	\begin{center}
	\begin{picture}(50,62)(0,0)
 	\put(0,0){\includegraphics*[scale=.7]{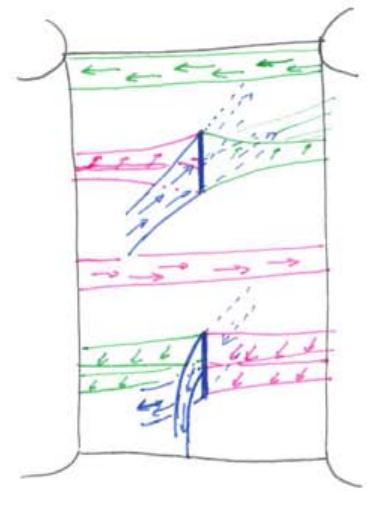}}
	\end{picture}
	\end{center}
	\caption{\small The templates~$\temp_{4g+2}^l$ and $\temp_{4g+2}^r$ inside a wall of type~$\un e_i$. Between the two branching segments, all ribbons have the same orientation. If $\gamma$ is a null-homologous collection of orbits of~$\temp_{4g+2}$, then there are as many arcs of~$\gamma_\pi$ traveling from left to right (along the pink ribbons) as arcs of~$\gamma_\pi$ traveling from right to left (along the green ribbons).}
	\label{F:Sd}
\end{figure}

\begin{proof}
Let~$\Rib_{i,j}$ be a ribbon of~$\temp_{4g+2}$. Then every arc of~$\gamma\cap\Rib_{i,j}$ projects on an arc of~$\gamma^l\cap\Rib^l_{i,j}$ and on an arc of~$\gamma^r\cap\Rib^r_{i,j}$. 
First suppose that the index~$c$ differs from both~$i$ and $j$. If~$e_c$ lies on the right of $\Rib_{i,j}$, then the arcs of $\gamma^r\cap\Rib^r_{i,j}$ travel along $\un e_c$, from~$\un e^r_c$ to~$\un e^l_c$. By construction of~$\Rib^r_{i,j}$, they are at the height of vectors almost tangent to~$e_c$. Therefore they contribute to~$(iii)$. 
Similarly, if~$e_c$ is on the left of $\Rib_{i,j}$, then the arcs of $\gamma^l\cap\Rib^l_{i,j}$ travel along $\un e_c$, and contribute to~$(iv)$. In the same vein, we obtain the two others terms of~$(iii)$ and $(iv)$ by recalling that~$\un e_c$ identified with~$\un e_{\cc c}$, so that, if $\cc c$ differs from both~$i$ and $j$, the arcs of $\gamma\cap\Rib_{i,j}$ also project on~$\un e_{\cc c}$ in the same way.

Suppose now~$c=j$. Then the arcs of $\gamma^l\cap\Rib^l_{i,j}$ finish their travel by connecting $\un e^l_c$ to~$\BS_{c,\cc c}$, and the arcs of $\gamma^r\cap\Rib^r_{i,j}$ connect $\un e^r_c$ to~$\BS_{c,\cc c}$. Thus they all contribute to~$(i)$. 
Similarly if $\cc c=i$, then the arcs of $\gamma^l\cap\Rib^l_{i,j}$ begin their travel by connecting $\BS_{c,\cc c}$ to~$\un e^r_{\cc c} = \un e^l_{c}$ to, and the arcs of $\gamma^r\cap\Rib^r_{i,j}$ connect $\BS_{c,\cc c}$ to~$\un e^r_{c}$, thus all contributing~$(i)$.
Similarly, we get the expression for~$(ii)$ by considering the cases~$c=i$ and $\cc c=j$.
\end{proof}

As the collection~$\gamma$ is null-homologous, the collection~$\gamma_\pi$ is also null-homologous, so that for every side~$e_c$ of~$\Polg$, the number of arcs of~$\gamma_\pi$ that travel along $e_c$ in one direction is equal to the number of arcs in the other direction. This implies that the numbers of arcs given by Lemma~\ref{L:CompoGPi}~$(iii)$ and $(iv)$ are equal. We then define~$\Delta_c$ to be their common value, which then admits the more symmetric expression 
$$\frac 1 4(\sum_{i<c<j<i} b_{i,j} + \sum_{i<j<\cc c<i} b_{i,j} + \sum_{i<j<c<i} b_{i,j} + \sum_{i<\cc c<j<i} b_{i,j}),$$
or simply~$\frac 1 4(\sum_{i,j\neq c}b_{i,j} + \sum_{i,j\neq \cc c}b_{i,j})$. Also, since every arc of~$\gamma$ that arrives on~$\BS_{c, \cc c}$ is followed by an arc that leaves~$\BS_{c, \cc c}$, the numbers $\sum_{i\neq \cc c} b_{i,\cc c}$ and $\sum_{j\neq c} b_{\cc c,j}$ are equal. Hence it is possible to choose a 2-chain in~$\un e_c$ whose boundary is~$\gamma_\pi\cap \un e_c$, plus some arcs in the fibers~$\un e^l_c$ and~$\un e^r_c$. This 2-chain is unique up to adding multiples of~$\un e_c$, so that we can make a specific choice that will be convenient for estimating the contributions of the last two components~$\Ss$ and~$\Sss$ of~$S^\gamma$.

\begin{defi}
\label{D:Sd}
(See Figure~\ref{F:Sd}.) 
With the above notation, we define $\Sd$ to be the 2-chain consisting, for every side~$c$ of~$\Polg$, of $\Delta_c$ cells in~$\un e_c$ whose oriented boundary consists of the $\Delta_c$~arcs of~$\gamma_\pi$ that join~$\un e^l_c$ to~$\un e^r_c$, plus the $\Delta_c$~arcs of~$\gamma_\pi$ that join~$\un e^r_c$ to~$\un e^l_c$, plus $\Delta_c/2$~arcs that go up and $\Delta_c/2$~arcs that go down in the fiber~$\un e^l_c$,  plus $\Delta_c/2$~arcs that go up and $\Delta_c/2$~arcs that go down in the fiber~$\un e^r_c$.
\end{defi}

\begin{lemma}
\label{L:LkSd}
The intersection number between the collection~$\gamma'$ and the rational 2-chain~$\Sd$ is equal to $$-\sum_{i,j,k,l} \frac 1 8(\{i\neq k \mbox{ and }  j\neq k\}+\{i\neq \cc k \mbox{ and } j\neq \cc k\}) \, b_{i,j}b'_{k,l}.$$
\end{lemma}

\begin{proof}
The collection~$\gamma'$ intersects~$\Sd$ only on branching segments. Figure~\ref{F:Sd} then shows that all intersection points have negative sign. For every side~$e_k$ of~$\Polg$, there are~$\sum_{l\neq k} b'_{k,l}$ arcs of~$\gamma'$ that cross~$\un e_k$ in each direction. Every such arc then intersects~$\Delta_k/2$ cells of~$\Sd$  negatively, so that the total contribution of~$\un e_k$ to the intersection number is $\Delta_k\sum_{l\neq k} b'_{k,l}$. Therefore the total intersection number is the sum over all sides~$e_k$ of~$\Polg $ of the terms $\Delta_k\sum_{l\neq k} b'_{k,l}$. As the sides $e_k$ and~$e_{\cc k}$ coincide, the latter sum admits the more symmetric expression $\frac 1 2 \sum_{k} \Delta_k\sum_{l\neq k} b'_{k,l}$. We then find the expected value by expanding~$\Delta_k$.
\end{proof}

The boundary of the 2-chain~$\Spi\cup\Sd$ that was constructed above is $\gamma$, plus some multiples of the fibers~$\un V_0$ and~$\un V_1$ that we now determine. 

\begin{lemma}
\label{L:ToursFibres}
Assume that~$\Rib_{i,j}$ is a ribbon of~$\temp_{4g+2}$. Then

$(i)$ the part of~$\Rib_{i,j}^l$ that lies in the neighbourhood of~$\un e_i^l$ goes down by a height $(2g{-}3)\pi/(4g{+}2)$,  the part of~$\Rib_{i,j}^l$ that lies in the neighbourhood of~$\un e_j^r$ goes down by a height $(2g{-}3)\pi/(4g{+}2)$,  except if $j=i+1$, in which case the part of~$\Rib_{i,j}^l$ that lies in the neighbourhood of~$\un e_i^l=\un e_j^r$ goes up by a height~ $4\pi/(4g{+}2)$, 

$(ii)$ the part of~$\Rib_{i,j}^l$ that lies in the neighbourhood of~$\un e_c^l$, for $
j+1<c<i$, goes down by a height $(4g{-}2)\pi/(4g{+}2)$, 

$(iii)$ the part of~$\Rib_{i,j}^r$ that lies in the neighbourhood of~$\un e_i^r$ goes up by a height $(2g{-}3)\pi/(4g{+}2)$, the part of~$\Rib_{i,j}^r$ that lies in the neighbourhood of~$\un e_j^l$ goes up by a height $(2g{-}3)\pi/(4g{+}2)$, except if $j=i-1$, in which case the part of~$\Rib_{i,j}^r$ that lies in the neighbourhood of~$\un e_i^r=\un e_j^l$ goes up by a height~$4\pi/(4g{+}2)$, 

$(iv)$ the part of~$\Rib_{i,j}^r$ that lies in the neighbourhood of~$\un e_c^r$, for $
i<c<j-1$, goes up by a height $4\pi/(4g{+}2)$. 
\end{lemma}

\begin{proof}
The proof is illustrated on Figures~\ref{F:ProjTemplate} and~\ref{F:Fiber}. It relies on the assumption that the angle between adjacent sides of~$\Polg$ is~$2\pi/(2g{+}1)$, and on the height we chose for the parts of the ribbons~$\Rib_{i,j}^l$ and~$\Rib_{i,j}^r$ above each edge of~$\Polgint$. The values follow from the equalities~$\pi/2-2\pi/(2g{+}1)=(2g{-}3)\pi/(4g{+}2)$ and $\pi-2\pi/(2g{+}1)=(4g{-}2)\pi/(4g{+}2)$. 
\end{proof}

\begin{defi}
\label{D:v}
For~$i,j$ in the range~$\{1, \dots, 4g{+}2\}$, we define~$\vzero i j$ to be the sum over all even vertices of~$\Polg$ of the increases of~$\Rib^l_{i,j}$ and of~$\Rib^r_{i,j}$ around this vertex. 
\end{defi}

 For example, if $i=1$ and $j=2$, then $\Rib^l_{1,2}$ contributes~$4\pi/(4g{+}2)$ to $\vzero 0 1$ and~$0$ to~$\vun 0 1$.  On the other hand, $\Rib^r_{1,2}$ contributes $2g$ times $+(4g{-}2)\pi/(4g{+}2)$ to~$\vzero 0 1$ and $2$ times~$(2g{-}3)\pi/(4g{+}2)$ plus $2g{-}1$ times~$(4g{-}2)\pi/(4g{+}2)$ to~$\vun 0 1$. Therefore we have~$\vzero 0 1 = (8g^2-4g+4)\pi/(4g{+}2)$ and~$\vun 0 1 = (8g^2-4g-4)\pi/(4g{+}2)$. 
With the above notation, the boundary of the 2-chain $\Spi\cup\Sd$ consists of the union of $\gamma$, of $-\sum_{i,j} \vzero i j b_{i,j}$ times the fiber~$\un V_0$ and of~$-\sum_{i,j} \vun i j b_{i,j}$ times the fiber~$\un V_1$.

\begin{figure}[hbt]
	\begin{center}
	\begin{picture}(110,65)(0,0)
 	\put(-2,-5){\includegraphics*[scale=.55]{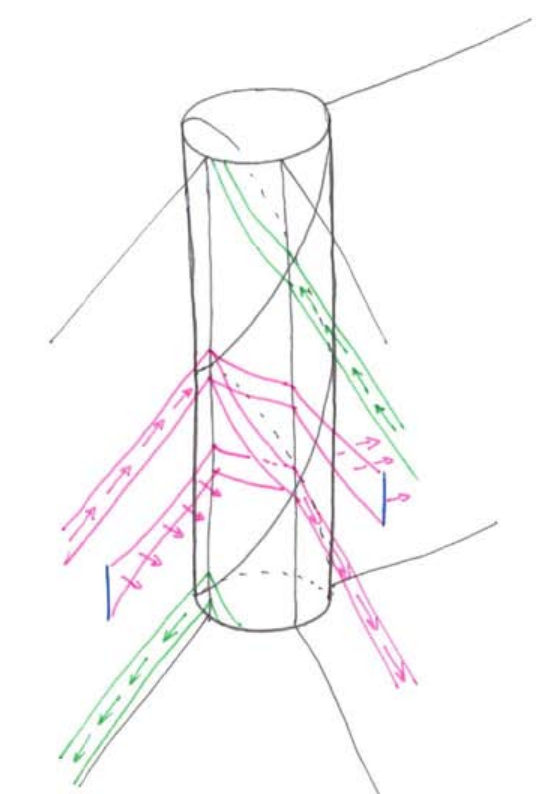}}
 	\put(60,10){\includegraphics*[scale=.85]{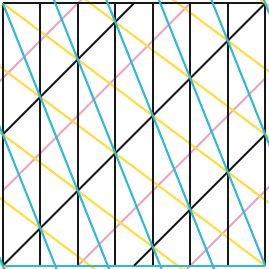}}
	\end{picture}
	\end{center}
	\caption{\small The templates~$\temp_{4g+2}^l$ and~$\temp_{4g+2}^r$ on the neighourhood of a vertex~$V_i$. For~$\temp_{4g+2}^l$, most of the ribbons go down by $(4g{-}2)\pi/(4g{+}2)$ (in blue), or they are close to a branching segment and they go down by $(2g{-}3)\pi/(4g{+}2)$ (in yellow), unless they are close to both branching segments and they go up by~$4\pi/(4g{+}2)$ (in pink). For~$\temp_{4g+2}^r$, the signs are reversed.}
	\label{F:Fiber}
\end{figure}

In order to complete the chain~$S^\gamma$, it suffices that we add a 2-chain whose boundary is $(\sum_{i,j} \vzero i j b_{i,j})\un V_0$ and a 2-chain whose boundary is $(\sum_{i,j} \vun i j b_{i,j})\un V_1$. Since~$\Sigma_g$ has Euler characteristic~$2{-}2g$, there exists a vector field on~$\Sigma_g$ with only one singularity at~$V_0$, the index of the latter being~$2{-}2g$. By lifting this vector field in~$\un\Sigma_g$, we obtain a surface with boundary $(2g{-}2)\un V_0$. We then define~$\Ss$ to be $\frac 1 {2g-2} (\sum_{i,j} \vzero i j b_{i,j})$ times this surface. Similarly, we can construct a surface with boundary~$(2g{-}2)\un V_1$, and we then define~$\Sss$ to be $\frac 1 {2g-2} (\sum_{i,j} \vun i j b_{i,j})$ times the latter surface. We have now only to determine the intersection number of~$\gamma'$ with both~$\Ss$ and~$\Sss$. For this it is enough to determine the linking number of~$\gamma'$ with the fibers~$\un V_0$ and~$\un V_1$, and then to multiply by $\frac 1 {2g-2} (\sum_{i,j} \vzero i j b_{i,j})$ and $\frac 1 {2g-2} (\sum_{i,j} \vun i j b_{i,j})$ respectively.

\begin{defi}
\label{D:h}
For~$i,j$ in the range~$\{1,\dots,4g{+}2\}$, we define~$\hzero i j$ as the number of even vertices of~$\Polg$ on the left of~$\Rib_{i,j}$, minus the number of even vertices on the right of~$\Rib_{i,j}$, divided by~$2g{+}1$. Similarly, we define $\hun i j$ as the number of odd vertices of~$\Polg$ on the left of~$\Rib_{i,j}$, minus the number of odd vertices on the right, divided by~$2g{+}1$. 
\end{defi}

The precise expressions for~$\hzero i j$ and $\hun i j$ are $((j-i)[4g{+}2]-(2g{+}1)+j[2]-i[2])/(2g{+}1)$ and $((j-i)[4g{+}2]-(2g{+}1)-j[2]+i[2])/(2g{+}1)$, respectively. Moreover, we fix two arbitrary points~$V'_0$ and~$V'_1$ on the boundaries of~$D_0$ and~$D_1$ respectively. We also choose two meridians~$m_0$ and~$m_1$ of the solid tori~$\un D_0$ and~$\un D_1$.

\begin{lemma}
\label{L:Cohomologous}
The collection~$\gamma'$ is cohomologous, in the complement of~$\un V_0 \cup \un V_1$, to $(\sum_{k,l}\vzero k l b'_{k,l}) \un V'_0 + (\sum_{k,l}\vun k l b'_{k,l}) \un V'_1 + (\sum_{k,l}\hzero k l b'_{k,l}) m_0 + (\sum_{k,l}\hun k l b'_{k,l}) m_1$.
\end{lemma}

\begin{proof}
A construction similar to the construction of the 2-chain~$\Spi\cup\Sd$, applied to~$\gamma'$ instead of~$\gamma$, realizes a cobordism between~$\gamma$ and the announced collection of curves.
\end{proof}

\begin{lemma}
\label{L:LkSs}
The intersection number between~$\Ss\cup\Sss$ and~$\gamma'$ is equal to 
\[ \sum_{i,j,k,l} \left[\vzero i j \hzero k l + \vun i j \hun k l + \frac 1 {2g{+}1} (\vzero i j + \vun i j) (\vzero k l + \vun k l)\right] b_{i,j}b'_{k,l}.\]
\end{lemma}

\begin{proof}
The curve~$m_0$ bounds a meridian disc for~$\un D_0$, so that its linking numbers with~$\un V_0$ and~$\un V_1$ are 0 and 1 respectively. Similarly one has~$\lk(m_1, \un V_0)= 0$ and $\lk(m_1, \un V_1)= 1$. The lift of the vector field on~$\Sigma_3$ with only one singularity~$p$ defines a surface in~$\un\Sigma_g$ whose boundary is~$(2g{-}2)\un p$ and which intersects every other fiber once. Therefore, $\lk(\un p, \un p') = \frac 1 {2g{-}2}$ holds for every point~$p'$ distinct from~$p$. 
\end{proof}

Proposition~\ref{P:Lk3} now follows from Lemmas~\ref{L:LkSpi},~\ref{L:LkSd}, and~\ref{L:LkSs}, which together give the expected bound for $\lk(\gamma, \gamma')$ directly.

The set of linear codes  $(b_{i,j})$ that correspond to geodesics on~$\Sigma_g$ is a subset of~$\R^{(4g+2)(4g+1)}$. Actually, it is a cone included in~$\R_+^{(4g+2)(4g+1)}$ that we denote by~$\Cone_g$. It is not hard to see that~$\Cone_g$ is a proper subset of~$\R_+^{(4g+2)(4g+1)}$, {\it i.e.}, that there are more constraints on the possible values of~$(b_{i,j})$ than the positivity of the coefficients. For example, there are linear equality constraints coming from the fact that every arc of the associated collection that crosses a side of~$\Polg$ continues on the other side, as well as linear equalities coming from the fact the the collection is null-homologous. There are also inequality constraints coming from the fact that the collection consists of geodesics, so that it cannot always wind around a vertex. Precisely, some coefficients of the form $b_{i,j}$ with $\vert i-j\vert \ge 2$ cannot be too small when compared with the coefficients of the form~$b_{i,i+1}$. 

Implementing the above constraints in a computer program leads to numerical bounds for the linking numbers of orbits of~$\GF{\Sigma_g}$. However, as we shall see in Section~\ref{S:Question}, some collections of orbits have a positive linking number, so there is no hope to prove a uniform negativity result.

\subsection{Linking of geodesics on the orbifolds~$\Sigma_{2,3,4g+2}$}
\label{S:Lk237}

We now consider the case of the orbifold~$\Orbg$. Our goal is to establish upper bounds for the linking numbers of pairs of orbits in the associated geodesic flow. We shall prove

\begin{prop}[case $(b)$ of Theorem~A]
\label{P:237}
Let $\gamma, \gamma'$ be two orbits of~$\GF{\Orb_{2,3,4g+2}}$ in~$\un\Orb_{2,3,4g+2}$. Then we have $\lk(\gamma, \gamma') < 0$.
\end{prop}

\begin{figure}[hbt]
	\begin{center}
	\includegraphics*[scale=.55]{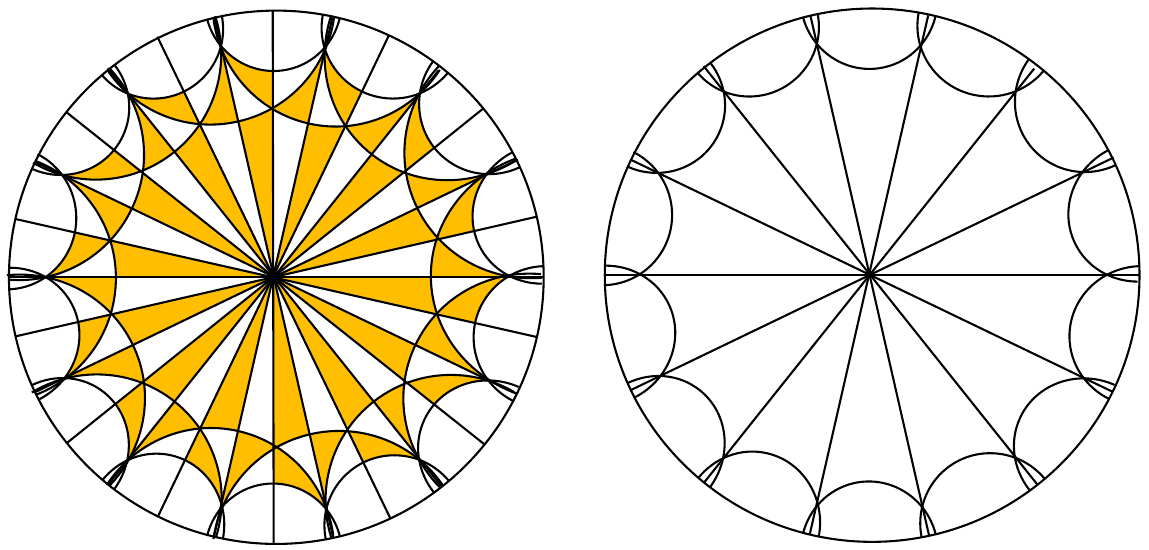}
	\end{center}
	\caption{\small On the left, the covering of~$\Orbg$ by genus~$g$ surface, for $g=3$. A fundamental domain for~$\Orbg$ is obtained by taking the union of any white triangle with a neighbouring orange triangle. On the right the intermediate tiling~$\Tiling_{4g+2}$ of the hyperbolic plane by equilateral triangles with angles~$\pi/({2g{+}1})$. Each triangle is a 3-fold cover of~$\Orbg$. The polygon~$\Polg$ (and therefore the surface~$\Sigma_g$) is obtained by gluing $4g{+}2$ triangles that are adjacent to a~vertex.}
	\label{F:Tessel14}
\end{figure}

The proof relies on a a more precise study of the template~$\temp_{4g+2}$ and refinement of Proposition~\ref{P:Lk3}. The starting point is that~$\Orb_{2,3,4g+2}$ admits a covering of index~$3(4g{+}2)$ by a genus $g$ surface~$\Sigma_g$ obtained by identifying sides of a regular $4g{+}2$-gon (see Figure~\ref{F:Tessel14}). So, by the behaviour under quotient of the linking number (Lemma~\ref{L:Covering}), in order to establish Proposition~\ref{P:237}, it is enough to prove that $\Gamma_{2,3,4g+2}$-invariant geodesics of~$\Sigma_g$ have a negative linking number. These $\Gamma_{2,3,4g+2}$-invariant geodesics have three advantages that are needed in the proof. First, their symmetry properties allows to use reduced linear codes with $4g{+}1$ coordinates instead of $(4g{+}2)(4g{+}1)$, thus also simplifying the matrix~$Q_{4g+2}$ bounding the linking number to a more simple $(4g{+}1)\times(4g{+}1)$ matrix (Lemma~\ref{L:AReduced}). Second, it is possible to refine the bounds on the linking number by refining the intersection number between the 2-chain~$\Spi$ and the curve~$\gamma'$, thus refining the first term in Definition~\ref{D:BoundLk}. The price to pay is to add $2g$ coordinates to the reduced linear code that describes how many consective times the family takes the rightmost and leftmost ribbons of the template. These two first steps then associate to every collection of $\Gamma_{2,3,4g+2}$-invariant geodesics a reduced linear code with $6g{+}1$ coordinates, so that the linking number between two collections in bounded by a bilinear form~$S_{4g+2}$ in the reduced linear code. Third, we determine a cone $\Cone_{2,3,4g+2}$ in~$\R^{6g+1}$ that (strictly) contains all reduced linear codes, and whose extremal rays are easy to determine. The proof of Proposition~\ref{P:237} then consists in proving that the form $S_{4g+2}$ is negative on all pairs of extremal rays of~$\Cone_{2,3,4g+2}$.
 

So, let $\gamma, \gamma'$ be two orbits of~$\GF{\Orbg}$. Let $\hat \gamma, \hat\gamma'$ be the images in the template~$\temp_{4g+2}$ of the $\Gamma_{2,3,4g+2}$-invariant lifts of~$\gamma$ and~$\gamma'$ in~$\un\Sigma_g$. Denote by~$b_{i,j}$ and $b'_{i,j}$ their linear codes, as defined in Definition~\ref{D:LinearCode}. Since the collection~$\hat\gamma$ is invariant under an order~$4g{+}2$ rotation around the center of~$\Polg$, we have $b_{i,j} = b_{i+1,j+1}$ for every $i,j$. Therefore, one can consider a simpler code~$\hat b_{i,j}$ defined for~$j=1, \dots, 4g{+}1$ by $\hat b_{j} = \sum_{i=0, \dots, 4g+1} b_{i,i+j}$. Similarly, we introduce a reduced form~$\hat Q_{2,3,4g+2}$ on~$\R^{4g+1}$ whose coefficients $\hat q_{j,l}$ are defined by $\hat q_{j,l} = \sum_{i,k=0, \dots, 4g+1}q_{i,i+j,k,k+l}$.

\begin{lemma}
\label{L:AReduced}
With the above definitions, for $j,l = 1, \dots, 4g{+}1$, we have 
\begin{equation}
\label{Eq:LkSym}
\hat q_{j,l} = (2g+1)\vert j - l \vert - 2g(2g+1) +\frac 1 {2g{-}2} (j-2g-1) (l-2g-1).
\end{equation}
\end{lemma}

\begin{proof}
We start from the formula for~$q_{i,j,k,l}$ given by Definition~\ref{D:BoundLk} with replacing~$j$ by $i+j$ and $l$ by~$k+l$. The first term $-\frac 1 2 (\{i<k<k+l\le i+j\} + \{k<i<i+j\le k+l\})$ equals~$\frac 1 2$ if the two chords connecting the edges~$e_i$ to $e_{i+j}$, and $e_k$ to $e_{k+l}$ do not intersect and are parallel, or if they have a common head (see Figure~\ref{F:Chord}). When the differences~$j$ and $l$ are fixed, they are~$4g+2$ possible choices for the first chord, and then there are $\vert j-l\vert$ positions for the second chord that give an admissible position. This gives the first term of Equation~\eqref{Eq:LkSym}.

\begin{figure}[hbt]
	\begin{center}
	\begin{picture}(60,60)(0,0)
 	\put(0,0){\includegraphics*[scale=.8]{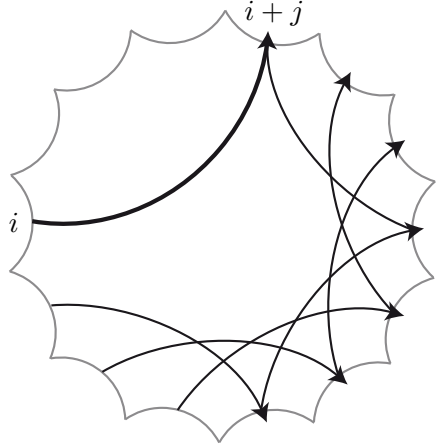}}
	\put(1,29){$i$}
	\put(32,57){$i+j$}
	\end{picture}
	\end{center}
	\caption{\small Once the chord connecting $e_i$ to $e_{i+j}$ is chosen (with $j=4$ on the picture), there are $\vert l-j\vert$ choices for $k$, so that the chord connecting $e_k$ to $e_{k+l}$ contributes to~$\{i < k < k{+}l\le i{+}j\}$ or to $\{k<i<i{+}j\le k{+}l\}$ (with $l=11$ on the picture).} 
	\label{F:Chord}
\end{figure}

For the second term, we note that when $i,i{+}j$ are fixed, there are $4g$ values of~$k$ that add~$1/8$ to the sum, and $4g$ values of~$\bar k$ that also add~$1/8$. So this yields a contribution of~$g$ when $i$ is fixed. By summing over all~$i$, we obtain the second term.

The third term in Definition~\ref{D:BoundLk} depends of the parity of~$i,i{+}j,k,k{+}l$, because we are considering the rotation amount of the chord with respect to the two different vertices of~$\Polg$. When summed over all $i, j$, these two rotation amounts are equal, so that we only consider the mean rotation of the chords. These are equal to~$j-2g-1$ and $l-2g-1$ respectively. Then the contribution to~$\hat q_{j,l}$ is a multiple of~$(j-2g-1)(l-2g-1)$. The constant is given by Lemma~\ref{L:ToursFibres}.
\end{proof}

The symmetry of the families $\hat\gamma, \hat\gamma'$ now allows us to refine Lemma~\ref{L:LkSpi}, at the expense of expanding the code. The idea is that if several consecutive arcs of~$\hat \gamma$ all travel along the rightmost ribbon, then they cannot cross as many arcs of~$\hat\gamma'$ as the bound (and the proof) of Lemma~\ref{L:LkSpi} suggests. For the sequel, it is important to remember that the families~$\hat\gamma$ and $\hat\gamma'$ are invariant by a rotation of order~$4g{+}2$ of~$\Polg$.

\begin{defi}
\label{D:ReducedLinearCode}
For $m=1,\dots,2g$, let $c_m$ ({\it resp.} $d_m$) denote the number of arcs of~$\hat\gamma$ that travel exactly $m$ consecutive times along~$\Rib_{0,1}$ ({\it resp.} $\Rib_{0,4g+1}$). Define~$c'_m$ and $d'_m$ similarly. The set $((b_j)_{j=1, \dots, 4g+1}, (c_m)_{m=1, \dots, 2g}, (d_m)_{m=1, \dots, 2g})$ is called the \emph{linear reduced code} of~$\hat\gamma$.

For $m, n = 1, \dots, 2g$, define $r_{m,n}$ to be ${-}2g{-}1$ if $i=j\neq 1$ and $0$ otherwise. Let $R_{4g+2}$ denote the bilinear form on~$\R^{2g}$ with coefficients $r_{m,n}$, and let $S_{4g+2}$ denote the bilinear form on~$\R^{8g+1}$ which is the direct sum $\hat Q_{4g+2}\oplus R_{4g+2} \oplus R_{4g+2}$.
\end{defi}

Note that if an arc travels~$m$ consecutive times along a ribbon, then it travels~$m-1$ times at it next move. Thus we have~$c_{m-1}\ge c_{m}$ and~$d_{m-1}\ge d_{m}$ for every~$m$. Note also that some orbits of the template could travel more than~$g$ times along the leftmost ribbon, thus making more than one half-turn around the corresponding vertex of~$\Polg$. These orbits do not interest us, since they cannot come from geodesics. 

\begin{lemma}
\label{L:LkSpiBis}
With the above notation, the intersection number between~$\hat\gamma'$ and the 2-chain~$\Spi$ is at most $\sum_{j,l=1}^{4g+1} (2g+1)\vert j - l \vert  b_{j}b'_{l} - (2g{+}1)\sum_{m=2}^{g} (c_m c'_{m} + d_m d'_{m})$. The linking number~$\lk(\hat\gamma, \hat\gamma')$ is smaller than $S_{4g+2}\big( ((b_j), (c_m), (d_m)), ((b_j), (c_m), (d_m)) \big).$
\end{lemma}

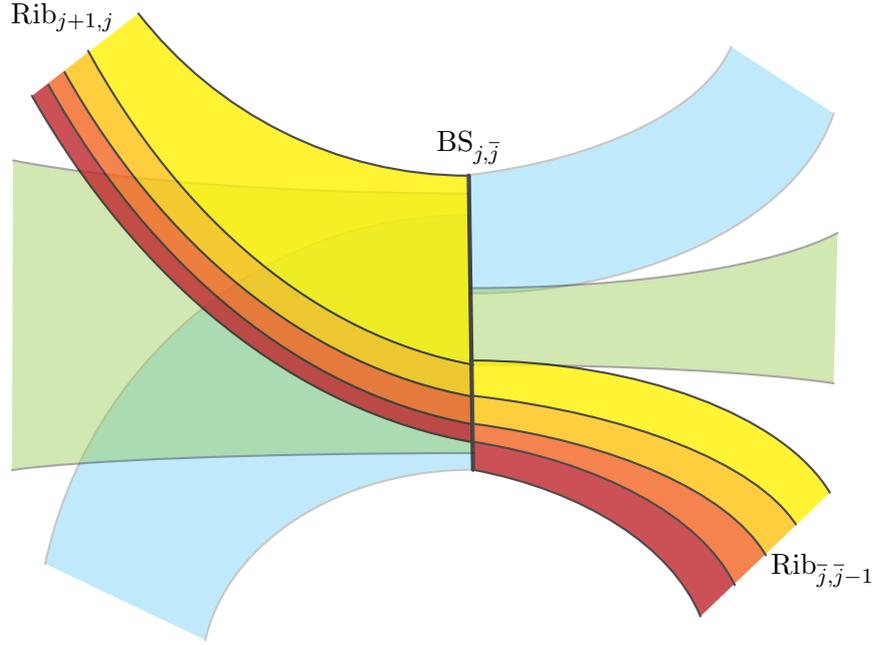
\begin{figure}[hbt]
	\begin{center}
	\begin{picture}(110,85)(0,0)
 	\put(0,0){\includegraphics*[scale=.8]{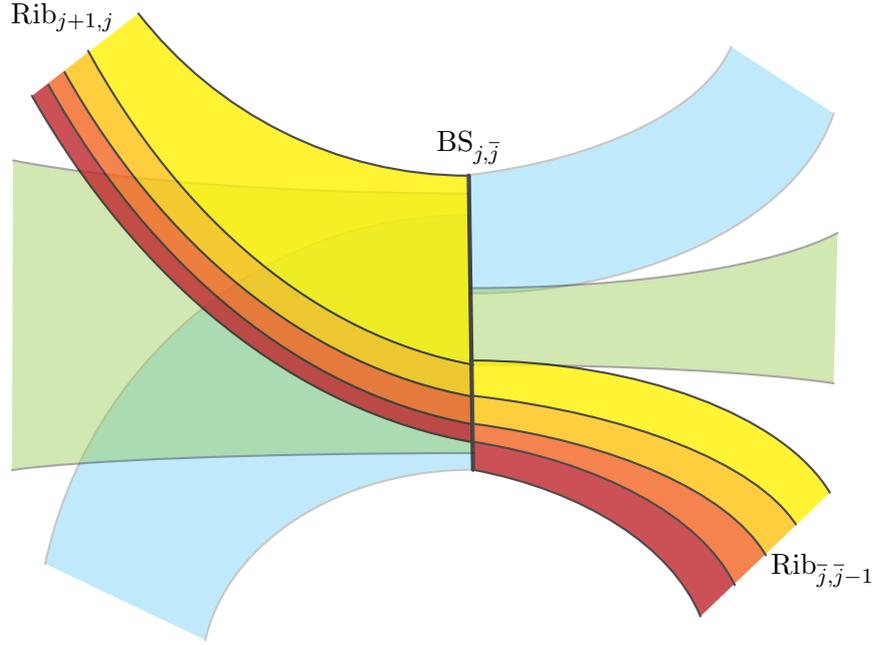}}
	\put(57,66){$\BS_{j,\cc j}$}
	\put(1,83){$\Rib_{j+1, j}$}
	\put(101,10){$\Rib_{\cc j, \cc j-1}$}
	\end{picture}
	\end{center}
	\caption{\small A neighbourhood of a branching segment~$\BS_{j,\cc j}$ in~$\un\Sigma_g$. Both ribbons~$\Rib_{j+1,j}$ and $\Rib_{\cc j, \cc j-1}$ are subdivided into subribbons containing arcs that travel $1, 2, 3, \dots$ consecutive times respectively along the rightmost ribbon. For a family of orbits of the template that is invariant by rotation of~$\Polg$, the arcs on $\Rib_{j+1,j}$ and on $\Rib_{\cc j, \cc j-1}$ are in one-to-one correspondance.} 
	\label{F:SpiBis}
\end{figure}

\begin{proof} 
(See Figure~\ref{F:SpiBis}.) We use the notation introduced in the proof of Lemma~\ref{L:LkSpi}. In the second case of this proof ($i\neq k, j=l$), we assumed that there was always an intersection between the considered arc~$\alpha'$ on~$\Rib_{k,l}$ and any elementary piece~$d_\alpha^l$ or~$d_\alpha^r$ of the 2-chain~$\Spi$. Actually, if $\alpha'$ is an arc that lies in the rightmost ribbon~$\Rib_{j+1,j}$ there is an intersection point with~$d_\alpha^l$ or~$d_\alpha^r$ if and only if $\alpha$ does not lie in~$\Rib_{j+1,j}$ and the head of~$\alpha'$ in the vertical branching segment~$\BS_{j, \cc j}$ is under the head of~$\alpha$. In particular, we know that there are~$c'_2$ elementary arcs of~$\hat\gamma'$ in~$\Rib_{j+1,j}$ whose heads are above all tails of arcs of~$\hat\gamma$ that will travel more than one time along the rightmost ribbon. Indeed, if an arc travels two or more times along the rightmost ribbon, then its direction at infinity is on the right of the direction of an arc travels only once on the rightmost ribbon (see Figure~\ref{F:SpiBis}). Since there are at least~$d_2$ such arcs in~$\hat\gamma$ at each branching segment, we can add a term~$-\frac{4g+2}2 d_2d'_2$ to the previous bound on the intersection number between~$\Spi$ and~$\hat\gamma'$. Similarly, we can consider the $d'_3$ arcs of~$\hat\gamma'$ that reach~$\BS_{j,\cc j}$ along the rightmost arc and that will travel along it two more times. Their heads cannot be above the tails of the~$d_3$ arcs of~$\hat\gamma$ that arrive at~$\BS_{j,\cc j}$ from a different ribbon and that travel two or more times along the rightmost ribbon. At the end, we can then add a term~$-(2g+1) (d_2d'_2 +d_3d'_3+\dots+d_gd'_g)$. Considering also the leftmost ribbons gives the announced extra-term.

The formula for total linking number then follows by replacing the first term in Equation~\eqref{Eq:LkSym} by the above one.
\end{proof}

The goal is now to bound the value of the quadratic form~$S_{4g+2}$ on the set of linear reduced codes that come from geodesics of~$\Orbg$. In order to do this, we first determine a cone in~$\R^{6g+1}$ that contains the set of linear reduced codes. 

\begin{defi}
\label{D:Cone}
For $x,y$ in~$\{1, \dots, 2g\}$, let $V_{x,y}$ be the vector in~$\R^{4g+1}\oplus \R^{g} \oplus \R^{g}$ with coordinates $((x-1, 0, \dots, 0, 1, 0, \dots, 0, 1, 0, \dots, 0, y-1),(2, \dots, 2, (1), 0, \dots, 0),(2, \dots, 2, (1),  0, \dots))$, where the two $1$ in the first block are in position $y+1$ and $4g-x+1$, where there are $\lfloor \frac{x-1}2\rfloor$ coefficients $2$ in the second block, one $1$ if $x$ is even, and there are $\lfloor \frac{y-1}2\rfloor$ coefficients $2$ in the last block, and one $1$ if $y$ is even.

Let $C_{4g+2}$ be the conway hull in~$\R^{6g+1}$ of the rays generated by the $4g^2$ vectors~$V_{x,y}$.
\end{defi}

\begin{lemma}
\label{L:Cone}
With the above definition, the reduced linear code of every collection of $\Gamma_{2,3,4g+2}$-periodic geodesics belongs to~$C_{4g+2}\setminus\{0\}$. 
\end{lemma}
  
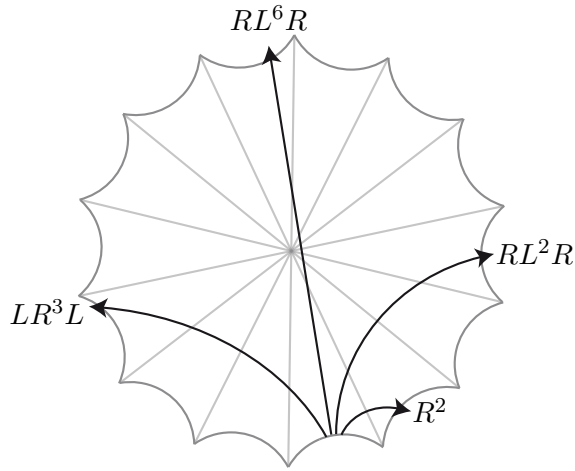
\begin{figure}[hbt]
	\begin{center}
	\begin{picture}(60,60)(0,0)
 	\put(0,0){\includegraphics*[scale=.8]{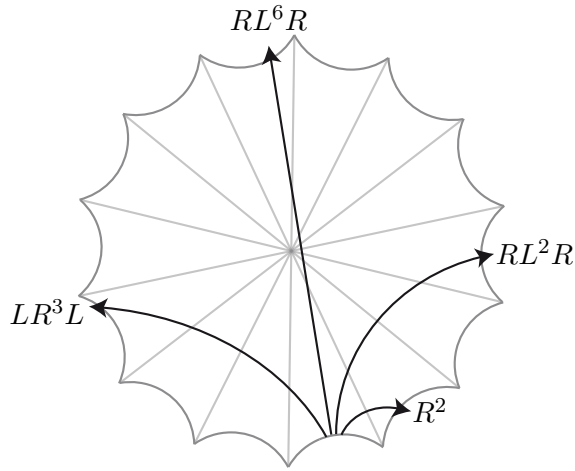}}
	\put(56,28){$RL^2R$}
	\put(21,59){$RL^6R$}
	\put(45,7){$R^2$}
	\put(-8,20){$LR^3L$}
	\end{picture}
	\end{center}
	\caption{\small Some dynamical codes. The lenght of the code equals the number of crossed triangles.} 
	\label{F:DynCode}
\end{figure}

\begin{proof}
(See Figure~\ref{F:DynCode}.)
Let~$\Tiling_{4g+2}$ denote the $\Gamma_{2,3,4g+2}$-invariant tessellation of~$\Hy^2$ by equilateral triangles with angles~$\frac{2\pi}{4g{+}2}$. Note that a fundamental domain for the action of $\Gamma_{2,3,4g+2}$ on~$\Hy^2$ is given by a third a tile of~$\Tiling_{4g+2}$. Note also that by considering the~$4g{+}2$ triangular tiles that are adjacent to a given vertex, we obtain a fundamental domain~$\Polg$ for the surface~$\Sigma_g$. 

As before, let~$\hat\gamma$ be a $\Gamma_{2,3,4g+2}$-periodic geodesics, considered in~$\Hy^2$. 
We associate to it a~\emph{dynamical code} in the following way. Starting from an arbitrary intersection point between~$\hat\gamma$ and an edge of~$\Tiling_{4g+2}$, we follow the geodesics~$\hat\gamma$. Everytime we cross a triangle of~$\Tiling_{4g+2}$, we add a letter $L$ to the dynamical code if $\hat\gamma$ goes to the left in this triangle, of a letter~$R$ if it goes to the right. Of course we stop after one period. At the expense of a cyclic permutation, the dynamical code can then be assumed to be of the form~$L^{x_1}R^{y_1}L^{x_2}\dots R^{y_n}$.

The key-point is that $1\le x_k\le 2g$ and~$1\le y_k\le 2g$ hold for every~$k$. Indeed, a curve that go more than~$2g$ consecutive times on the left crosses one of the geodesics containing edges of the tiling more than once, and therefore it cannot be a geodesics.

The second point is that the linear reduced code depends linearly of the exponents~$x_k, y_k$ in an explicit way. Indeed, every arc~$\alpha$ of~$\hat\gamma$ in~$\Polg$ is associated to a unique position in the dynamical code that describes the dynamical code when starting at the tail of~$\alpha$. Conversely, to every position in the dynamical code are associated $4g+2$ arcs of~$\hat\gamma$ that are obtained one from another by a rotation about the center of~$\Polg$.

Now, if an arc $\alpha$ goes from an edge $e_i$ to the edge~$e_{i+j}$ in~$\Polg$ with~$2\le j\le 2g{+}1$, then the corresponding dynamical code is~$LR^{j-1}L$, while the linear reduced code contains only a $1$ in~$j{-}1$-st position. Similarly if an arc goes from~$e_i$ to~$e_{i+j}$ with~$2g{+}1\le j\le 4g$, then the dynamical code is~$RL^{4g+1-j}R$ and the linear reduced code contains only a $1$ in~$j{-}1$-st position. (There is an ambiguity in the case~$j=2g+1$ for the dynamical code, depending on which side of the center of~$\Polg$ the geodesics go. But both give rise to the same the linear reduced code, so that this ambiguity is of no consequence.) 

In the remaining case, if an arc goes from $e_i$ to $e_{i+1}$, then the dynamical code begins with~$L^2$, and the linear reduced code begins with a $1$ in $1$-st position. However, the second block of coordinates (that corresponds to the vector $(c_m)$) can be non-zero, depending on how many times the geodesics~$\hat\gamma$ will go on the left after escaping~$\Polg$. The point here is that the dynamical code actually begins with~$L^xR$, and the number of times that $\hat\gamma$ will travel along the leftmost ribbon is $\lfloor \frac {x} 2\rfloor$. Therefore the second block there contains a $1$ in $\lfloor \frac x 2\rfloor$-th position. The case~$j=4g+1$ is treated similarly.

Finally, we truncate the dynamical code of~$\hat\gamma$ into the $n$~blocks $RL^{x_1}R^{y_1-1}$, $RL^{x_2}R^{y_2-1}$, \dots, $RL^{x_n}R^{y_n-1}$. The linear reduced code that corresponds to a block~$RL^{x_k}R^{y_k-1}$ is the sum of the linear codes corresponding to each of the $x_k+y_k$ letters, which turns out to be~$V_{x_k, y_k}$ by the above discussion. Therefore the linear code asociated to~$\hat\gamma$ is the sum of~$n$ such vectors. Thus it belongs to~$C_{4g+2}\setminus \{0\}$.
\end{proof}

\begin{lemma}
\label{L:FormNegative}
The form~$S_{4g+2}$ is negative on~$C_{4g+2}\setminus\{0\}$.
\end{lemma}
 
\begin{proof}
The form~$S_{4g+2}$ is bilinear, so that it is enough to show that it is negative when evaluated on every pair~$(V_{x,y}, V_{x',y'})$ of extremal vectors. Now we note that~$S_{4g+2}(V_{x,y}, V_{x',y'})$ is a function of~$x,x',y,y'$ which is almost linear in each of the coordinates. Indeed, we can expend Equation~\eqref{Eq:LkSym} corrected with the term of Lemma~\ref{L:LkSpiBis} and obtain for $S_{4g+2}(V_{x,y}, V_{x',y'})$ the value
\begin{eqnarray*}
\label{Eq:LkV}
& & (2g{+}1)\big( (x{-}1)(y'{+}4g{-}x') + (y{-}1)(4g{-}y'{+}x') + (x'{-}1)(y{+}4g{-}x) + (y'{-}1)(4g{-}y{+}x) 
\\
& & \qquad + 4g(x{-}1)(y'{-}1) + 4g(x'{-}1)(y{-}1) 
+ \vert y{-}y'\vert + \vert y{-}4g{+}x'\vert + \vert x {-} x'\vert + \vert y'{-}4g{+}x\vert \big)
\\
&-& (2g+1)\big( 4\min(\lfloor \frac {x-1} 2\rfloor, \lfloor \frac {x'-1} 2\rfloor) + 4\min(\lfloor \frac {y-1} 2\rfloor, \lfloor \frac {y'-1} 2\rfloor) +\dots \big)
\\
&-&2g(2g+1)(x+y)(x'+y')
\\
&+&\frac{2g{-}1}{2g{-}2}\big( -2g(x{-}1)+(y{-}2g) +(2g{-}x)+2g(y{-}1)\big)
\\
& &
\qquad\qquad\times \big( -2g(x'{-}1)+(y'{-}2g) +(2g{-}x')+2g(y'{-}1)\big).
\end{eqnarray*}

The second term---which corresponds to the correction that we added---contains an extra term dealing with the parity of $x$ and $y$. Since it is negative, forgetting it can only increase the result.

The observation here is that, except for what concerns the integer part operation in the second term, the above formula is linear in the variables $x, x', y, y'$ on the four regions $\{x<x', y<y'\}$, $\{x<x', y>y'\}$, $\{x>x', y<y'\}$, and $\{x>x', y>y'\}$.
By replacing $\lfloor \frac {x-1} 2\rfloor$ by $\frac {x-1} 2$, thus slightly increasing the result, we obtain a formula that is linear is all four variables. 

Therefore, in order to prove that $S_{4g+2}(V_{x,y}, V_{x',y'})$ is negative, we only have to evaluate the above formula on the extremal points of the four connected components of the domain that we are considering. These turn out to be 16 vertices of the cube $[1, 2g]^4$. Using symmetries, we can actually reduce the computation to six points, namely to $(1,1,1,1), (1,1,1,2g), (1,1,2g,2g)$, $(1,2g,1,2g)$, $(1,2g,2g,1),$ and $(2g,2g,2g,2g)$. It is then easy to check that the form is negative on these points. 
\end{proof}

Note that for all six points except $(1,1,2g,2g)$, the correction term provided by Lemma~\ref{L:LkSpiBis} is useless. However, at~ $(1,1,2g,2g)$, the uncorrected form is positive, while the corrected one is negative. This vertex corresponds to the linking number of two collections that go as right as possible, that is whose dynamical code is $LR^{2g}$. It is not a surprise that this vertex is where the form is the least negative, as the linking number of two (non-geodesic) collections whose dynamical code contains only~$R$ is positive (such collections are isotopic to a multiple of a fiber in~$\un\Orbg$, and two such fibers are positively linked).

We can now conclude.

\begin{proof}[Proof of Proposition~\ref{P:237} (case $(b)$ of Theorem~A)]
Since~$\un\Sigma_g$ is a finite cover of~$\un\Orbg$, it is enough by Lemma~\ref{L:Covering} to show that the invariant lifts of the families~$\gamma$ and~$\gamma'$ have negative linking number in~$\un\Sigma_g$. By the construction of the template~$\temp_{4g+2}$ and by Theorem~\ref{P:Template}, these lifts are isotopic to two families~$\hat\gamma, \hat\gamma'$ of periodic orbits of~$\temp_{4g+2}$. By Lemma~\ref{L:Cone}, the reduced linear codes of~$\hat\gamma, \hat\gamma'$ belong to the cone~$C_{4g+2}\setminus\{0\}$, and, by Lemma~\ref{L:FormNegative}, the form~$S_{4g+2}$ is negative on the pair formed by the two codes. By Lemma~\ref{L:LkSpiBis}, the linking number between~$\lk(\hat\gamma, \hat\gamma')$ is then negative, and so is~$\lk(\gamma, \gamma')$
\end{proof}

Thus the proof of Theorem~A is complete.

\section{Further questions}
\label{S:Question}

We conclude with a few remarks and questions about extensions of the above results. Here we shall both construct counter-examples showing some limitations for possible generalizations and discuss a few plausible conjectures.

\subsection{Left-handed flows}

We exhibited in Theorem~A some hyperbolic orbifolds with no rational homology on which any two orbits of the geodesic flow have a negative linking number. It is natural to ask for further examples of orbifolds with the same property. One could even wonder whether the property could be true for every hyperbolic orbifold. This is \emph{not} the case, and there exist counter-examples on every hyperbolic surface.

\begin{prop}
\label{P:Cex}
Let $\Sigma_2$ be a genus two hyperbolic surface. Then there exist two null-homologous collections~$\gamma, \gamma'$ of periodic orbits of~$\GF{\Sigma_2}$ satisfying~$\lk(\gamma, \gamma') >0$.
\end{prop}

\vspace{-1cm}
\begin{proof}
\vspace{-.55cm}
\parbox[t]{.6\textwidth}{Let $\gamma$ be the lift of the green collection, and $\gamma'$ be the lift of the orange collection in the picture on the right.}\hfill
\parbox{.4\textwidth}{\vspace{1cm}\hspace{.5cm}\includegraphics[width=.3\textwidth]{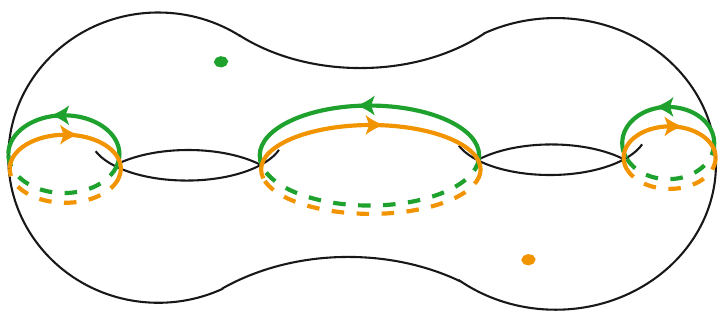}}

\vspace{-.2cm}
\parbox{.3\textwidth}{\includegraphics[width=.25\textwidth]{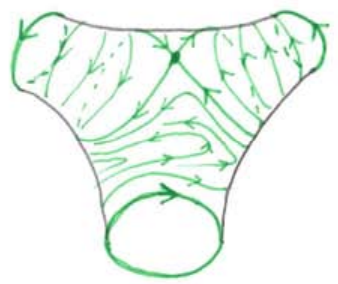}}
\parbox{.655\textwidth}{Then the lift of the green vector field is a surface whose boundary is the union of $\gamma$ and twice the fiber of a point, and which does not intersect~$\gamma'$. The same vector field on the other pair of pants connects~$\gamma'$ to twice another fiber. Then one checks that the linking number between two fibers is~$+\frac 1 2$, and we thus obtain~$\lk(\gamma, \gamma') = +2$.}
\vspace{-.5cm}
\end{proof}

However, let us mention that such counter-examples are rare. Indeed, using the techniques of Section~\ref{S:237} and a computer, we have explored the possible linking numbers of periodic orbits of~$\GF{\Sigma_2}$ and~$\GF{\Sigma_3}$. In a vast majority of cases, the linking number is negative, and the situation of Proposition~\ref{P:Cex} is exceptional. So far we have no explanation for this rarety.

\begin{ques}
Let $\Sigma_g$ be a genus~$g$ hyperbolic surface.  Characterize those pairs of collections of periodic orbits of~$\GF{\Sigma_g}$ that have a positive linking number. 
\end{ques}

We note that the counter-examples of Proposition~\ref{P:Cex} involve parallel collections of geodesics. A more specific, and maybe more accessible question, could be

\begin{ques} 
Let $\Sigma_g$ be a genus~$g$ hyperbolic surface.  If~$\gamma, \gamma'$ are two collections of periodic orbits of~$\GF{\Sigma_g}$ whose projections are not parallel and intersect, do we have~$\lk(\gamma, \gamma') \le 0$?
\end{ques}

In another direction, it is natural to wonder whether the assumption of a negative curvature can be dropped. Corollary~\ref{C:Sph2} shows that the geodesic flow is also left-handed on orbifolds with constant positive curvature, and, although their unit tangent bundle is not a homology sphere, orbifolds with constant zero curvature also yields flows that are left-handed in some weak sense (see Theorem~B and its corollaries). Nevertheless, one cannot hope for the geodesic flow on every sphere to be left-handed.

\begin{prop}
\label{P:Intersecting}
If a surface~$\Sigma$ admits at least two separating geodesics that do not intersect, then the geodesic flow~$\GF{\Sigma}$ is not left-handed.
\end{prop}

\noindent\parbox[b]{.65\textwidth}{\begin{proof}
The picture in the margin corresponds to the case of a sphere whose curvature has a non-constant sign. The lifts of the two drawn curves are cohomologous, in the complement of the other curve, to a fiber and the opposite of a fiber respectively. Their linking number is $+\frac 1 2$. The argument is similar in the general case.\end{proof}}
\hspace{.5cm}
\parbox[b]{.5\textwidth}{\includegraphics[width=.3\textwidth]{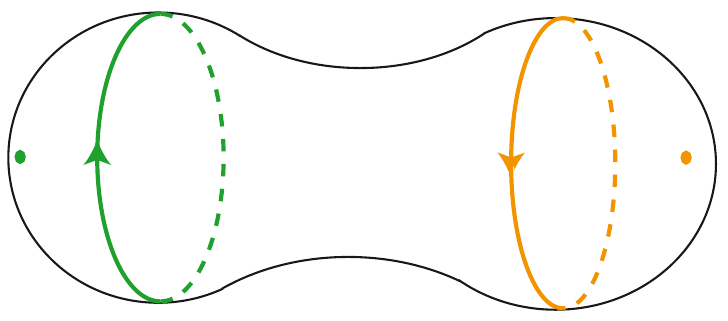}\vspace{5mm}
}

The situation of Proposition~\ref{P:Intersecting} cannot happen for a sphere with a positive curvature, and we propose

\begin{conj}
\label{C:Sphere}
Assume that $\Sigma$ is a 2-sphere with a (not necessarily constant) positive curvature. Then the geodesic flow~$\GF{\Sigma}$ is left-handed.
\end{conj}

The particular case of an ellipsoid could be accessible as, in this case, the geodesic flow is integrable.

On the other hand, the counter-example of Proposition~\ref{P:Cex} heavily relies on the fact that the homology of~$\Sigma_2$ is non-trivial. Therefore the conjecture of Ghys claiming that, if~$\Sigma$ is a hyperbolic 2-orbifold with~$H_1(\Sigma, \Q)=0$, then the geodesic flow~$\GF{\Sigma}$ is left-handed remains open and plausible.

So, in view of the known results and the above conjectures, the only cases for which the situation is totally unclear are those of orbifolds whose curvature has a non-constant sign and in which any two geodesics intersect, typically a pair of pants capped with three round hemispheres and slightly distorted so that the circles bounding the pants are not geodesic.

\subsection{Template knots}

The construction of Section~\ref{S:Template} associates a (multi)-template with every regular tessellation of~$\Hy^2$. Among all templates arising in this way, it is natural to pay special attention to those associated with the orbifolds~$\Orbpq$ of Section~\ref{S:q}. In this case, the tiles exclusively are ideal polygons and, therefore, there exists a one-to-one correspondence between the periodic orbits of the template and the periodic geodesics on the orbifold. The knots appearing in this approach generalize Lorenz knots, which correspond to the special case $p = 2, q= 3$. Lorenz knots have many interesting properties, and one can wonder whether similar properties could be true for those knots that appear in the above more general setting. 

\begin{ques}
Which knots appear as periodic orbits of~$\GF{\Orbpq}$?
\end{ques}

In this direction, Pinsky announced~\cite{Tali} that every periodic orbit of~$\GF{\Orbdq}$ is a prime knot (in the non-compact manifold~$\un\Orbpq$. Also, our current results show that these knots are fibered in~$\barre{\un\Orbpq}$. So, in particular, all knots cannot appear in this way.

\subsection{Gauss linking forms}
Let~$M$ be 3-manifold. A Gauss linking form on~$M$ is a differential $(1,1)$-form whose integral along every pair of null-homologous curves equals their linking number. Gauss linking forms exist on arbitrary 3-manifolds, but explicit formulas are known in very few cases: essentially, the only known examples are the those of~\cite{DeTurck} for the cases of~$\Sph^3, \R^3$, and~$\Hy^3$.

Now, Ghys' theorem~\cite{GhysJapan} states that a flow is left-handed if and only if their exists a Gauss linking form that is negative on the flow. Therefore, Theorem~A implies the existence, for the considered orbifolds~$\Sigma$, of a Gauss linking form in~$\un\Sigma$ that is negative along~$\GF{\Sigma}$. However, our proof of Theorem~A gives no indication about the involved Gauss linking forms. 

\begin{ques}
Are there explicit formulas for the Gauss linking forms implicitly involved in Theorem~A?
\end{ques}

More generally, better understanding Gauss linking forms appears as a plausible way to address Question~\ref{Q:OneHanded} and Conjecture~\ref{C:Sphere}.


\end{document}